\documentclass[11pt]{amsart}
\usepackage{geometry}
\usepackage[dvips]{graphicx}
\usepackage{url}

\usepackage{amscd,amsfonts,amssymb,amsmath,amsthm,latexsym}

\usepackage{graphics}
\usepackage[all]{xy}
\xyoption{curve}
\xyoption{import}
\xyoption{arc}
\xyoption{ps}

\numberwithin{equation}{section}

\theoremstyle{plain}
    \newtheorem{thm}{Theorem}[section]
    \newtheorem{lemma}[thm]{Lemma}
    \newtheorem{coro}[thm]{Corollary}
    \newtheorem{prop}[thm]{Proposition}

\theoremstyle{definition}
    \newtheorem{defi}[thm]{Definition}
    \newtheorem{ex}[thm]{Example}
\theoremstyle{remark}
    \newtheorem{remark}[thm]{Remark}
\theoremstyle{remark}
    \newtheorem{case}{Case}
    \newtheorem{mainthmcase}{Case}

\newcommand{\suchthat}{\ | \ }
\newcommand{\marked}{\mathbb{M}}
\newcommand{\punct}{\mathbb{P}}
\newcommand{\surf}{(\Sigma,\marked)}
\newcommand{\surfnoM}{\Sigma}
\newcommand{\qtau}{Q(\tau)}
\newcommand{\stau}{S(\tau,\mathbf{x})}
\newcommand{\qstau}{(\qtau,\stau)}
\newcommand{\unredqtau}{\widehat{Q}(\tau)}
\newcommand{\unredstau}{\widehat{S}(\tau,\mathbf{x})}
\newcommand{\unredqstau}{(\unredqtau,\unredstau)}
\newcommand{\wtau}{W(\tau,\mathbf{x},i,j)}
\newcommand{\qwtau}{(\qtau,\wtau)}
\newcommand{\qsigma}{Q(\sigma)}
\newcommand{\unredqsigma}{\widehat{Q}(\sigma)}
\newcommand{\ssigma}{S(\sigma,\mathbf{x})}
\newcommand{\wsigma}{W(\sigma,\mathbf{x},i,j)}

\newcommand{\qssigma}{(\qsigma,\ssigma)}
\newcommand{\qwsigma}{(\qsigma,\wsigma)}

\newcommand{\arc}{i}

\newcommand{\tagfunction}{\mathfrak{t}}
\newcommand{\arcsinsurf}{\mathbf{A}^\circ\surf}
\newcommand{\taggedinsurf}{\mathbf{A}^{\bowtie}\surf}

\newcommand{\maxid}{\mathfrak{m}}
\newcommand{\idealM}{\maxid}
\newcommand{\completeRQ}{\RA {Q}}
\newcommand{\depth}{\operatorname{depth}}
\newcommand{\field}{K}

\newcommand{\short}{\operatorname{short}}
\newcommand{\length}{\ell}

\newcommand{\RA}[1]{R\langle\hspace{-0.05cm}\langle #1\rangle\hspace{-0.05cm}\rangle}

\makeatletter
\def\blfootnote{\xdef\@thefnmark{}\@footnotetext}
\makeatother

\begin{document}

\title[QPs associated to triangulated surfaces IV]{Quivers with potentials associated to triangulated surfaces, Part IV: Removing boundary assumptions}
\author{Daniel Labardini-Fragoso}
\address{Instituto de Matem\'aticas, Universidad Nacional Aut\'onoma de M\'exico}
\email{labardini@matem.unam.mx}
\dedicatory{Dedicated to the memory of Professor Andrei Zelevinsky}
\maketitle

\begin{abstract} We prove that the quivers with potentials associated to triangulations of surfaces with marked points, and possibly empty boundary, are non-degenerate, provided the underlying surface with marked points is not a closed sphere with exactly 5 punctures. This is done by explicitly defining the QPs that correspond to tagged triangulations and proving that whenever two tagged triangulations are related by a flip, their associated QPs are related by the corresponding QP-mutation. As a byproduct, for (arbitrarily punctured) surfaces with non-empty boundary we obtain a proof of the non-degeneracy of the associated QPs which is independent from the one given by the author in the first paper of the series.

The main tool used to prove the aforementioned compatibility between flips and QP-mutations is what we have called \emph{Popping Theorem}, which, roughly speaking, says that an apparent lack of symmetry in the potentials arising from ideal triangulations with self-folded triangles can be fixed by a suitable right-equivalence.
\end{abstract}

{\small\tableofcontents}\blfootnote{This is the final version. The version published by \textsl{Selecta Mathematica (New Series)} contains a few minor redaction errors that were introduced during the final edition process. E.g., ``related by a flip'' and ``related by a QP-mutation'' were incorrectly replaced by ``related to a flip'' and ``related to a QP-mutation'' in the published version.}

\section{Introduction}

The quivers with potentials associated in \cite{Labardini1} to ideal triangulations of arbitrarily punctured surfaces, and independently in \cite{ABCP} in the particular case of unpunctured surfaces, have had appearances both in Mathematics (cf. for example \cite{Bridgeland-Smith}, \cite{Nagao1}, \cite{Nagao2}) and Physics (cf. for example \cite{ACCERV1}, \cite{ACCERV2}, \cite{Cecotti}, \cite{Xie}), with the compatibility between ideal flips and QP-mutations proved in \cite{Labardini1} appearing as well.

The definition of the aforementioned QPs is quite simple-minded: every ``\emph{sufficiently nice}" ideal triangulation possesses two obvious types of cycles on its signed adjacency quiver, and the associated potential just adds them up. Simplicity has come with a price though, namely, a considerable amount of algebraic-combinatorial computations in order to prove statements that are otherwise quite natural (cf. for example \cite{Labardini1}). The price is paid once more in the present note, where to each tagged triangulation $\tau$ of a surface with marked points $\surf$ we associate a QP $\qstau$, and prove that whenever two tagged triangulations $\tau$ and $\sigma$ are related by the flip of a tagged arc $i$, the QPs $\mu_i\qstau$ and $\qssigma$ are right-equivalent, provided $\surf$ is not a closed sphere with exactly 5 punctures. As a direct consequence of this \emph{flip $\leftrightarrow$ QP-mutation compatibility}, we establish the non-degeneracy of all the QPs associated to (tagged) triangulations (except in the case just mentioned).

Let us describe the contents of the paper in more detail. In Section \ref{sec:background} we give some background concerning homomorphisms between complete path algebras (Subsection \ref{subsec:automorphisms}) and signatures of tagged triangulations (Subsection \ref{subsec:signatures}). In Section \ref{sec:definition-of-Stau} we define a potential $\stau$ for each tagged triangulation $\tau$ of a punctured surface with respect to a choice $\mathbf{x}=(x_p)_{p\in\punct}$ of non-zero scalars for the punctures of the surface, and then recall \cite[Theorem 30]{Labardini1}, stated as Theorem \ref{thm:ideal-flips<->mutations} below, which says that flips between ideal triangulations are always compatible with QP-mutations.

In the somewhat intuitive Section \ref{sec:folded-sides-the-problem}, we use an explicit example to illustrate a general problem that arises when trying to prove that flips of folded sides of ideal triangulations are compatible with QP-mutations. In Section \ref{sec:def-of-popped-potentials} we go back to formal considerations. There, for each ideal triangulation $\tau$ having a self-folded triangle we define a \emph{popped potential} $\wtau$ as the result of applying an obvious quiver automorphism, induced by the self-folded triangle, to the potential associated to $\tau$ with respect to another choice $\mathbf{y}=(y_p)_{p\in\punct}=((-1)^{\delta_{p,q}}x_p)$ of non-zero scalars, where $q$ is the punctured enclosed in the self-folded triangle (thus in this paper the signs of the scalars attached to the punctures will play some role, at least in the absence of boundary). Based on Theorem \ref{thm:ideal-flips<->mutations} we then show Lemma \ref{lemma:pop-commutes-with-mutation}, which says that as long as none of the two arcs contained in a fixed self-folded triangle is ever flipped, the popped potentials $\wtau$ associated to the ideal triangulations containing the fixed self-folded triangle have the same flip/mutation dynamics possessed by the potentials $\stau$. That is, if two ideal triangulations share a self-folded ideal triangle and are related by a flip, then their popped QPs are related by the corresponding QP-mutation.

Section \ref{sec:pop-is-re} is the technical core of the present work. Its Subsection \ref{subsec:Pop-Thm-statement} is devoted to state what is instrumentally the main result of the paper, Theorem \ref{thm:popping-is-right-equiv}, which we call the \emph{Popping Theorem} and says that if $\surf$ is not a closed sphere with less than 6 punctures, then for any ideal triangulation $\sigma$ with a self-folded triangle, the pop in $\sigma$ of such self-folded triangle induces right-equivalence, that is, the QP $\qssigma$ and the QP $\qwsigma$ with popped potential are right-equivalent (see \eqref{eq:pop-is-re-for-tau} below). Our first step towards the proof of Theorem \ref{thm:popping-is-right-equiv} makes use of Theorem \ref{thm:ideal-flips<->mutations} and Lemma \ref{lemma:pop-commutes-with-mutation} in order to reduce the alluded proof to showing the mere existence of a single ideal triangulation $\tau$ with a self-folded triangle whose pop in $\tau$ induces right-equivalence. The existence of such a $\tau$ is proved in Subsection \ref{subsec:proof-Pop-Thm-empty-boundary} for punctured surfaces with empty boundary that are different from a sphere with less than 7 punctures\footnote{For the sphere with exactly 6 punctures, the proof of existence of the alluded $\tau$ is omitted here, the reader can find such proof in the fourth arXiv version of this paper.}.

Having proved the Popping Theorem for surfaces with empty boundary, in Subsection \ref{subsec:proof-Pop-Thm-nonempty-boundary} we use restriction of QPs and gluing of disks along boundary components to deduce that it holds as well in the presence of boundary (cf. Proposition \ref{prop:pop-is-re-non-empty-boundary}). The key properties of restriction used are the facts that it preserves right-equivalences and takes potentials of the form $\stau$ to potentials of the same form and popped potentials to popped potentials.

Despite its somewhat technical proof, the Popping Theorem easily yields Theorem \ref{thm:leaving-positive-stratum}, the second main result of the paper, stated and proved in Section \ref{sec:leaving-pos-stratum}, which says that flips of folded sides of self-folded triangles are compatible with QP-mutations. Theorems \ref{thm:ideal-flips<->mutations} and \ref{thm:leaving-positive-stratum} are used in Section \ref{sec:flip<->mutations} to deduce our third main result, Theorem \ref{thm:tagged-flips<->mutations}, which states that arbitrary flips of tagged arcs are always compatible with QP-mutation, that is, that any two tagged triangulations related by a flip give rise to quivers with potentials related by the corresponding QP-mutation. Theorem \ref{thm:tagged-flips<->mutations} has the non-degeneracy of the QPs $\qstau$ as an immediate consequence. This is the fourth main result of the present note and is stated in Section \ref{sec:nondegeneracy} as Corollary \ref{coro:non-degenerate}.

In Section \ref{sec:irrelevant} we show that for surfaces with non-empty boundary, the scalars $x_p$ attached to the punctures $p\in\punct$, and more importantly, the signs in the potentials $\stau$ arising from the weak signatures $\epsilon_\tau$, are irrelevant; that is, for any two choices $\mathbf{x}=(x_p)_{p\in\punct}$ and $\mathbf{y}=(y_p)_{p\in\punct}$ of non-zero scalars, the QPs $\qstau$ and $(\qtau,S(\tau,\mathbf{y}))$ are right-equivalent. This implies that the QPs defined in \cite{CI-LF} for surfaces with non-empty boundary are right-equivalent to the ones defined here.

Finally, in Section \ref{sec:Jacobi-finiteness} we state recent results, proved independently by Ladkani (cf. \cite{Ladkani}) and Trepode--Valdivieso-D\'{i}az (cf. \cite{TV}), that answer a question of the author on the Jacobi-finiteness of the QPs associated to ideal triangulations of surfaces with empty boundary.

We remind the reader that starting in \cite{Labardini1} we have decided not to work with the situation where $\surf$ is an unpunctured or once-punctured monogon or digon, an unpunctured triangle, or a closed sphere with less than five punctures (though the 4-punctured sphere has proven important both in cluster algebra theory and representation theory, cf. \cite{BG}). In the remaining situations, up to now,
potentials have been defined for all ideal triangulations of arbitrarily-punctured surfaces, regardless of emptiness or non-emptiness of the boundary (cf. \cite{Labardini1}; and in the unpunctured non-empty boundary case, \cite{ABCP}), and for all tagged triangulations of arbitrarily-punctured surfaces with non-empty boundary (cf. \cite{CI-LF}), but not for non-ideal tagged triangulations of (necessarily punctured) surfaces with empty boundary. Also, non-degeneracy has been shown for all arbitrarily-punctured surfaces with non-empty boundary and for all empty-boundary positive-genus surfaces with exactly one puncture (cf. \cite{Labardini1}), but not for empty-boundary surfaces with more than one puncture. Furthermore, the compatibility between flips and QP-mutations has not been shown in general for arbitrary flips of tagged triangulations, but only for flips between ideal triangulations (with no extra assumption on the boundary, cf. \cite{Labardini1}), and for flips occurring inside certain subsets $\bar{\Omega}'_\epsilon$ of Fomin-Shapiro-Thurston's \emph{closed strata} $\bar{\Omega}_\epsilon$ (with the assumption of non-empty boundary, cf. \cite{CI-LF}).

Hence, the results presented here strongly improve the results obtained so far regarding both the non-degeneracy question, and the compatibility between flips of tagged triangulations and mutations of quivers with potentials. Indeed, here we give the definition of a potential for any tagged triangulation of any surface (definition completely missing in \cite{Labardini1}, and missing in \cite{CI-LF} for surfaces with empty boundary), and prove the desired compatibility between QP-mutations and flips of tagged triangulations for all surfaces but the 5-punctured sphere (a strong improvement of
\cite[Theorem 4.4 and Corollary 4.9]{CI-LF}). As a byproduct, we obtain a direct proof of the non-degeneracy of the QPs $\qstau$ for all surfaces $\surf$ different from the 5-punctured sphere, with no assumptions on the possible emptiness of the boundary of $\Sigma$ (in contrast to \cite{Labardini1}, where, as we said in the previous paragraph, non-degeneracy was shown only for surfaces with non-empty boundary and for positive-genus empty-boundary surfaces with exactly one punctures).
Furthermore, the proof given here of the non-degeneracy of
the QPs arising from ideal triangulations of surfaces with non-empty boundary is independent from, and more direct than, the one given in
\cite{Labardini1},
precisely because we calculate the potentials corresponding to tagged triangulations and do not appeal to a proof via rigidity.

Some words on the background needed to understand the statements in this note: Since detailed background sections on Derksen-Weyman-Zelevinsky's QP-mutation theory and Fomin-Shapiro-Thurston's development of surface cluster algebras have been included in \cite{Labardini1} and \cite{CI-LF}, we have decided not to include similar sections here. The reader unfamiliar with the combinatorial and algebraic background from tagged triangulations and quivers with potentials not provided here is kindly asked to look at \cite{DWZ1}, \cite{FST}, \cite{Zelevinsky-Oberwolfach}, \cite[Section 2]{Labardini1} or \cite[Sections 2 and 3]{CI-LF}.

\section*{Acknowledgements}

I am sincerely grateful to Tom Bridgeland for his interest in this work and many very pleasant discussions. I thank him as well for informing me of Gaiotto-Moore-Neitzke's use of the term ``pop" (cf. \cite[Pages 11-12, and Sections 5.6 and 5.8]{GMN}) for the analogue in Geometry and Physics of what was called ``swap" in a first version of this paper.

Thanks are owed to the anonymous referee for a number of useful suggestions.

I started considering the problem of compatibility between flips and QP-mutations when I was a PhD student at Northeastern University (Boston, MA, USA). As such, I profitted from numerous discussions with Jerzy Weyman and Andrei Zelevinsky. I am deeply grateful to both of them for their teachings.

Professor Andrei Zelevinsky unfortunately passed away a few weeks before this paper was submitted. No words can express my gratitude towards him for all the teachings, guidance and encouragement I constantly received from him during my Ph.D. studies and afterwards. My admiration for his way of doing Mathematics will always be of the deepest kind.

\section{Algebraic and combinatorial background}\label{sec:background}

\subsection{Automorphisms of complete path algebras}\label{subsec:automorphisms}

Here we briefly recall some basic facts and definitions concerning homomorphisms between complete path algebras.

The following proposition tells us that in order to have an $R$-algebra homomorphism between complete path algebras, it is sufficient to send each arrow $a$ to a (possibly infinite) linear combination of paths with the same starting and ending points as $a$. It also gives us a criterion to decide whether a given $R$-algebra homomorphism is an isomorphism using basic linear algebra.

\begin{prop}[{\cite[Proposition 2.4]{DWZ1}}] Let $Q$ and $Q'$ be quivers on the same vertex set, and let their respective arrow spans be $A$ and $A'$. Every pair $(\varphi^{(1)},\varphi^{(2)})$ of $R$-$R$-bimodule homomorphisms $\varphi^{(1)}:A\rightarrow A'$, $\varphi^{(2)}:A\rightarrow\idealM(Q')^2$, extends uniquely to a continuous $R$-algebra homomorphism $\varphi:\completeRQ\rightarrow \RA{ Q'}$ such that $\varphi|_A=(\varphi^{(1)},\varphi^{(2)})$.
Furthermore, $\varphi$ is $R$-algebra isomorphism if and only if $\varphi^{(1)}$ is an $R$-$R$-bimodule isomorphism.
\end{prop}

\begin{defi}[{\cite[Definition 2.5]{DWZ1}}]\label{def:types-of-automorphisms} An $R$-algebra automorphism $\varphi:\completeRQ\rightarrow \RA{ Q}$ is said to be \begin{itemize}
\item \emph{unitriangular} if $\varphi^{(1)}$ is the identity of $A$;
\item \emph{of depth $\ell<\infty$} if $\varphi^{(2)}(A)\subseteq\maxid^{\ell+1}$, but $\varphi^{(2)}(A)\nsubseteq\maxid^{\ell+2}$;
\item \emph{of infinite depth} if $\varphi^{(2)}(A)=0$.
\end{itemize}
The depth of $\varphi$ will be denoted $\depth(\varphi)$.
\end{defi}

Thus for example, the identity is the only unitriangular automorphism of $\RA{ Q}$ that has infinite depth. Also, composition of unitriangular automorphisms is unitriangular.

For the following lemma, we use the convention that $\maxid^\infty=0$.

\begin{lemma}[{\cite[Equation (2.4)]{DWZ1}}]\label{lemma:depth-is-short-cycle-friendly} If $\varphi$ is a unitriangular automorphism of $\RA{ Q}$, then for every $n\geq 0$ and every $u\in\maxid^n$ we have $\varphi(u)-u\in\maxid^{n+\depth(\varphi)}$.
\end{lemma}

\begin{lemma}\label{lemma:well-defined-limit-automorphism} Let $Q$ be any
quiver, and $(\psi_n)_{n>0}$ be a sequence of unitriangular $R$-algebra
automorphisms of $\RA{ Q}$. Suppose that
$\lim_{n\to\infty}\depth(\psi_n)=\infty$. Then the limit
\[
\psi=\lim_{n\to\infty}\psi_n\psi_{n-1}\ldots \psi_2\psi_1
\]
is a well-defined  unitriangular $R$-algebra automorphism of $\RA{ Q}$.
If, moreover, $S$ and $(S_n)_{n>0}$ are respectively a potential and a
sequence of potentials on $Q$, satisfying $\lim_{n\to\infty}S_n=0$ and such that
$\psi_n$ is a right-equivalence $(Q,S+S_n)\to(Q,S+S_{n+1})$ for all $n>0$, then
$\psi=\lim_{n\to\infty}\psi_n\psi_{n-1}\ldots \psi_2\psi_1$
is a right-equivalence $(Q,S+S_1)\to(Q,S)$.
\end{lemma}

\begin{proof} We can suppose, without loss of generality, that
$\depth(\psi_n)<\infty$ for all $n>0$ (this is because the only unitriangular
automorphism of $\RA{ Q}$ that has infinite depth is
the identity). Since $\lim_{n\to\infty}\depth(\psi_n)=\infty$, there exists
$m_2>0$ such that $\depth(\psi_k)>\depth(\psi_1)$ for all $k\geq m_2$. And
once we have a positive integer $m_n$, we can find $m_{n+1}>m_n$ such that
$\depth(\psi_k)>\depth(\psi_{m_n})$ for all $k\geq m_n$. Set
$\varphi_n=\psi_{m_{n}}\psi_{m_{n}-1}\ldots\psi_{m_{n-1}+1}$, with the convention
that $m_1=1$ and $m_0=0$. Then $(\depth(\varphi_n))_{n>0}$ is an increasing
sequence of positive integers.

The limit
$\lim_{n\to\infty}\varphi_n\varphi_{n-1}\ldots \varphi_2\varphi_1$ is well-defined if and only if so is the
limit $\lim_{n\to\infty}\psi_{n}\psi_{n-1}\ldots \psi_2\psi_1$, and if both of these limits are well-defined, then they are equal. In order to show
that $\varphi=\lim_{n\to\infty}\varphi_n\varphi_{n-1}\ldots \varphi_2\varphi_1$ is
well-defined it suffices to show that for every $u\in\maxid$ and every $d>0$
the sequence $(u_n^{(d)})_{n>0}$ formed by the degree-$d$ components of the
elements $u_n=\varphi_n\ldots\varphi_1(u)$ eventually stabilizes as
$n\to\infty$. Notice that $\depth(\varphi_n)\geq n-1$ for all $n>0$. From this
and Lemma \ref{lemma:depth-is-short-cycle-friendly} we deduce that, for a
given $u\in\maxid$, there exists a sequence $(v_n)_{n>0}$ such that
$v_n\in\maxid^n$ and $u_n=u+\sum_{j=1}^nv_j$ for all $n>0$. From this we see that
$u_n^{(d)}=u_{n+1}^{(d)}$ for $n>d$, so the sequence $(u_n^{(d)})_{n>0}$ stabilizes.

For the second statement of the lemma, let $V$ be the topological closure of
the $\field$-vector subspace of $\RA{ Q}$ generated by all
elements of the form $\alpha_1\alpha_2\ldots\alpha_d-\alpha_2\ldots\alpha_d\alpha_1$
with $\alpha_1\ldots \alpha_d$ a
cycle on $Q$. Then $V$ is a $\field$-vector subspace of
$\RA{ Q}$ as well and two potentials are
cyclically-equivalent if and only if their difference belongs to $V$.

Compositions of right-equivalences is again a right-equivalence, hence
$\varphi_n$ is a right-equivalence $(Q,S+S_{m_{n}})\to(Q,S+S_{m_{n+1}})$ for all
$n>0$. We deduce that
\[
\varphi_n\varphi_{n-1}\ldots\varphi_2\varphi_1(S+S_1)-(S+S_{m_{n+1}})\in V.
\]
The fact that both sequences
$(\varphi_n\varphi_{n-1}\ldots\varphi_2\varphi_1(S+S_1))_{n>0}$ and
$(-(S+S_{m_{n+1}}))_{n>0}$ are convergent in $\RA{ Q}$
implies that
\[
(\varphi_n\varphi_{n-1}\ldots\varphi_2\varphi_1(S+S_1)-(S+S_{m_{n+1}}))_{n>0}
\]
converges as well. Since $V$ is closed, we have
$$
\varphi(W+S_1)-W=
\lim_{n\to\infty}(\varphi_n\varphi_{n-1}\ldots
\varphi_2\varphi_1(S+S_1)-(S+S_{m_{n+1}}))\in V.
$$
This finishes the proof of Lemma~\ref{lemma:well-defined-limit-automorphism}.
\end{proof}

\subsection{Signatures and weak signatures of tagged triangulations}\label{subsec:signatures}

In this subsection we review some elementary facts concerning tagged triangulations, their signatures, and the behavior of the latter ones under flips.

\begin{defi}\label{def:surf-with-marked-points} A \emph{bordered surface with marked points}, or simply a \emph{surface}, is a pair $\surf$, where $\surfnoM$ is a compact connected oriented Riemann surface with (possibly empty) boundary, and $\marked$ is a finite set of points on $\surfnoM$, called \emph{marked points}, such that $\marked$ is non-empty and has at least one point from each connected component of the boundary of $\surfnoM$. The marked points that lie in the interior of $\surfnoM$ are called \emph{punctures}, and the set of punctures of $\surf$ is denoted $\punct$. Throughout the paper we will always assume that $\surf$ is none of the following:
\begin{itemize}
\item a sphere with less than four punctures;
\item an unpunctured monogon, digon or triangle;
\item a once-punctured monogon or digon.
\end{itemize}
Here, by a monogon (resp. digon, triangle) we mean a disk with exactly one (resp. two, three) marked point(s) on the boundary.
\end{defi}

For the definitions of the following concepts we kindly ask the reader to look at the corresponding reference:
\begin{itemize}
\item ordinary arc \cite[Definition 2.2]{FST};
\item compatibility of pairs of ordinary arcs \cite[Definition 2.4]{FST};
\item ideal triangulation \cite[Definition 2.6]{FST};
\item tagged arc \cite[Defintion 7.1 and Remark 7.3]{FST};
\item compatibility of pairs of tagged arcs \cite[Definition 7.4 and Remark 7.5]{FST};
\item tagged triangulation \cite[Page 111]{FST}.
\end{itemize}

The set of ordinary arcs, taken up to isotopy relative to $\marked$, is denoted by $\arcsinsurf$, while the set of tagged arcs is denoted by $\taggedinsurf$.

\begin{defi}\label{def:types-of-ideal-triangles} Let $\tau$ be an ideal triangulation of a surface $\surf$.
\begin{enumerate}\item For each connected component of the complement in $\surfnoM$ of the union of the arcs in $\tau$, its topological closure $\triangle$ will be called an \emph{ideal triangle} of $\tau$.
\item An ideal triangle $\triangle$ is called \emph{interior} if its intersection with the boundary of $\surfnoM$ consists only of (possibly none) marked points. Otherwise it will be called \emph{non-interior}.
\item An interior ideal triangle $\triangle$ is \emph{self-folded} if it contains exactly two arcs of $\tau$ (see Figure \ref{Fig:selffoldedtriang}).
\end{enumerate}
\end{defi}

        \begin{figure}[!h]
                \caption{Self-folded triangle}\label{Fig:selffoldedtriang}
                \centering
                \includegraphics[scale=.4]{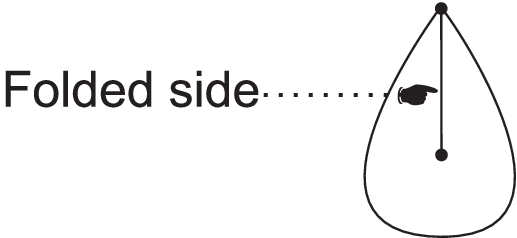}
        \end{figure}

All tagged triangulations of $\surf$
have the same cardinality and
every tagged arc $i$ in a tagged triangulation $\tau$ can be replaced by a uniquely defined, different tagged arc that together with the remaining tagged arcs from $\tau$ forms a tagged triangulation $\sigma$. This combinatorial replacement will be called \emph{flip}, written $\sigma=f_i(\tau)$. Furthermore, a sequence $(\tau_0,\ldots,\tau_\ell)$ of tagged triangulations will be called a \emph{flip-sequence} if $\tau_{k-1}$ and $\tau_k$ are related by a flip for $k=1,\ldots,\ell$. A flip-sequence will be called \emph{ideal flip-sequence} if it involves only ideal triangulations.

\begin{prop}
\label{prop:ideal-triangs-seqs-of-flips} Let $\surf$ be a surface.
\begin{itemize}\item Any two ideal triangulations $\tau$ and $\sigma$ of $\surf$ are members of an ideal flip-sequence. A flip sequence $(\tau,\tau_1,\ldots,\tau_{\ell-1},\sigma)$ can always be chosen in such a way that $\tau\cap\sigma\subseteq\tau_k$ for all $k=1,\ldots,\ell-1$.
\item There is at least one ideal triangulation of $\surf$ that does not have self-folded triangles.
\item Any two ideal triangulations without self-folded triangles are members of an ideal flip-sequence that involves only ideal triangulations without self-folded triangles.
\item If $\surf$ is not a once-punctured surface with empty boundary, then any two tagged triangulations are members of a flip-sequence.
\end{itemize}
\end{prop}

The first assertion of Proposition \ref{prop:ideal-triangs-seqs-of-flips} is well known and has many different proofs, we refer the reader to \cite[Pages 36-41]{Mosher} for an elementary one. The second and fourth assertions of the proposition are proved in \cite{FST}. A proof of the third assertion can be found in \cite[Proof of Corollary 6.7]{Labardini2}.

Let us recall how to represent ideal triangulations with tagged ones and viceversa.

\begin{defi}[{\cite[Definitions 7.2 and 9.2]{FST}}]\label{def:tagfunction(ideal-arc)} Let $\epsilon:\punct\rightarrow\{-1,1\}$ be any function. We define a function $\tagfunction_\epsilon:\arcsinsurf\rightarrow\taggedinsurf$ that represents ordinary arcs by tagged ones as follows.
\begin{enumerate}\item If $\arc$ is an ordinary arc that is not a loop enclosing a once-punctured monogon, set $\arc$ to be the underlying ordinary arc of the tagged arc $\tagfunction_\epsilon(\arc)$. An end of $\tagfunction_\epsilon(\arc)$ will be tagged notched if and only if the corresponding marked point is an element of $\punct$ where $\epsilon$ takes the value $-1$.
\item If $j$ is a loop, based at a marked point $q$, that encloses a once-punctured monogon, being $p$ the puncture inside this monogon, then the underlying ordinary arc of $\tagfunction(j)$ is the arc that connects $q$ with $p$ inside the monogon. The end at $q$ will be tagged notched if and only if $q\in\punct$ and $\epsilon(q)=-1$, and the end at $p$ will be tagged notched if and only if $\epsilon(p)=1$.
\end{enumerate}
\end{defi}

\begin{ex} In Figure \ref{Fig:tagfunction_examples}
        \begin{figure}[!h]
                \caption{}\label{Fig:tagfunction_examples}
                \centering
                \includegraphics[scale=.35]{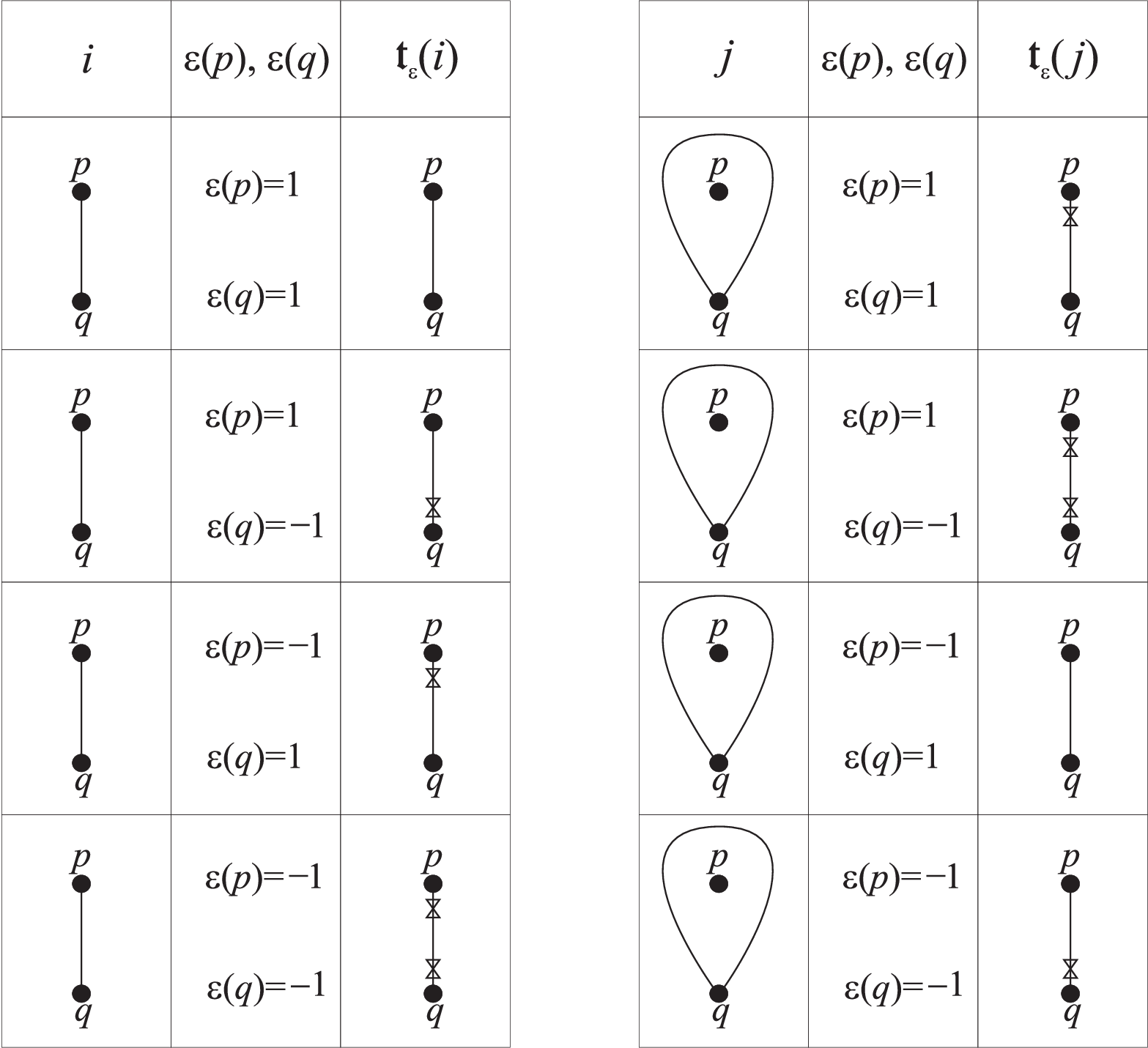}
        \end{figure}
we can see all possibilities for $\tagfunction_{\epsilon}(i)$ (resp. $\tagfunction_{\epsilon}(j)$) given an ordinary arc $i$ that does not cut out a once-punctured monogon (resp. a loop $j$ that cuts out a once-punctured monogon), and the values of $\epsilon$ at the endpoints of $i$ (resp. at the base-point of $j$ and the puncture enclosed by $j$). $\hfill{\blacktriangle}$
\end{ex}

To pass from tagged triangulations to ideal ones we need the notion of signature.

\begin{defi}\label{def:signature} Let $\tau$ be a tagged triangulation of $\surf$.
\begin{enumerate}
\item Following \cite[Definition 9.1]{FST}, we define the \emph{signature} of $\tau$ to be the function
$\delta_\tau:\punct\rightarrow\{-1,0,1\}$ given by
\begin{equation}
\delta_\tau(p)=
\begin{cases} 1 & \text{if all ends of tagged arcs in $\tau$ incident to $p$ are tagged plain;}\\
-1 & \text{if all ends of tagged arcs in $\tau$ incident to $p$ are tagged notched;}\\
0 & \text{otherwise.}
\end{cases}
\end{equation}
Note that if $\delta_\tau(p)=0$, then there are precisely two arcs in $\tau$ incident to $p$, the untagged versions of these arcs coincide and they carry the same tag at the end different from $p$.
\item The \emph{weak signature} of $\tau$ is the function $\epsilon_\tau:\punct\rightarrow\{-1,1\}$ defined by
\begin{equation}
\epsilon_\tau(p)=
\begin{cases} 1 & \text{if $\delta_\tau(p)\in\{0,1\}$;}\\
-1 & \text{otherwise.}
\end{cases}
\end{equation}
\end{enumerate}
\end{defi}

\begin{defi}[{\cite[Definition 9.2]{FST}}]\label{def:tau-circ} Let $\tau$ be a tagged triangulation of $\surf$. We replace each tagged arc in $\tau$ with an ordinary arc by means of the following rules:
\begin{enumerate}
\item delete all tags at the punctures $p$ with non-zero signature;
\item for each puncture $p$ with $\delta_\tau(p)=0$, replace the tagged arc $\arc\in\tau$ which is notched at $p$ by a loop enclosing $p$ and $\arc$.
\end{enumerate}
The resulting collection of ordinary arcs will be denoted by $\tau^\circ$. For $i\in\tau$ we will denote by $i^\circ$ the arc in $\tau^\circ$ that replaces $i$, although we stress the fact that $i^\circ$ really depends on $\tau$ and not on $i$ alone.
\end{defi}

The next proposition follows from the results in \cite[Subsection 9.1]{FST}.

\begin{prop}\label{prop:tagfunction-and-circ} Let $\surf$ be a surface.
\begin{enumerate}
\item For every function $\epsilon:\punct\rightarrow\{-1,1\}$, the function $\tagfunction_\epsilon:\arcsinsurf\to\taggedinsurf$ is injective and preserves compatibility. Thus, if $\arc_1$ and $\arc_2$ are compatible ordinary arcs, then $\tagfunction_\epsilon(\arc_1)$ and $\tagfunction_\epsilon(\arc_2)$ are compatible tagged arcs. Consequently, if $T$ is an ideal triangulation of $\surf$, then $\tagfunction_\epsilon(T)=\{\tagfunction_\epsilon(\arc)\suchthat\arc\in T\}$ is a tagged triangulation of $\surf$. Moreover, if $T_1$ and $T_2$ are ideal triangulations such that $T_2=f_\arc(T_1)$ for an arc $\arc\in T_1$, then $\tagfunction_\epsilon(T_2)=f_{\tagfunction_\epsilon(\arc)}(\tagfunction_{\epsilon}(T_1))$.
\item If $\tau$ is a tagged triangulation of $\surf$, then $\tau^\circ$ is an ideal triangulation of $\surf$ and $\arc\mapsto\arc^\circ$ is a bijection between $\tau$ and $\tau^\circ$.
\item For every ideal triangulation $T$, we have $\tagfunction_\mathbf{1}(T)^\circ=T$, where $\mathbf{1}:\punct\to\{-1,1\}$ is the constant function taking the value 1.
\item For every tagged triangulation $\tau$ and every tagged arc $i\in\tau$ we have $\tagfunction_{\epsilon_\tau}(i^\circ)=i$, where $i^\circ\in\tau^\circ$ is the ordinary arc that replaces $i$ in Definition \ref{def:tau-circ}. Consequently, $\tagfunction_{\epsilon_\tau}(\tau^\circ)=\tau$.
\item Let $\tau$ and $\sigma$ be tagged triangulations such that $\epsilon_\tau=\epsilon_\sigma$.
    If $\sigma=f_{\arc}(\tau)$ for a tagged arc $\arc\in\tau$, then $\sigma^\circ=f_{\arc^\circ}(\tau^\circ)$, where $\arc^\circ\in\tau^\circ$ is the ordinary arc that replaces $\arc$ in Definition \ref{def:tau-circ}. Moreover, the diagram of functions
    \begin{equation}\label{eq:commutative-diagram-ordinary-flips}
    \xymatrix{\tau \ar[r]^{?^\circ} \ar[d] & \tau^\circ \ar[d] \ar[r]^{\tagfunction_{\epsilon_\tau}} & \tagfunction_{\epsilon_\tau}(\tau^\circ) \ar[d] \\
    \sigma \ar[r]_{?^\circ} & \sigma^\circ\ar[r]_{\tagfunction_{\epsilon_\sigma}} & \tagfunction_{\epsilon_\sigma}(\sigma^\circ)
    }
    \end{equation}
    commutes, where the three vertical arrows are canonically induced by the operation of flip.
\item Let $\tau$ and $\sigma$ be tagged triangulations such that $\sigma=f_i(\tau)$ for some tagged arc $i\in\tau$. Then $\epsilon_\tau$ and $\epsilon_\sigma$ either are equal or differ at exactly one puncture $q$. In the latter case, if $\epsilon_\tau(q)=1=-\epsilon_\sigma(q)$, then $\arc^\circ$ is a folded side of $\tau^\circ$ incident to the puncture $q$, and $\tagfunction_{\epsilon_\tau\epsilon_\sigma}(\sigma^\circ)=f_{\tagfunction_{\mathbf{1}}(i^\circ)}(\tagfunction_{\mathbf{1}}(\tau^\circ))$, where $\epsilon_\tau\epsilon_\sigma:\punct\to\{-1,1\}$ is the function defined by $p\mapsto\epsilon_\tau(p)\epsilon_\sigma(p)$.
    Moreover, the diagram of functions
    \begin{equation}\label{eq:commutative-diagram-tagged-flips}
    \xymatrix{\tau \ar[r]^{?^\circ} \ar[d] & \tau^\circ \ar[r]^{\tagfunction_{\mathbf{1}}} & \tagfunction_{\mathbf{1}}(\tau^\circ) \ar[d]\\
\sigma=f_i(\tau) \ar[r]_{\phantom{xxx}?^\circ} & \sigma^\circ \ar[r]_{\tagfunction_{\epsilon_\tau\epsilon_\sigma}\phantom{xxxxxxxxxx}} & \tagfunction_{\epsilon_\tau\epsilon_\sigma}(\sigma^\circ)=f_{\tagfunction_{\mathbf{1}}(i^\circ)}(\tagfunction_{\mathbf{1}}(\tau^\circ))
}
    \end{equation}
    commutes, where the two vertical arrows are canonically induced by the operation of flip.
\end{enumerate}
\end{prop}

\begin{ex} Consider the tagged triangulation $\tau$ depicted on the left side of Figure \ref{Fig:flip_same_weak_signature_example}. Given any tagged arc $i\in\tau$, the tagged triangulation $f_i(\tau)$ will satisfy $\epsilon_\tau=\epsilon_{f_i(\tau)}$ if and only if $i$ is not one of the tagged arcs in $\tau$ that have been drawn bolder than the rest.
        \begin{figure}[!h]
                \caption{}\label{Fig:flip_same_weak_signature_example}
                \centering
                \includegraphics[scale=.4]{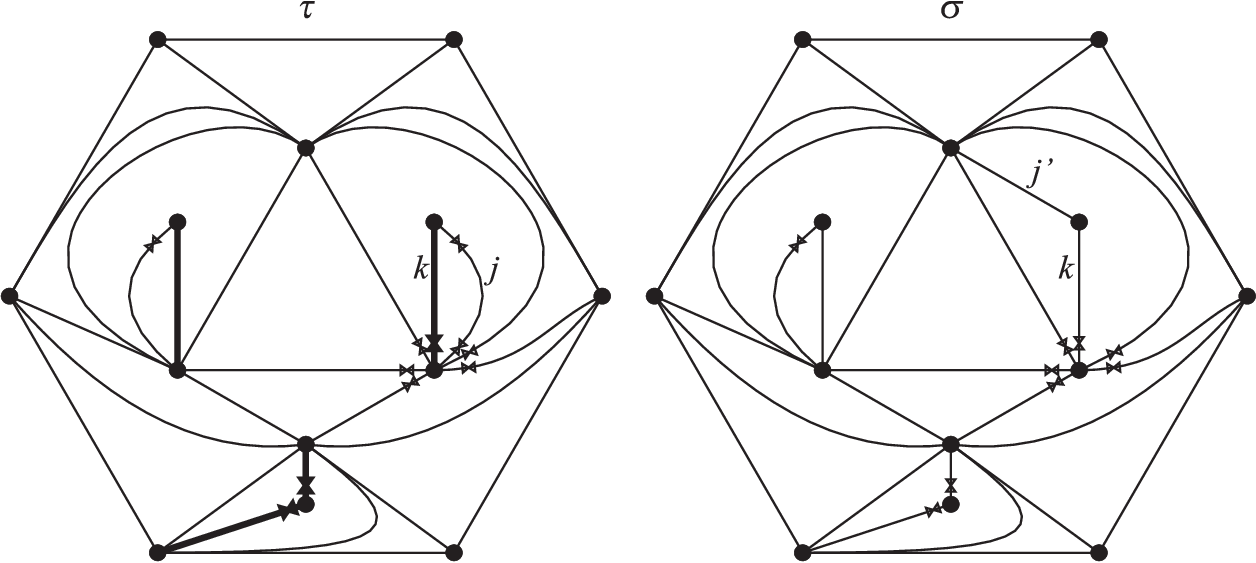}
        \end{figure}

If we flip the tagged arc $j\in\tau$ the resulting tagged triangulation $\sigma=f_j(\tau)$, depicted on the right side of Figure \ref{Fig:flip_same_weak_signature_example}, certainly satisfies $\epsilon_\tau=\epsilon_\sigma$. In Figure \ref{Fig:flip_same_weak_signature_example_circ} we have drawn the ideal triangulations $\tau^\circ$ and $\sigma^\circ$, and the reader can easily verify that, as stated in the fifth assertion of Proposition \ref{prop:tagfunction-and-circ}, we have $\sigma^\circ=f_{j^\circ}(\tau^\circ)$.
        \begin{figure}[!h]
                \caption{}\label{Fig:flip_same_weak_signature_example_circ}
                \centering
                \includegraphics[scale=.4]{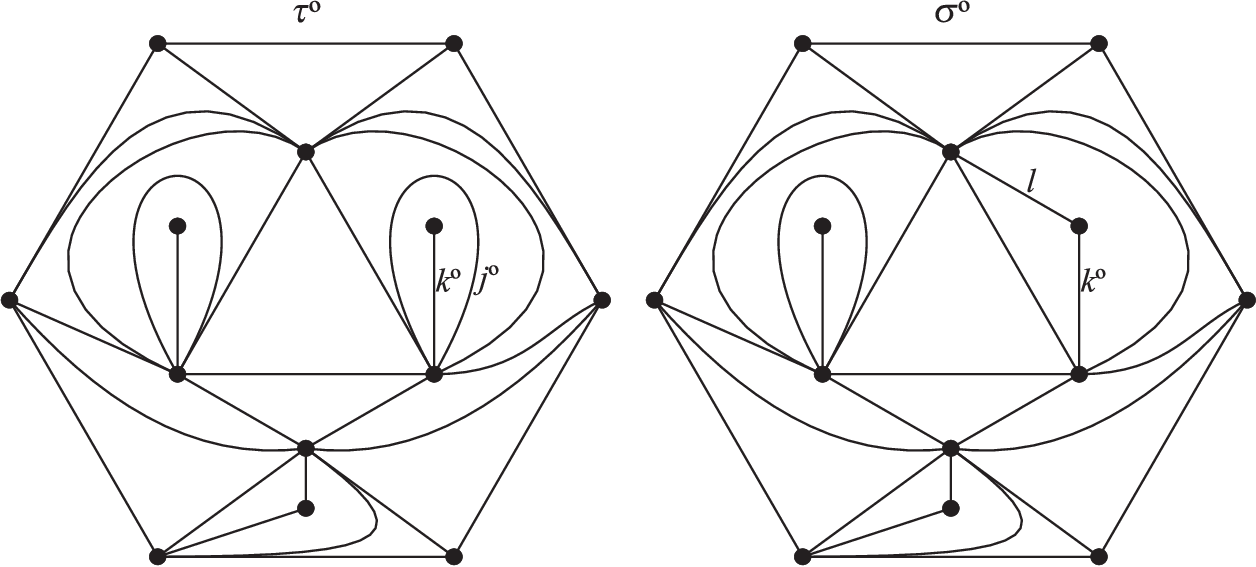}
        \end{figure}
The fact that the tagged arc $j'$ in Figure \ref{Fig:flip_same_weak_signature_example} equals the ordinary arc $l$ in Figure \ref{Fig:flip_same_weak_signature_example_circ} is pure coincidence: $\tau$ and $\sigma$ could as well have been such that the end of $j'$ not incident to $k$ were a notch.What is important is that the functions $?^\circ:\tau\to\tau^\circ$ and $?^\circ:\sigma\to\sigma^\circ$ respectively send $j$ to $j^\circ$ and $j'$ to $l$, which amounts to the commutativity of the left square in the diagram \eqref{eq:commutative-diagram-ordinary-flips} (for tagged triangulations with the same weak signature). Let us stress also the fact that $\tagfunction_{\epsilon_\tau}(k^\circ)=k=\tagfunction_{\epsilon_\sigma}(k^\circ)$, $\tagfunction_{\epsilon_\tau}(j^\circ)=j$, $\tagfunction_{\epsilon_\sigma}(l)=j'$, which amounts to the commutativity of the right square of the diagram \eqref{eq:commutative-diagram-ordinary-flips} (provided $\epsilon_\tau=\epsilon_\sigma$, of course).

If we on the other hand flip the tagged arc $k\in\tau$, we obtain the tagged triangulation $\rho=f_k(\tau)$ depicted on the left side of Figure \ref{Fig:flip_different_weak_signature_example}.
        \begin{figure}[!h]
                \caption{}\label{Fig:flip_different_weak_signature_example}
                \centering
                \includegraphics[scale=.4]{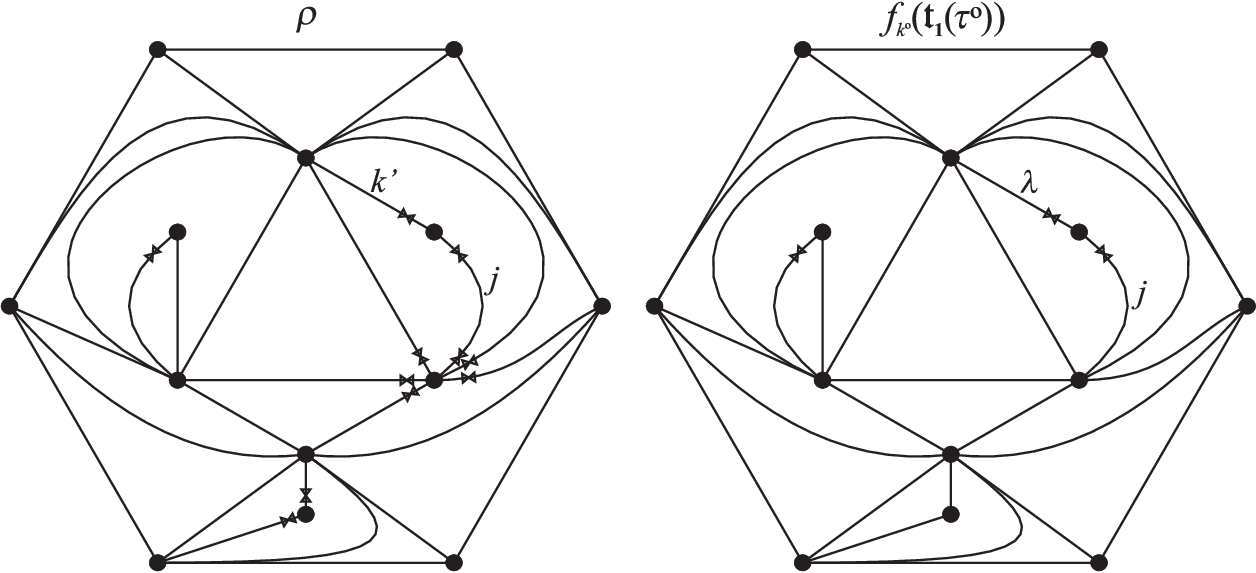}
        \end{figure}
On the right side of this Figure we can see the tagged triangulation $f_{\tagfunction_{\mathbf{1}}(k^\circ)}(\tagfunction_{\mathbf{1}}(\tau^\circ))$, and the equality $\tagfunction_{\epsilon_\tau\epsilon_\rho}(\rho^\circ)=f_{\tagfunction_{\mathbf{1}}(k^\circ)}(\tagfunction_{\mathbf{1}}(\tau^\circ))$, stated in the fifth assertion of Proposition \ref{prop:tagfunction-and-circ}, is easily verified.
Here again, the fact that the tagged arc $k'\in\rho$ equals the tagged arc $\lambda\in\tagfunction_{\epsilon_\tau\epsilon_\rho}(\rho^\circ)=f_{\tagfunction_{\mathbf{1}}(k^\circ)}(\tagfunction_{\mathbf{1}}(\tau^\circ))$ is mere coincidence. What matters is that the diagram \eqref{eq:commutative-diagram-tagged-flips} commutes.
$\hfill{\blacktriangle}$
\end{ex}

Let $\surf$ be any surface with marked points. To each ideal triangulation $\tau$ we associate two quivers $\unredqtau$ and $\qtau$ as follows. Both $\unredqtau$ and $\qtau$ have $\tau$ as vertex set, i.e., their vertices are precisely the arcs in $\tau$.

To define the number of arrows between any two given vertices of $\unredqtau$, we need an auxiliary function $\pi_\tau:\tau\rightarrow\tau$ which we now describe. Given an arc $i\in\tau$, let $\pi_\tau(i)$ be the arc in $\tau$ defined as follows. If $i$ is not the folded side of a self-folded triangle, then $\pi_\tau(i)=i$, whereas if $i$ is the folded side of a self-folded triangle $\triangle$, and $k\in\tau$ is the loop that encloses $\triangle$, then $\pi_\tau(i)=k$. Given arcs $i,j\in \tau$, the number of arrows of $\unredqtau$ that go from $j$ to $i$ is defined to be precisely the number of ideal triangles $\triangle$ of $\tau$ that satisfy the following conditions:
\begin{itemize}
\item $\triangle$ is not a self-folded triangle;
\item $\pi_\tau(i)$ and $\pi_\tau(j)$ are contained in $\triangle$;
\item inside $\triangle$, $\pi_\tau(j)$ directly precedes $\pi_\tau(i)$ according to the clockwise orientation of $\triangle$ which is inherited from the orientation of $\Sigma$.
\end{itemize}

The quiver $\unredqtau$ is said to be the \emph{unreduced signed-adjacency quiver} of $\tau$ (cf. \cite[Definition 8]{Labardini1}). Depending on $\tau$, it may or may not be 2-acyclic. By definition, $\qtau$ is the quiver obtained from $\unredqtau$ by deleting all 2-cycles. We call $\qtau$ the \emph{signed-adjacency quiver} of $\tau$ (cf. the two paragraphs that precede Theorem 7 in \cite{Labardini1}).

How about the (unreduced) signed-adjacency quiver of an arbitrary tagged triangulation of $\surf$? Let $\tau$ be any such. By definition, $\unredqtau$ and $\qtau$ are the quivers respectively obtained from $\widehat{Q}(\tau^\circ)$ and $Q(\tau^\circ)$ by replacing each $i\in\tau^\circ$ with $\tagfunction_{\epsilon_\tau}(i)\in\tau$ as a vertex of the respective quiver (see Definition \ref{def:tagfunction(ideal-arc)}).

The following result of Fomin-Shapiro-Thurston is, together with Derksen-Weyman-Zelevinsky's mutation theory of quivers with potential, one of the main motivations for the constructions of \cite{Labardini1} and this paper.

\begin{thm}\label{thm:FST}\cite[Proposition 4.8 and Lemma 9.7]{FST} Let $\surf$ be any surface with marked points, and let $\tau$ and $\sigma$ be tagged triangulations of $\surf$. If $\sigma$ can be obtained from $\tau$ by the flip of a tagged arc $i\in\tau$, then $\qsigma=\mu_i(\qtau)$.
\end{thm}

\begin{remark}\label{rem:Qtau-and-unredQtau}\begin{enumerate}\item The definition of $\qtau$ for an arbitrary tagged triangulation of a surface with marked points is due to
Fomin-Shapiro-Thurston. Indeed, in \cite[Definitions 4.1 and 9.6]{FST} they associated a skew-symmetric matrix $B(\tau)$ to each tagged triangulation $\tau$. By a well-known correspondence between skew-symmetric matrices and 2-acyclic quivers, $B(\tau)$ gives rise to a 2-acyclic quiver. This quiver is precisely $\qtau$.
\item The quiver $\unredqtau$, defined in \cite[Definition 8]{Labardini1} in the particular case where one has an ideal triangulation $\tau$, is introduced with the aim of being able to define a potential $\stau$ on $\qtau$ in a way which is as uniform as possible (see Definitions \ref{def:obvious-cycles} and \ref{def:QP-of-tagged-triangulation} below). See Remark \ref{rem:reduction-already-done} as well.
\end{enumerate}
\end{remark}

\section{Definition of $\qstau$}\label{sec:definition-of-Stau}

Let $\surf$ be a surface, with puncture set $\punct$. No assumptions on the boundary of $\Sigma$ are made throughout this section.

\begin{defi}\label{def:obvious-cycles}\cite[Definition 23]{Labardini1} Let $\tau$ be an ideal triangulation of $\surf$, and $\mathbf{y}=(y_p)_{p\in\punct}$ be a choice of non-zero scalars (one scalar $y_p\in\field$ per puncture $p$).
\begin{itemize} \item Each interior non-self-folded ideal triangle $\triangle$ of $\tau$ gives rise to an oriented triangle $\alpha\beta\gamma$ of $\unredqtau$, let $\widehat{S}^\triangle(\tau)=\alpha\beta\gamma$ be such oriented triangle up to cyclical equivalence.
\item If the interior non-self-folded ideal triangle $\triangle$ with sides $j$, $k$ and $l$, is adjacent to two self-folded triangles like in the configuration of Figure \ref{adsftriangs},
        \begin{figure}[!h]
                \caption{}\label{adsftriangs}
                \centering
                \includegraphics[scale=.5]{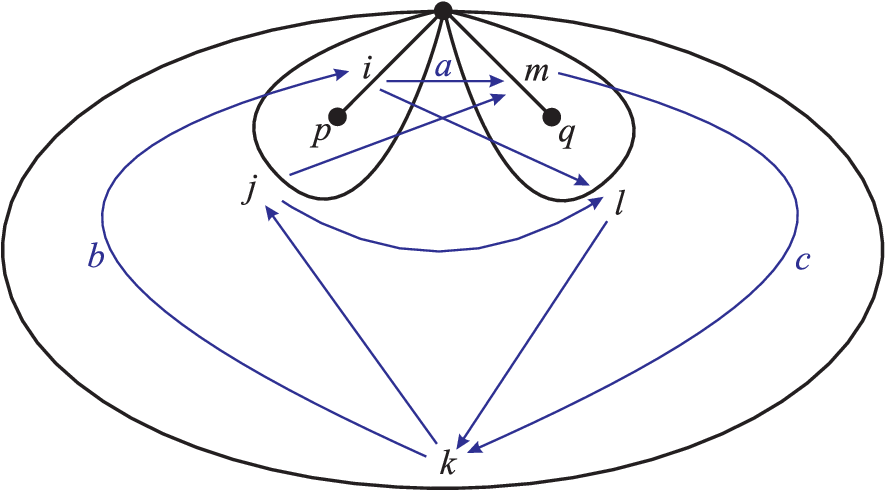}
        \end{figure}
define
$\widehat{U}^\triangle(\tau,\mathbf{y})=y_p^{-1}y_q^{-1}abc$ (up to cyclical equivalence), where $p$ and $q$ are the punctures enclosed in the self-folded triangles adjacent to $\triangle$.
Otherwise, if it is adjacent to less than two self-folded triangles, define $\widehat{U}^\triangle(\tau,\mathbf{y})=0$.
\item If a puncture $p$ is adjacent to exactly one arc $\arc$ of $\tau$, then $\arc$ is the folded side of a self-folded triangle of $\tau$ and around $\arc$ we have the configuration shown in Figure \ref{sftriangle}.
        \begin{figure}[!h]
                \caption{}\label{sftriangle}
                \centering
                \includegraphics[scale=.5]{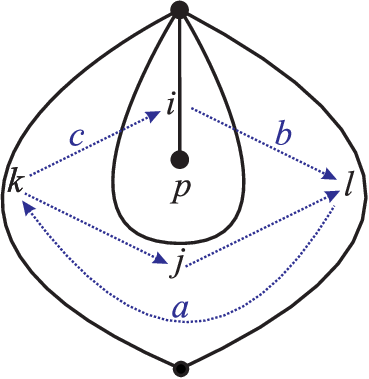}
        \end{figure}
In case both $k$ and $l$ are indeed arcs of $\tau$ (and not part of the boundary of $\surfnoM$), we define
$\widehat{S}^p(\tau,\mathbf{y})=-y_p^{-1}abc$ (up to cyclical equivalence).
Otherwise, if either $k$ or $l$ is a boundary segment, we define $\widehat{S}^{p}(\tau,\mathbf{y})=0$.
\item If a puncture $p$ is adjacent to more than one arc, delete all the loops incident to $p$ that enclose self-folded triangles. The arrows between the remaining arcs adjacent to $p$ form a unique cycle $a^p_1\ldots a^p_{d_p}$, without repeated arrows, that exhausts all such remaining arcs and gives a complete round around $p$ in the counter-clockwise direction defined by the orientation of $\surfnoM$. We define $\widehat{S}^p(\tau,\mathbf{y})=y_pa^p_1\ldots a^p_{d_p}$ (up to cyclical equivalence).
\end{itemize}
The \emph{unreduced potential $\widehat{S}(\tau,\mathbf{y})\in \RA{\unredqtau}$ associated to the ideal triangulation $\tau$ with respect to the choice $\mathbf{y}=(y_p)_{p\in\punct}$} is defined to be:
\begin{eqnarray}\widehat{S}(\tau,\mathbf{y}) &=&
\sum_\triangle\left(\widehat{S}^\triangle(\tau)+
\widehat{U}^\triangle(\tau,\mathbf{y})\right)+
\sum_{p\in\punct}\left(\widehat{S}^p(\tau,\mathbf{y})\right)
\end{eqnarray}
where the first sum runs over all interior non-self-folded triangles of $\tau$. We define $(\qtau,S(\tau,\mathbf{y}))$ to be the reduced part of $(\unredqtau,\widehat{S}(\tau,\mathbf{y}))$.
\end{defi}

Recall that given a tagged triangulation $\tau$, the function $\tagfunction_{\epsilon_\tau}:\tau^\circ\rightarrow\tau$ is a specific bijection. By its very definition, $\qtau$ is the quiver obtained from $Q(\tau^\circ)$ by replacing each vertex $i\in\tau^\circ$ with $\tagfunction_{\epsilon_\tau}(i)\in\tagfunction_{\epsilon_\tau}(\tau^\circ)=\tau$. During this replacement, the arrow set does not suffer any change whatsoever: if $a$ is an arrow in $Q(\tau^\circ)$ going from $j\in\tau^\circ$ to $i\in\tau^\circ$, then $a$ itself is an arrow in $\qtau$ going from $\tagfunction_{\epsilon_\tau}(j)$ to $\tagfunction_{\epsilon_\tau}(i)$. Therefore, the function $\tagfunction_{\epsilon_\tau}:\tau^\circ\rightarrow\tau$, together with the identity function of the arrow set, is a specific quiver isomorphism $Q(\tau^\circ)\rightarrow\qtau$, which we denote by $\tagfunction_{\epsilon_\tau}$ also. As such, $\tagfunction_{\epsilon_\tau}$ induces a $K$-algebra isomorphism $\RA{Q(\tau^\circ)}\rightarrow\RA{Q(\tau)}$, which, in a slight abuse of notation, we denote by $\tagfunction_{\epsilon_\tau}$ as well.

We are ready to define the potential of an arbitrary tagged triangulation. Given a tagged triangulation $\tau$ of $\surf$ and a choice $\mathbf{x}=(x_p)_{p\in\punct}$ be a choice of non-zero scalars, we shall write $\epsilon_\tau\cdot\mathbf{x}$ to denote the choice $(\epsilon_\tau(p)x_p)_{p\in\punct}$, which clearly consists of non-zero scalars as well.

\begin{defi}\label{def:QP-of-tagged-triangulation} Let $\tau$ be a tagged triangulation of $\surf$, and $\mathbf{x}=(x_p)_{p\in\punct}$ be a choice of non-zero scalars (one scalar $x_p\in\field$ per puncture $p$). The \emph{potential associated to $\tau$ with respect to the choice $\mathbf{x}=(x_p)_{p\in\punct}$} is defined to be $\stau=\tagfunction_{\epsilon_\tau}(S(\tau^\circ,\epsilon_\tau\cdot\mathbf{x}))\in\RA{\qtau}$.
\end{defi}

\begin{remark}\label{rem:reduction-already-done}
\begin{enumerate}\item If $\tau$ is an ideal triangulation, then, identifying $\tau$ with $\tagfunction_{\mathbf{1}}(\tau)$ via the bijection $\tagfunction_{\mathbf{1}}:\tau\rightarrow\tagfunction_{\mathbf{1}}(\tau)$ (that is, identifying each $i\in\tau$ with $\tagfunction_{\mathbf{1}}(i)\in\tagfunction_{\mathbf{1}}(\tau)$), it is straightforward to see that the QP $\qstau$ defined above coincides with the one defined in \cite{Labardini1}.
\item If $\tau$ is a non-ideal tagged triangulation of a punctured surface with non-empty boundary, the QP $\qstau$ defined above does not coincide with the one defined in \cite{CI-LF} since the latter does not take the weak signature $\epsilon_\tau$ into account. On the other hand, the cycles that appear in the potential defined in \cite{CI-LF} for $\tau$ are precisely the cycles that appear in the potential $\stau$ defined here. Proposition \ref{prop:scalars-are-irrelevant} below will imply the two alluded QPs are right-equivalent (see the line right after Example \ref{ex:walk-on-dimer}).
\item For any ideal triangulation $\tau$, it is possible to define the potential $\stau\in\RA{\qtau}$ directly, that is, without using the QP $\unredqstau$ as an intermediary, but it is the author's opinion that a direct definition of $\stau$ would be more cumbersome as it would make it necessary to consider further separate configurations to locate `obvious' cycles on $\qtau$. On the other hand, the only situation where one needs to apply reduction to $(\unredqtau,\unredstau)$ in order to obtain $\stau$ is when there is some puncture incident to exactly two arcs of the ideal triangulation $\tau$. The reduction is done explicitly in \cite[Section 3]{Labardini1}.
\end{enumerate}
\end{remark}

Let us illustrate Definition \ref{def:QP-of-tagged-triangulation} with a couple of examples.

\begin{ex}\label{example:tagged-triang-tau} On the left side of Figure \ref{Fig:3_punctured_torus_tagged_triang_and_quiver} we can see a three-times-punctured torus (with no boundary), with the scalars $x_p$, $p\in\punct$, labeling the three punctures, and a tagged triangulation $\tau$ of it. On the right side we can see the ideal triangulation $\tau^\circ$ and the quiver $Q(\tau^\circ)$ drawn on the surface.
        \begin{figure}[!h]
                \caption{}\label{Fig:3_punctured_torus_tagged_triang_and_quiver}
                \centering
                \includegraphics[scale=.5]{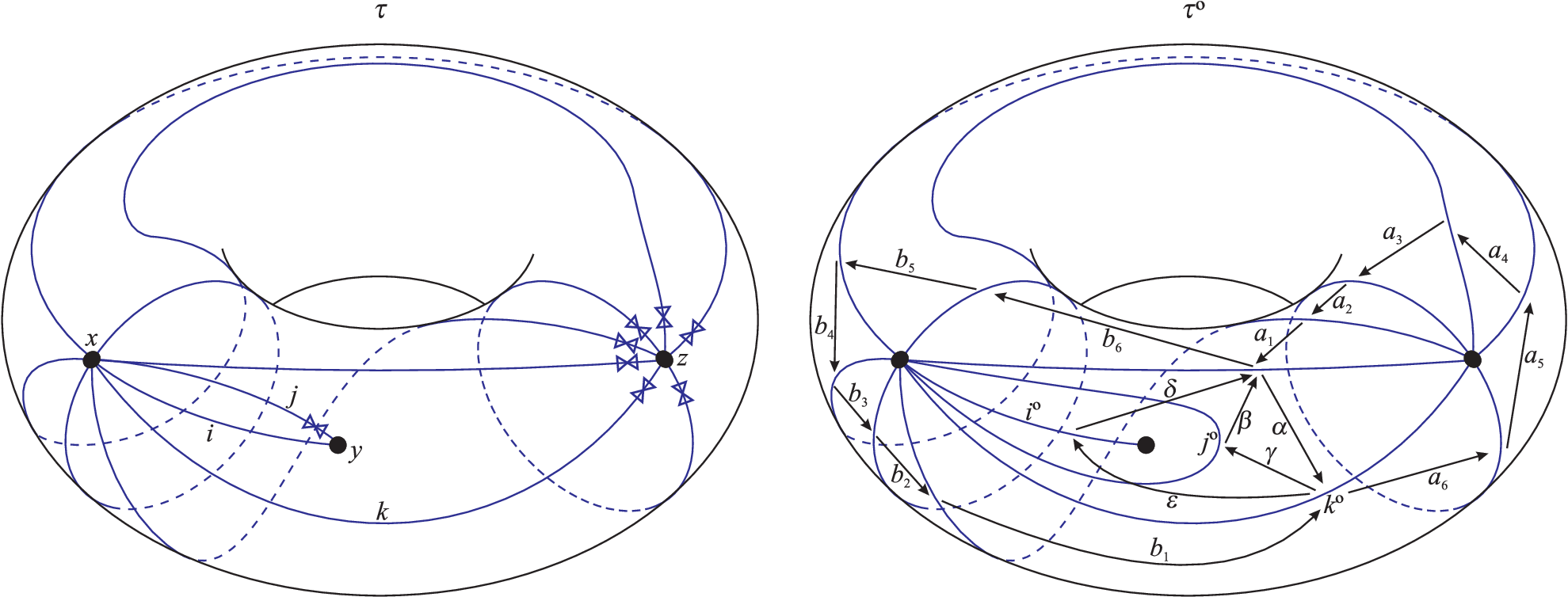}
        \end{figure}

The function $\tagfunction_{\epsilon_\tau}:\arcsinsurf\to\taggedinsurf$ restricts to a bijection $\tagfunction_{\epsilon_\tau}:\tau^\circ\to\tau$ acting by
$$
\tagfunction_{\epsilon_\tau}:\ \ i^\circ\mapsto i, \ \ j^\circ\mapsto j,
$$
and in the obvious way on the rest of the arcs of $\tau^\circ$. Since the 2-acyclic quiver $\qtau$ clearly coincides with the unreduced quiver $\unredqtau$, Definition \ref{def:QP-of-tagged-triangulation} tells us that
$$
\stau=\alpha\beta\gamma+a_1b_2b_6+a_2a_6b_1+a_3b_4a_5+a_4b_5b_3+x\delta\varepsilon b_1b_2b_3b_4b_5b_6-y^{-1}\alpha\delta\varepsilon-z\alpha a_1a_2a_3a_4a_5a_6.
$$
$\hfill{\blacktriangle}$
\end{ex}

\begin{ex}\label{example:tagged-triang-sigma} Let us flip the tagged arc $k$ of the tagged triangulation $\tau$ from the previous example. The tagged triangulation $\sigma=f_k(\tau)$, its ideal counterpart $\sigma^\circ$, and the quiver $Q(\sigma^\circ)$ are depicted in Figure \ref{Fig:3_punctured_torus_tagged_triang_and_quiver_flipped}.
        \begin{figure}[!h]
                \caption{}\label{Fig:3_punctured_torus_tagged_triang_and_quiver_flipped}
                \centering
                \includegraphics[scale=.5]{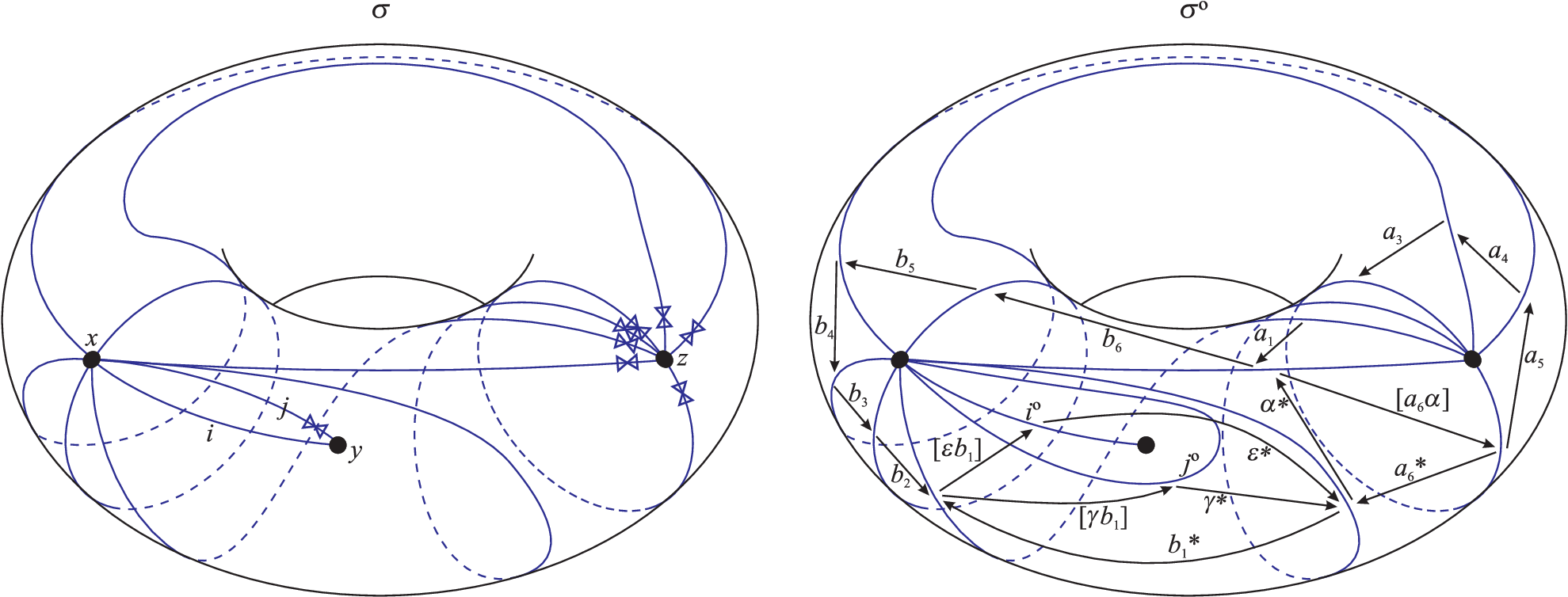}
        \end{figure}

Just as in the previous example, we have the bijection $\tagfunction_{\epsilon_\sigma}:\sigma^\circ\to\sigma$ acting by
$$
\tagfunction_{\epsilon_\sigma}:\ \ i^\circ\mapsto i, \ \ j^\circ\mapsto j,
$$
and in the obvious way on the rest of the arcs of $\sigma^\circ$. Since the 2-acyclic quiver $\qsigma$ clearly coincides with $\unredqsigma$, applying Definition \ref{def:QP-of-tagged-triangulation} we obtain
\begin{eqnarray}\nonumber
\ssigma & = & \gamma^*[\gamma b_1]b_1^*+a_6^*[a_6\alpha]\alpha^*+a_1b_2b_6+a_3b_4a_5+a_4b_5b_3
\\
\nonumber
& & +x\alpha^*\varepsilon^*[\varepsilon b_1]b_2b_3b_4b_5b_6-y^{-1}\varepsilon^*[\varepsilon b_1]b_1^*-za_1b_1^*a_6^*a_3a_4a_5[a_6\alpha].
\end{eqnarray}
$\hfill{\blacktriangle}$
\end{ex}

Now we recall one of the main results of \cite{Labardini1}.

\begin{thm}[{\cite[Theorem 30]{Labardini1}}]\label{thm:ideal-flips<->mutations} Let $\tau$ and $\sigma$ be ideal triangulations of $\surf$. If
$\sigma=f_k(\tau)$ for some arc $k\in\tau$, then $\mu_k\qstau$ and $\qssigma$ are right-equivalent QPs.
\end{thm}

\begin{remark}\label{rem:k-is-not-a-folded-side} Note that in the situation of Theorem \ref{thm:ideal-flips<->mutations}, the fact that $\tau$ and $\sigma=f_k(\tau)$ are ideal triangulations implies that the arc $k$ cannot be a folded side of $\tau$. Conversely, $\tau$ is obtained from $\sigma$ by the flip of an arc which is not a folded side of $\sigma$.
\end{remark}

Let us illustrate Theorem \ref{thm:ideal-flips<->mutations} with an example.

\begin{ex} In this example we put tagged triangulations completely aside. Consider the ideal triangulation $\tau^\circ$ depicted on the right side of Figure \ref{Fig:3_punctured_torus_tagged_triang_and_quiver}. Its associated potential is given by
$$
S(\tau^\circ,\mathbf{x})=\alpha\beta\gamma+a_1b_2b_6+a_2a_6b_1+a_3b_4a_5+a_4b_5b_3+x\delta\varepsilon b_1b_2b_3b_4b_5b_6-y^{-1}\alpha\delta\varepsilon+z\alpha a_1a_2a_3a_4a_5a_6
$$
(note that the terms $x\delta\varepsilon b_1b_2b_3b_4b_5b_6$ and $z\alpha a_1a_2a_3a_4a_5a_6$ are accompanied by a positive sign, whereas $y^{-1}\alpha\delta\varepsilon$ is accompanied by a negative sign; this is in sync with the fact that ideal triangulations correspond to tagged triangulations with non-negative signature).

If we apply the so-called \emph{premutation} $\widetilde{\mu}_{k^\circ}$ to $(Q(\tau^\circ),S(\tau^\circ,\mathbf{x}))$ (cf. \cite[Equations (5.8) and (5.9)]{DWZ1}), we obtain the QP $(\widetilde{\mu}_{k^\circ}(Q(\tau^\circ)),\widetilde{\mu}_{k^\circ}(S(\tau^\circ,\mathbf{x})))$, where
\begin{eqnarray}\nonumber
\widetilde{\mu}_{k^\circ}(S(\tau^\circ,\mathbf{x})) & = & [S(\tau^\circ,\mathbf{x})]+\triangle_{k^\circ}(Q(\tau^\circ)) \ \ = \ \
\beta[\gamma\alpha]+a_1b_2b_6+a_2[a_6b_1]+a_3b_4a_5+a_4b_5b_3\\
\nonumber
& & +x\delta[\varepsilon b_1]b_2b_3b_4b_5b_6-y^{-1}\delta[\varepsilon\alpha]+z a_1a_2a_3a_4a_5[a_6\alpha]\\
\nonumber
& & +\gamma^*[\gamma\alpha]\alpha^*+\varepsilon^*[\varepsilon\alpha]\alpha^*+a_6^*[a_6\alpha]\alpha^*
+\gamma^*[\gamma b_1]b_1^*+\varepsilon^*[\varepsilon b_1]b_1^*+a_6^*[a_6b_1]b_1^*
\end{eqnarray}
(we have not drawn the quiver $\widetilde{\mu}_{k^\circ}(Q(\tau^\circ))$). If we then apply Derksen-Weyman-Zelevinsky's reduction process we obtain the QP $(\mu_{k^\circ}(Q(\tau^\circ)),\mu_{k^\circ}(S(\tau^\circ,\mathbf{x})))$, where
\begin{eqnarray}\nonumber
\mu_{k^\circ}(S(\tau^\circ,\mathbf{x})) & = &a_1b_2b_6+a_3b_4a_5+a_4b_5b_6+xy\alpha^*\varepsilon^*[\varepsilon b_1]b_2b_3b_4b_5b_6-zb_1^*a_6^*a_3a_4a_5[a_6\alpha]a_1\\
\nonumber
& & +a_6^*[a_6\alpha]\alpha^*+\gamma^*[\gamma b_1]b_1^*+\varepsilon^*[\varepsilon b_1]b_1^*.
\end{eqnarray}

If we on the other hand flip the arc $k^\circ\in\tau^\circ$ we obtain the ideal triangulation $\sigma^\circ=f_{k^\circ}(\tau^\circ)$ depicted on the right side of Figure \ref{Fig:3_punctured_torus_tagged_triang_and_quiver_flipped}. Its associated potential is
\begin{eqnarray}\nonumber
S(\sigma^\circ,\mathbf{x}) & = & \gamma^*[\gamma b_1]b_1^*+a_6^*[a_6\alpha]\alpha^*+a_1b_2b_6+a_3b_4a_5+a_4b_5b_3
+x\alpha^*\varepsilon^*[\varepsilon b_1]b_2b_3b_4b_5b_6\\
\nonumber
& & -y^{-1}\varepsilon^*[\varepsilon b_1]b_1^*+za_1b_1^*a_6^*a_3a_4a_5[a_6\alpha].
\end{eqnarray}

The right-equivalence $\varphi:\mu_{k^\circ}(Q(\tau^\circ),S(\tau^\circ,\mathbf{x}))\rightarrow(Q(\sigma^\circ),S(\sigma^\circ,\mathbf{x}))$ witnessing the proof of Theorem \ref{thm:ideal-flips<->mutations} can be chosen to act on arrows by
$$
\varphi: \ \ \ \alpha^*\mapsto-\alpha^*, \ \varepsilon^*\mapsto-y^{-1}\varepsilon^*, \ a_6^*\mapsto-a_6^*.
$$
We end the example with a remark: The term ``premutation", with which we refer to $\widetilde{\mu}_{k^\circ}$, is very rarely used in the literature (in particular, it is not used in \cite{DWZ1}).  $\hfill{\blacktriangle}$
\end{ex}

\section{Flip/mutation of folded sides: The problem}\label{sec:folded-sides-the-problem}

Consider the tagged triangulations $\tau$ and $\sigma=f_\arc(\tau)$ of the three-times-punctured hexagon shown in Figure \ref{Fig:hexagon_3_punct}, where the quivers $\qtau$ and $\qsigma$ are drawn as well.
        \begin{figure}[!h]
                \caption{}\label{Fig:hexagon_3_punct}
                \centering
                \includegraphics[scale=.8]{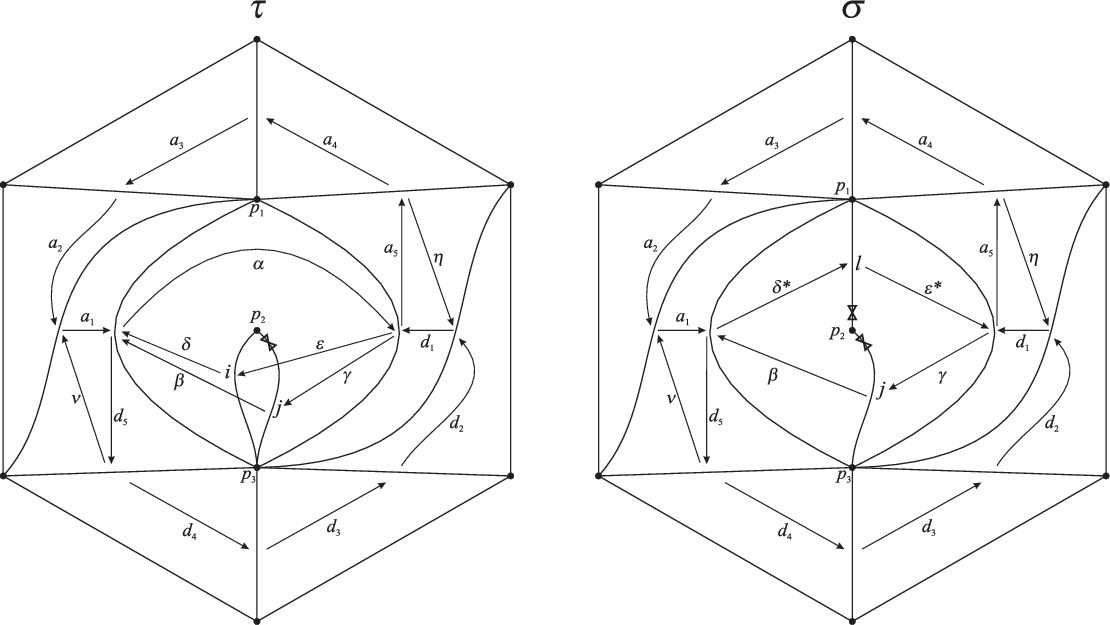}
        \end{figure}
The signature $\delta_\tau$ of $\tau$ is certainly non-negative, which means that $\tau$ can be thought to be an ideal triangulation (identifying each $k\in\tau^\circ$ with $\tagfunction_{1}(k)\in\tagfunction_{1}(\tau^\circ)=\tau$). The arc $i\in\tau$ is then the folded side of a self-folded triangle $\triangle$ of $\tau$ and the tagged arc $j$ corresponds to the loop contained in $\triangle$ that encloses $i$ in $\tau$.

Since $\sigma$ is obtained from $\tau$ by the flip of $i$, we would ``like" the QPs $\mu_i(\qstau)$ and $\qssigma$ to be right-equivalent, a fact that does not follow from Theorem \ref{thm:ideal-flips<->mutations}, because $i$ is a folded side of $\tau$ (see Remark \ref{rem:k-is-not-a-folded-side}). According to Definition \ref{def:QP-of-tagged-triangulation}, we have
\begin{eqnarray}\nonumber
\stau & = & \alpha\beta\gamma+a_1\nu d_5+a_5d_1\eta+x_{p_1}\alpha a_1a_2a_3a_4a_5-x_{p_2}^{-1}\alpha\delta\varepsilon+x_{p_3}\delta\varepsilon d_1d_2d_3d_4d_5 \ \ \ \text{and}\\
\nonumber
\ssigma & = & a_1\nu d_5+ a_5d_1\eta+x_{p_1}\varepsilon^*\delta^*a_1a_2a_3a_4a_5+x_{p_2}^{-1}\varepsilon^*\delta^*\beta\gamma+x_{p_3}\beta\gamma d_1d_2d_3d_4d_5.
\end{eqnarray}
If we apply the $\arc^{\operatorname{th}}$ QP-mutation to $\qstau$ we obtain the QP $(\mu_\arc(\qtau),\mu_\arc(\stau))$, where $\mu_\arc(\qtau)=\qsigma$ and
\begin{eqnarray}\nonumber
\mu_\arc(\stau) & = & a_1\nu d_5+a_5d_1\eta+x_{p_1}x_{p_2}\varepsilon^*\delta^* a_1a_2a_3a_4a_5+x_{p_2}\varepsilon^*\delta^*\beta\gamma
+x_{p_2}x_{p_3}\beta\gamma d_1d_2d_3d_4d_5\\
\nonumber
&&+x_{p_1}x_{p_2}x_{p_3} d_1d_2d_3d_4d_5a_1a_2a_3a_4a_5.
\end{eqnarray}
After a quick look at $\ssigma$ and $\mu_i(\stau)$ we see that not only do the scalars accompanying the cycles around the punctures not match, the set of (rotation classes of) cycles appearing in $\ssigma$ and $\mu_i(\stau)$ do not match either: the cycle $d_1d_2d_3d_4d_5a_1a_2a_3a_4a_5$ does not appear in $\ssigma$.

The appearance of the aforementioned cycle in $\mu_i(\stau)$ is general: whenever we have a self-folded triangle in an ideal triangulation $\tau$, with the property that the digon surrounding the self-folded triangle is entirely contained in the interior of $\surf$ (so that all marked points appearing in the digon are punctures), the QP-mutation with respect to the folded side will involve all the cycles appearing in $\ssigma$ and the cycle surrounding the digon. This may suggest that our definition of $\ssigma$ is ``wrong" and ``should" be changed taking the alluded cycle into account. However, if we make such change, after a few flips/mutations we will soon lose control over the potentials, for ``many" ``strange" cycles will have to be taken into account in each potential (not only ``many" at a time, also ``stranger and stranger").

There is nevertheless another possibility: the definition of $\ssigma$ may indeed be ``the right one". If so, then we have to show that there is an $R$-algebra isomorphism $\varphi:\RA{\mu_i(\qtau)}\rightarrow \RA{\qsigma}$ that simultaneously eliminates the digon-surrounding cycle and makes the coefficients of puncture-surrounding cycles match.

In the very particular situation of Figure \ref{Fig:hexagon_3_punct}, the $R$-algebra isomorphism $\varphi:\RA{\qsigma}\rightarrow \RA{\mu_\arc(\qtau)}$ acting by
\begin{equation}\label{eq:hexagon_3_punct_right_equiv}
\nu \mapsto \nu+x_{p_1}x_{p_2}x_{p_3}a_2a_3a_4a_5d_1d_2d_3d_4, \ \ \ \varepsilon^*\mapsto x_{p_2}\varepsilon^*, \ \ \ \beta\mapsto x_{p_2}\beta,
\end{equation}
and the identity on the remaining arrows of $\qsigma$, is a right-equivalence from $\qssigma$ to $(\mu_\arc(\qtau),\mu_\arc(\stau))$ (and so, $\varphi^{-1}$ eliminates the digon-surrounding cycle and makes the coefficients match). But if we have a careful look at the potentials $\ssigma$ and $\mu_i(\stau)$, we realize that the reason why $\varphi$ serves as a right-equivalence in this particular situation is the combination of the following two facts:
\begin{itemize}
\item The arrow $\nu$ appears in only one term of $\ssigma$;
\item $\partial_\nu(\ssigma)$ is a factor of the cycle in $\mu_i(\stau)$ that does not appear in $\ssigma$.
\end{itemize}
Had $\nu$ appeared in two terms of $\ssigma$, then the rule \eqref{eq:hexagon_3_punct_right_equiv} would have added two terms to $\ssigma$, not only one, and would have failed to be a right-equivalence $\qssigma\rightarrow(\mu_\arc(\qtau),\mu_\arc(\stau))$.

Now, the reason why $\nu$ appears in only one term of $\ssigma$ is the fact that it arises from the (signed) adjacency between two arcs at a marked point lying on $\partial\Sigma$. All in all, this means that for \eqref{eq:hexagon_3_punct_right_equiv} to indeed define a desired right-equivalence, we ultimately relied on the fact that the underlying surface has non-empty boundary.
We want to show that the desired right-equivalence exists independently of any assumption on the boundary.

\section{Popped potentials: Definition and flip/mutation dynamics}\label{sec:def-of-popped-potentials}

Here we go back to formal considerations. Let $\surf$ be a surface. Throughout this section we do not impose any assumption on the boundary of $\Sigma$ nor on the number of punctures.

Suppose $i$ is the folded side of a self-folded triangle of an ideal triangulation $\tau$. Let $j\in\tau$ be the loop that cuts out a once-punctured monogon and encloses $i$, and $q\in\punct$ be the puncture that lies inside the monogon cut out by $j$. Let $\pi_{i,j}^\tau:\tau\to\tau$ be the bijection that fixes $\tau\setminus\{i,j\}$ pointwise and interchanges $i$ and $j$. Then $\pi_{i,j}^\tau$ extends uniquely to a quiver automorphism of $\qtau$ that acts as the identity on all the arrows of $\qtau$ that are not incident to $i$ nor to $j$. This quiver automorphism yields a $\field$-algebra automorphism of $\RA{\qtau}$. In a slight notational abuse we shall denote this quiver and algebra automorphisms by $\pi_{i,j}^\tau$ as well.

\begin{remark}\label{rem:pi-is-not-R-alg-auto} Notice that $\pi_{i,j}^\tau$ is not an $R$-algebra automorphism of $\RA{\qtau}$.
\end{remark}

\begin{defi}\label{def:popped-potential} Let $\mathbf{x}=(x_p)_{p\in\punct}$ be a choice of non-zero scalars. With the notations from the preceding paragraph, let $S(\tau,\mathbf{y})$ be the potential associated to $\tau$ with respect to the choice of scalars $\mathbf{y}=(y_p)_{p\in\punct}$ defined by $y_p=(-1)^{\delta_{p,q}}x_p$ for all $p\in\punct$, where $\delta_{p,q}$ is the \emph{Kronecker delta}. The \emph{popped potential} of the quadruple $(\tau,\mathbf{x},i,j)$ is the potential $\wtau=\pi_{i,j}^\tau(S(\tau,\mathbf{y}))\in \RA{\qtau}$.
\end{defi}

\begin{remark}\label{rem:on-popped-pots}\begin{enumerate}\item The $\field$-algebra automorphism $\pi_{i,j}^\tau$ of $\RA{\qtau}$ has only been used to define $\wtau$, and it is actually possible to define $\wtau$ without any mention of $\pi_{i,j}^\tau$ (we have not done so in order to avoid an unnecessarily cumbersome definition), but we have not claimed that $\pi_{i,j}^\tau$ is a right-equivalence: it cannot be, precisely because it does not act as the identity on the vertex span $R$ (see \cite[Definition 4.2]{DWZ1}).
\item Note that an ideal triangulation $\tau$ can give rise to several popped potentials: each self-folded triangle of $\tau$ gives rise to one such. So, when speaking of the popped potential associated to an ideal triangulation, one has to specify the self-folded triangle with respect to which one is taking such potential.
\item The ultimate goal of this paper is to prove that whenever two tagged triangulations $\tau$ and $\sigma$ are related by a flip, the QPs $\qstau$ and $\qssigma$ are related by the corresponding QP-mutation (a fact that at the moment we know to be true only when both $\tau$ and $\sigma$ are ideal triangulations, see Theorem \ref{thm:ideal-flips<->mutations} and Remark \ref{rem:k-is-not-a-folded-side}). The popped potentials $\wtau$ will be used as a tool to achieve our goal: As we have seen in Section \ref{sec:folded-sides-the-problem}, showing that the flip of a folded side of an ideal triangulation $\tau$ results in a QP related to $\qstau$ by the corresponding QP-mutation may be way less straightforward or obvious than we would like. So, what we will do is show that whenever $i$ and $j$ are arcs forming a self-folded triangle of an ideal triangulation $\tau$, with $i$ as folded side and $j$ as enclosing loop, the QPs $\qstau$ and $\qwtau$ are right-equivalent (see Theorem \ref{thm:popping-is-right-equiv}); then we will show that the QP $\mu_i\qwtau$ is right-equivalent to $\qssigma$, where $\sigma$ is the tagged triangulation obtained from $\tau$ by flipping the arc $i$ (see Proposition \ref{prop:flip-mut-folded-sides} --the formal statement of this proposition uses a more accurate notation). From this and Theorem \ref{thm:ideal-flips<->mutations} we will deduce that if we flip any arc of an ideal triangulation $\tau$, the resulting QP is related to $\qstau$ by the corresponding QP-mutation, regardless of whether the flipped arc was or not a folded side of $\tau$  (see Theorem \ref{thm:leaving-positive-stratum}). Finally, we will reduce to Theorem \ref{thm:leaving-positive-stratum} the proof of Theorem \ref{thm:tagged-flips<->mutations}, which states that whenever two tagged triangulations $\tau$ and $\sigma$ are related by a flip, the QPs $\qstau$ and $\qssigma$ are related by the corresponding QP-mutation.
\end{enumerate}
\end{remark}

\begin{ex}\label{ex:stau-wtau} On the left side of Figure \ref{Fig:3_punctured_torus_ideal_triang_and_quiver} we can see a three-times-punctured torus (with no boundary), with the scalars $x_p$, $p\in\punct$, labeling the three punctures, and an ideal triangulation $\tau$ of it that has a (unique) self-folded triangle. On the right side we can see the quiver $\qtau$ drawn on the surface.
        \begin{figure}[!h]
                \caption{}\label{Fig:3_punctured_torus_ideal_triang_and_quiver}
                \centering
                \includegraphics[scale=.5]{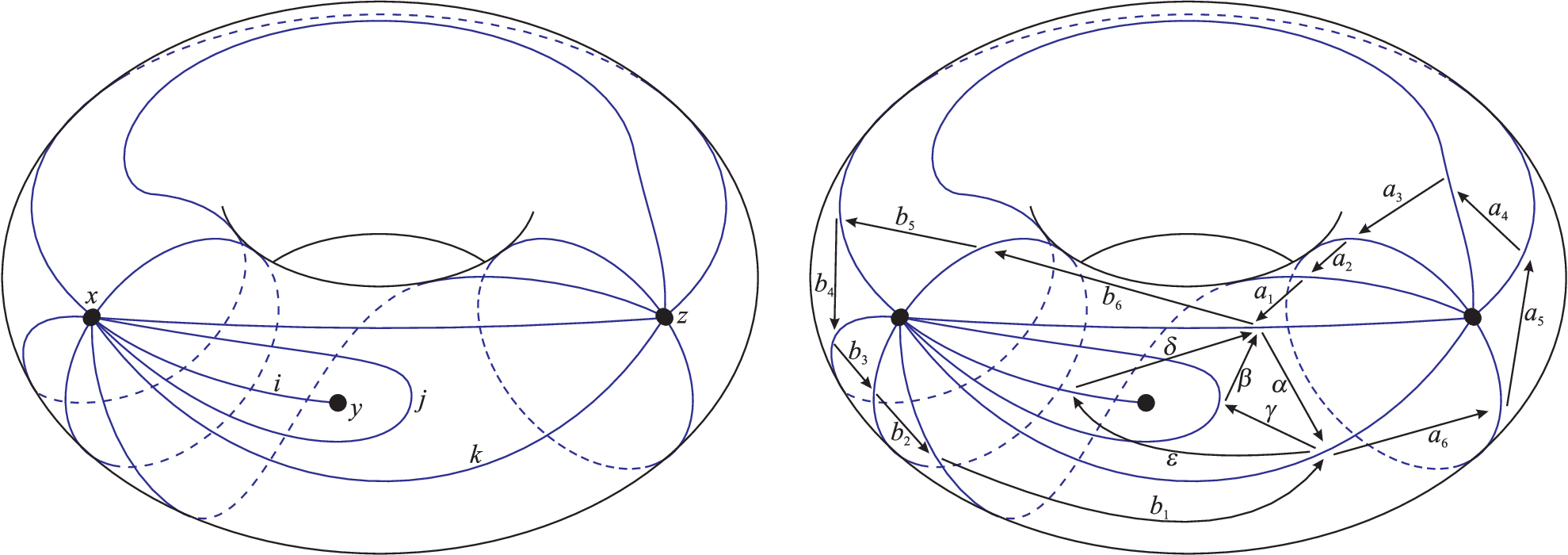}
        \end{figure}
Definition \ref{def:QP-of-tagged-triangulation} tells us that
$$
\stau=\alpha\beta\gamma+a_1b_2b_6+a_2a_6b_1+a_3b_4a_5+a_4b_5b_3+x\delta\varepsilon b_1b_2b_3b_4b_5b_6-y^{-1}\alpha\delta\varepsilon+z\alpha a_1a_2a_3a_4a_5a_6,
$$
whereas Definition \ref{def:popped-potential} tells us that the popped potential is
$$
\wtau=\alpha\delta\varepsilon+a_1b_2b_6+a_2a_6b_1+a_3b_4a_5+a_4b_5b_3+x\beta\gamma b_1b_2b_3b_4b_5b_6+y^{-1}\alpha\beta\gamma+z\alpha a_1a_2a_3a_4a_5a_6.
$$
$\hfill{\blacktriangle}$
\end{ex}

\begin{lemma}\label{lemma:pop-commutes-with-mutation} Let $\tau$ an ideal triangulation, and $\mathbf{x}=(x_p)_{p\in\punct}$ and $i,j\in\tau$ be as above. If $k\in\tau\setminus\{i,j\}$ is an arc such that $\sigma=f_k(\tau)$ happens to be an ideal triangulation, then $\mu_k(\qtau,\wtau)$ is right-equivalent to $(\qsigma,\wsigma)$, where $\wsigma$ is the popped potential of the quadruple $(\sigma,\mathbf{x},i,j)$.
\end{lemma}

\begin{proof}
Let $\mathbf{y}=((-1)^{\delta_{p,q}}x_p)_{p\in\punct}$ be as in Definition \ref{def:popped-potential}.
Notice that the bijection $\pi=\pi_{i,j}^\tau:\tau\to\tau$ induces a quiver automorphism $\psi$ of $\widetilde{\mu}_k(\qtau)$. Following a slight notational abuse, we shall denote also by $\psi$ the induced $K$-algebra automorphism of $\RA{\widetilde{\mu}_k(\qtau)}$ (which is not an $R$-algebra automorphism). It is straightforward to see that $\psi([S(\tau,\mathbf{y})])=[\pi(S(\tau,\mathbf{y}))]$ and $\psi(\triangle_k(Q))=\triangle_k(Q)$, and hence
\begin{eqnarray}\nonumber
\psi(\widetilde{\mu}_{k}(S(\tau,\mathbf{y}))) & = & \psi\left([S(\tau,\mathbf{y})]+\triangle_k(\qtau)\right) \ \ \ \ \ \text{by \cite[Equations (5.8) and (5.9)]{DWZ1}}\\
\nonumber & = & [\pi(S(\tau,\mathbf{y}))]+\triangle_k(\qtau)\\
\nonumber & = & \widetilde{\mu}_k(\pi(S(\tau,\mathbf{y}))) \ \ \ \ \ \ \ \ \ \ \ \ \ \ \ \ \ \ \ \text{by \cite[Equations (5.8) and (5.9)]{DWZ1}}\\
\nonumber & = &\widetilde{\mu}_k(\wtau)  \ \ \ \ \ \ \ \ \ \ \ \ \ \ \ \ \ \text{by Definition \ref{def:popped-potential}}.
\end{eqnarray}

Since $\sigma=f_k(\tau)$ is an ideal triangulation, Theorem \ref{thm:ideal-flips<->mutations} guarantees the existence of a right-equivalence $\varphi:\widetilde{\mu}_{k}(\qtau,S(\tau,\mathbf{y}))\rightarrow(Q(\sigma),S(\sigma,\mathbf{y}))\oplus(C,T)$, where $(C,T)$ is a trivial QP. The restriction of the quiver automorphism $\psi$ of $\widetilde{\mu}_k(\qtau)=(Q(\sigma),S(\sigma,\mathbf{y}))\oplus(C,T)$ to $Q(\sigma)$ is precisely $\pi^\sigma_{i,j}$.
This, together with the obvious fact that $\psi\varphi\psi^{-1}$ acts as the identity on the vertex span $R$, implies that $\psi\varphi\psi^{-1}:\RA{\widetilde{\mu}_k(\qtau)}\rightarrow \RA{\qsigma\oplus C}$ is an $R$-algebra isomorphism such that $\psi\varphi\psi^{-1}(\widetilde{\mu}_k(\wtau))=\psi\varphi(\widetilde{\mu}_{k}(S(\tau,\mathbf{y})))$ is cyclically-equivalent to
$\psi(S(\sigma,\mathbf{y})+T)=\wsigma+\psi(T)$. That is, $\psi\varphi\psi^{-1}$ is a right-equivalence $\psi\varphi\psi^{-1}:\widetilde{\mu}_{k}(\qtau,\wtau)\rightarrow(\qsigma,\wsigma)\oplus(\psi(C),\psi(T))$. Since $(\qsigma,\wsigma)$ is a reduced QP and $(\psi(C),\psi(T))$ is a trivial QP, Derksen-Weyman-Zelevinsky's Splitting Theorem \cite[Theorem 4.6]{DWZ1} allows us to conclude that the QPs $\mu_k(\qtau,\wtau)$ and $(\qsigma,\wsigma)$ are right-equivalent.
\end{proof}

Intuitively and roughly speaking, Lemma \ref{lemma:pop-commutes-with-mutation} says that for a fixed self-folded triangle $\triangle$, the popped QPs $(\qtau,\wtau)$ associated to ideal triangulations containing $\triangle$ have the same QP-mutation dynamics (ie, compatible with flips) as the QPs $(\qtau,\stau)$ as long as the two arcs in $\triangle$ are never flipped.

\begin{ex} Let us illustrate the proof of Lemma \ref{lemma:pop-commutes-with-mutation} in the situation of Example \ref{ex:stau-wtau}. The quiver automorphism $\pi=\pi_{i,j}^{\tau}:\qtau\to\qtau$ is easily seen to act by the rules
$$
\pi: \ \ \ \beta\mapsto\delta, \ \delta\mapsto\beta, \ \gamma\mapsto\varepsilon, \ \varepsilon\mapsto\gamma.
$$
This implies that the quiver automorphism $\psi:\widetilde{\mu}_k(\qtau)\to\widetilde{\mu}_k(\qtau)$ acts by the rules
$$
\psi: \ \ \ \beta\mapsto\delta, \ \delta\mapsto\beta, \ \gamma^*\mapsto\varepsilon^*, \ \varepsilon^*\mapsto\gamma^*, \ [\gamma\alpha]\mapsto[\varepsilon\alpha], \ [\varepsilon\alpha]\mapsto[\gamma\alpha], \ [\gamma b_1]\mapsto[\varepsilon b_1], \ [\varepsilon b_1]\mapsto[\gamma b_1],
$$
and hence
\begin{eqnarray}\nonumber
\psi([S(\tau,\mathbf{y})]) & = &
\psi(\beta[\gamma\alpha]+a_1b_2b_6+a_2[a_6b_1]+a_3b_4a_5+a_4b_5b_3\\
\nonumber &&+x\delta[\varepsilon b_1]b_2b_3b_4b_5b_6+y^{-1}\delta[\varepsilon\alpha]+za_1a_2a_3a_4a_5[a_6\alpha])\\
\nonumber
 & = &
\delta[\varepsilon\alpha]+a_1b_2b_6+a_2[a_6b_1]+a_3b_4a_5+a_4b_5b_3\\
\nonumber&&+x\beta[\gamma b_1]b_2b_3b_4b_5b_6+y^{-1}\beta[\gamma\alpha]+za_1a_2a_3a_4a_5[a_6\alpha].
\end{eqnarray}

Since
\begin{eqnarray}\nonumber
[\pi(S(\tau,\mathbf{y}))] & = &
[\alpha\delta\varepsilon+a_1b_2b_6+a_2a_6b_1+a_3b_4a_5+a_4b_5b_3\\
\nonumber&&+x\beta\gamma b_1b_2b_3b_4b_5b_6+y^{-1}\alpha\beta\gamma+z\alpha a_1a_2a_3a_4a_5a_6]\\
\nonumber
 & = &
\delta[\varepsilon\alpha]+a_1b_2b_6+a_2[a_6b_1]+a_3b_4a_5+a_4b_5b_3\\
\nonumber&&+x\beta[\gamma b_1]b_2b_3b_4b_5b_6+y^{-1}\beta[\gamma\alpha]+za_1a_2a_3a_4a_5[a_6\alpha],
\end{eqnarray}
the equality $\psi([S(\tau,\mathbf{y})])=[\pi(S(\tau,\mathbf{y}))]$ is obvious.

Turning to $\triangle_k(\qtau)$ and $\psi(\triangle_k(\qtau))$, we have
\begin{eqnarray}\nonumber
\triangle_k(\qtau) & = & \gamma^*[\gamma\alpha]\alpha^*+\varepsilon^*[\varepsilon\alpha]\alpha^*+a_6^*[a_6\alpha]\alpha^*
+\gamma^*[\gamma b_1]b_1^*+\varepsilon^*[\varepsilon b_1]b_1^*+a_6^*[a_6 b_1]b_1^*
\end{eqnarray}
and hence
\begin{eqnarray}\nonumber
\psi(\triangle_k(\qtau)) & = & \varepsilon^*[\varepsilon\alpha]\alpha^*+\gamma^*[\gamma\alpha]\alpha^*+a_6^*[a_6\alpha]\alpha^*
+\varepsilon^*[\varepsilon b_1]b_1^*+\gamma^*[\gamma b_1]b_1^*+a_6^*[a_6 b_1]b_1^*
\end{eqnarray}
obviously equals $\triangle_k(\qtau)$.

Therefore, as stated in the proof of Lemma \ref{lemma:pop-commutes-with-mutation}, we have $\psi(\widetilde{\mu}_k(S(\tau,\mathbf{y})))=\psi([S(\tau,\mathbf{y})]+\triangle_k(\qtau))=
[\pi(S(\tau,\mathbf{y}))]+\triangle_k(\qtau)=\widetilde{\mu}_k(\pi(S(\tau,\mathbf{y})))=\widetilde{\mu}_k(\wtau)$.

On the other hand, in Figure \ref{Fig:3_punctured_torus_ideal_triang_and_quiver_flipped} we see the ideal triangulation $\sigma=f_k(\tau)$ that results from flipping the arc $k$ of the ideal triangulation $\tau$ shown in Figure \ref{Fig:3_punctured_torus_ideal_triang_and_quiver}.
        \begin{figure}[!h]
                \caption{$\sigma=f_k(\tau)$ (cf. Figure \ref{Fig:3_punctured_torus_ideal_triang_and_quiver})}\label{Fig:3_punctured_torus_ideal_triang_and_quiver_flipped}
                \centering
                \includegraphics[scale=.5]{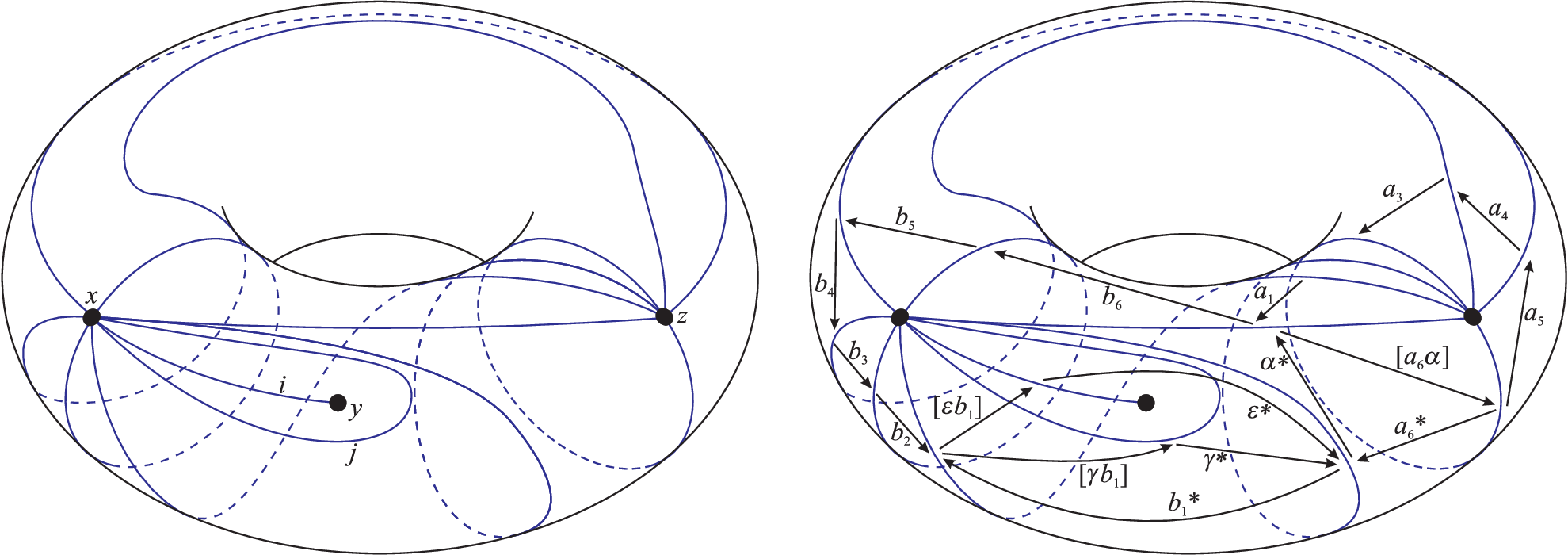}
        \end{figure}\\
Since it is the scalars $x_p$, $p\in\punct$, that are labeling the punctures, the potential associated to $\sigma$ with respect to the choice $\mathbf{y}=((-1)^{\delta_{p,q}}x_p)_{p\in\punct}$ is
\begin{eqnarray}\nonumber
S(\sigma,\mathbf{y}) & = &
\gamma^*[\gamma b_1]b_1^*+a_6^*[a_6\alpha]\alpha^*+a_1b_2b_6+a_3b_4a_5+a_4b_5b_3\\
\nonumber
& & +x[\varepsilon b_1]b_2b_3b_4b_5b_6\alpha^*\varepsilon^*+y^{-1}\varepsilon^*[\varepsilon b_1]b_1^*+za_1a_2a_3a_4a_5[a_6\alpha],
\end{eqnarray}
hence popped potential of the quadruple $(\sigma,\mathbf{x},i,j)$ is
\begin{eqnarray}\nonumber
\wsigma & = &
\varepsilon^*[\varepsilon b_1]b_1^*+a_6^*[a_6\alpha]\alpha^*+a_1b_2b_6+a_3b_4a_5+a_4b_5b_3\\
\nonumber
& & +x[\gamma b_1]b_2b_3b_4b_5b_6\alpha^*\gamma^*+y^{-1}\gamma^*[\gamma b_1]b_1^*+za_1a_2a_3a_4a_5[a_6\alpha].
\end{eqnarray}

The quiver $C=(Q_0,C_1,t,h)$, the potential $T$ and the right-equivalence $\varphi:\widetilde{\mu}_k(\qtau,S(\tau,\mathbf{y}))\to(\qsigma,S(\sigma,\mathbf{y}))\oplus(C,T)$ can be chosen so that \begin{eqnarray}\nonumber
C_1 & = & \{[\gamma\alpha], \beta, [\varepsilon\alpha], \delta, [a_6b_1], a_2\}, \ \text{with the functions $t,h:C_1\to Q_0$ defined in the obvious way},\\
\nonumber
T & = & \beta[\gamma\alpha]+a_2[a_6a_1]+y^{-1}\delta[\varepsilon\alpha],\\
\nonumber
\varphi & : & \beta\mapsto\beta+\alpha^*\gamma^*, \ a_2\mapsto a_2-b_1^*a_6^*, \ [a_6b_1]\mapsto[a_6b_1]+za_3a_4a_5[a_6\alpha]a_1, \ \delta\mapsto\delta+\alpha^*\varepsilon^*,\\
\nonumber
 & & [\varepsilon\alpha]\mapsto[\varepsilon\alpha]-xy[\varepsilon b_1]b_2b_3b_4b_5b_6, \ \varepsilon^*\mapsto y^{-1}\varepsilon^*, \ \alpha^*\mapsto-\alpha^*, [a_6\alpha]\mapsto-[a_6\alpha].
 \end{eqnarray}
This means that the right-equivalence $\psi\varphi\psi^{-1}:\widetilde{\mu}_k(\qtau,\wtau)\to(\qsigma,\wsigma)\oplus(\psi(C),\psi(T))$ acts according to the rules
\begin{eqnarray}\nonumber
\psi\varphi\psi^{-1} & : & \delta\mapsto\delta+\alpha^*\varepsilon^*, \ a_2\mapsto a_2-b_1^*a_6^*, \ [a_6b_1]\mapsto[a_6b_1]+za_3a_4a_5[a_6\alpha]a_1, \ \beta\mapsto\beta+\alpha^*\gamma^*,\\
\nonumber
 & & [\gamma\alpha]\mapsto[\gamma\alpha]-xy[\gamma b_1]b_2b_3b_4b_5b_6, \ \gamma^*\mapsto y^{-1}\gamma^*, \ \alpha^*\mapsto-\alpha^*, \ [a_6\alpha]\mapsto-[a_6\alpha].
\end{eqnarray}
In other words, the desired right-equivalence between $\mu_k(\qtau,\wtau)$ and $(\qsigma,\wsigma)$ can be obtained from the existent one between $\mu_k(\qtau,S(\tau,\mathbf{y}))$ and $(\qsigma,S(\sigma,\mathbf{y}))$ by formally applying the symbol exchanges $\beta\leftrightarrow\delta$ and $\gamma\leftrightarrow\varepsilon$.  $\hfill{\blacktriangle}$
\end{ex}

\section{Pop is right-equivalence: Statement and proof}\label{sec:pop-is-re}

Let $\surf$ be a surface. We assume that
\begin{equation}\label{eq:not-sphere-with-few-punctures}
\text{$\surf$ is not a sphere with less than 6 punctures}.
\end{equation}
This is actually an extremely mild assumption. Indeed, all the following surfaces satisfy \eqref{eq:not-sphere-with-few-punctures}:
\begin{eqnarray}\label{eq:no-boundary}
&&\text{all surfaces with non-empty boundary, with or without punctures, regardless of their genus;}\\
\label{eq:positive-genus}
&&\text{all positive-genus surfaces without boundary, and any number of punctures;}\\
\label{eq:spheres}
&&\text{all spheres with at least 6 punctures.}
\end{eqnarray}

\subsection{Statement}\label{subsec:Pop-Thm-statement}

The main result of the section (and instrumentally the main result of the paper) is the Popping Theorem, which we now state.

\begin{thm}[Popping Theorem]\label{thm:popping-is-right-equiv}
Suppose that $\surf$ satisfies \eqref{eq:not-sphere-with-few-punctures}. Let $\sigma$ be any ideal triangulation of $\surf$. If $i\in\sigma$ is a folded side of $\sigma$, enclosed by the loop $j$, then $\qssigma$ is right-equivalent to $\qwsigma$, where $\wsigma$ is the popped potential of the quadruple $(\sigma,\mathbf{x},i,j)$.
\end{thm}

\begin{proof} Suppose for a moment that there exists an ideal triangulation $\tau$ with $i,j\in\tau$, such that the pop of $(i,j)$ in $\tau$ induces right-equivalence, that is, such that
\begin{eqnarray}\label{eq:pop-is-re-for-tau}
&&\text{$\qstau$ is right-equivalent to $\qwtau$,}
\end{eqnarray}
where $\qwtau$ is the popped potential of the quadruple $(\tau,\mathbf{x},i,j)$. By Proposition \ref{prop:ideal-triangs-seqs-of-flips}, there exists an ideal flip-sequence $(\sigma_0,\sigma_1,\ldots,\sigma_s)$ with $\sigma_0=\sigma$ and $\sigma_s=\tau$, such that $i,j\in\sigma_r$ for all $r=0,\ldots,s$. Let $k_1\in\sigma_1,k_2\in\sigma_2,\ldots,k_s\in\sigma_s$ be the (ordinary) arcs such that $\sigma_{r-1}=f_{k_r}(\sigma_r)$ for $r=1,\ldots,s$. Then, using the symbol $\simeq$ to abbreviate ``is right-equivalent to",
\begin{eqnarray}
\nonumber\qssigma&\simeq&\mu_{k_1}\mu_{k_2}\ldots\mu_{k_s}\qstau \ \ \ \ \ \ \ \ \ \ \ \text{(by Theorem \ref{thm:ideal-flips<->mutations})}\\
\nonumber&\simeq&\mu_{k_1}\mu_{k_2}\ldots\mu_{k_s}\qwtau \ \ \ \ \ \text{(by \eqref{eq:pop-is-re-for-tau} and \cite[Corollary 5.4]{DWZ1})}\\
\nonumber&\simeq&\qwsigma \ \ \ \ \ \ \ \ \ \ \ \ \ \ \ \ \ \ \ \ \ \ \text{(by Lemma \ref{lemma:pop-commutes-with-mutation}).}
\end{eqnarray}
Theorem \ref{thm:popping-is-right-equiv} is thus proved up to showing the existence of an ideal triangulation $\tau$ such that $i,j\in\tau$ and satisfying \eqref{eq:pop-is-re-for-tau}.
\end{proof}

We devote the rest of the section to provide the piece missing in the proof of the Popping Theorem \ref{thm:popping-is-right-equiv}, namely:

\begin{prop}\label{prop:existence-of-popping-triangulation} Suppose $\surf$ is not a sphere with less than 6 punctures. Let $i$ and $j$ be (ordinary) arcs on $\surf$, such that $j$ is a loop cutting out a once-punctured monogon, and $i$ is the unique arc that connects the basepoint of $j$ with the puncture enclosed by $j$ and is entirely contained in the once-punctured monogon cut out by $j$. Then there exists an ideal triangulation $\tau$ of $\surf$ such that $i,j\in\tau$ and satisfying \eqref{eq:pop-is-re-for-tau}.
\end{prop}

\subsection{Proof for surfaces with empty boundary}\label{subsec:proof-Pop-Thm-empty-boundary} In this subsection we prove Proposition \ref{prop:existence-of-popping-triangulation} for all surfaces with empty boundary that satisfying the assumption of being different from a sphere with less than 7 punctures. The sphere with 6 punctures is dealt with in the fourth arXiv version of this paper.

Since any permutation of the puncture set $\punct=\marked$ can be extended to an orientation-preserving auto-diffeomorphism of $\Sigma$, in order to prove Proposition \ref{prop:existence-of-popping-triangulation} and this way finish the proof of the Popping Theorem for empty-boundary surfaces, it suffices to show the existence of an ideal triangulation $\tau$ of $\surf$ having a self-folded triangle whose pop in $\tau$ induces right-equivalence.

\begin{lemma}\label{lemma:existence-certain-triang} Let $\surf$ be a surface with empty boundary. If $|\marked|\geq 2$ and $\surf$ is not a sphere with less than 7 punctures, then there exists an ideal triangulation $\tau$ of $\surf$ and a set of arcs $\{i_1,i_2,i_3,i_4,i_5,i_6,i_7\}\subseteq\tau$ with the following properties:
\begin{enumerate}
\item $i_1$, $i_2$, $i_3$, $i_4$, $i_5$ and $i_6$ are six different arcs in $\tau$;
\item $i_2$, $i_3$, $i_4$, $i_5$, $i_6$ and $i_7$ are six different arcs in $\tau$;
\item the arcs $i_1$ and $i_7$ are loops (with basepoints that may or may not coincide; the arcs $i_1$ and $i_7$ themselves may or may not coincide);
\item the arcs $i=i_3$ and $j=i_2$ form a self-folded triangle of $\tau$, with $i$ as folded side and $j$ as enclosing loop;
\item the full subquiver of $\qtau$ determined by $\{i_1,i_2,i_3,i_4,i_5,i_6,i_7\}$ is
\begin{center}
$Q \ = \ $\begin{tabular}{c}
$ \xymatrix{
  & i_2 \ar[dl]_{\beta} & & i_5 \ar@<0.5ex>[dd]_{\eta_3 \ } \ar@<-0.5ex>[dd]^{\ \lambda_1} & \\
  i_1  \ar[rr]^{\alpha} & & i_4 \ar[ul]_{\gamma} \ar[dl]^{\varepsilon} \ar[ur]^{\eta_1}  & &  i_7 \ar[ul]_{\lambda_2}\\
 & i_3 \ar[ul]^{\delta} & & i_6 \ar[ul]^{\eta_2} \ar[ur]_{\lambda_3} &
}$
\end{tabular}
\end{center}
\item under the notation just established for the arrows of the full subquiver $Q$ of $\qtau$, the potentials $\stau$ and $\wtau$ are given by the formulas
\begin{eqnarray}\nonumber
\stau &=& -x_{p_2}^{-1}\alpha\delta\varepsilon+\eta_1\eta_2\eta_3+\lambda_1\lambda_2\lambda_3+\alpha\beta\gamma\\
\nonumber
&&
+Y(x_{p_1}\delta\varepsilon\eta_2\lambda_1\eta_1\alpha\Lambda+x_{p_3}\lambda_3\eta_3\lambda_2\Omega)\\
\nonumber
&&
+
Z(x_{p_1} \delta\varepsilon\eta_2\lambda_1\eta_1\alpha\Lambda\lambda_3\eta_3\lambda_2\Omega)+S'(\tau),\\
\nonumber
\wtau &=& \alpha\delta\varepsilon+\eta_1\eta_2\eta_3+\lambda_1\lambda_2\lambda_3+x_{p_2}^{-1}\alpha\beta\gamma\\
\nonumber
&&
+Y(x_{p_1}\beta\gamma\eta_2\lambda_1\eta_1\alpha\Lambda+x_{p_3}\lambda_3\eta_3\lambda_2\Omega)\\
\nonumber
&&
+
Z(x_{p_1} \beta\gamma\eta_2\lambda_1\eta_1\alpha\Lambda\lambda_3\eta_3\lambda_2\Omega)+S'(\tau),
\end{eqnarray}
where
\begin{enumerate}
\item $p_1$ is the puncture the loop $i_1$ is based at, $p_2$ is the puncture inside the self-folded triangle formed by $i=i_2$ and $j=i_3$, and $p_3$ is the puncture the loop $i_7$ is based at;
\item if $p_1\neq p_3$, then $\Lambda$ and $\Omega$ are positive-length cycles on $\qtau$ respectively starting at $i_1$ and $i_7$;
\item if $p_1\neq p_3$, then $Y=1\in K$ and $Z=0\in K$;
\item if $p_1=p_3$, then $\Lambda$ and $\Omega$ are paths on $\qtau$, possibly of length 0, and respectively going from $i_7$ to $i_1$ and from $i_1$ to $i_7$;
\item if $p_1= p_3$, then $Y=0\in K$ and $Z=1\in K$;
\item regardless of whether $p_1\neq p_3$ or $p_1=p_3$, none of the paths $\Lambda$ and $\Omega$ involves any of the arrows $\alpha$, $\beta$, $\gamma$, $\delta$, $\varepsilon$, $\eta_1$, $\eta_2$, $\eta_3$, $\lambda_1$, $\lambda_2$, $\lambda_3$.
\item $S'(\tau)\in\RA{\qtau}$ is a potential not involving any of the arrows $\alpha$, $\beta$, $\gamma$, $\delta$, $\varepsilon$, $\eta_1$, $\eta_2$, $\eta_3$, $\lambda_1$, $\lambda_2$, $\lambda_3$.
\end{enumerate}
\end{enumerate}
\end{lemma}

\begin{proof} Consider the triangulated once-punctured cylinder $(\Sigma_0,\marked_0)$ depicted in Figure \ref{Fig:pop_basic_cylinder}.
        \begin{figure}[!h]
                \caption{When the empty-boundary surface $\surf$ is not a sphere with less than 7 punctures, the triangulation of the once-punctured cylinder depicted in this figure can be completed to a triangulation $\tau$ of $\surf$ satisfying the conclusion of Lemma \ref{lemma:existence-certain-triang}.}\label{Fig:pop_basic_cylinder}
                \centering
                \includegraphics[scale=.75]{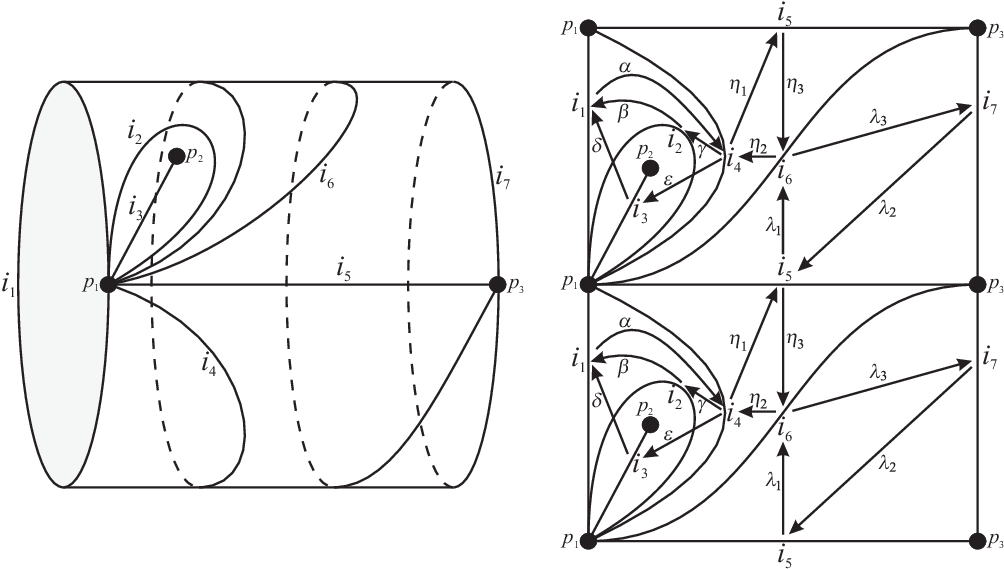}
        \end{figure}
To obtain from it a triangulation of an $n$-punctured sphere with $n\geq 7$, let $m_1$ and $m_2$ be positive integers such that $m_1+m_2=n-3$, $m_1\geq 2$ and $m_2\geq 2$, take a triangulated $m_1$-punctured monogon $(\Sigma_1,\marked_1)$ and a triangulated $m_2$-punctured monogon $(\Sigma_2,\marked_2)$, and glue $(\Sigma_1,\marked_1)$ to $(\Sigma_0,\marked_0)$ and $(\Sigma_2,\marked_2)$ to $(\Sigma_0,\marked_0)$ along boundary components, making sure that the marked point of $(\Sigma_1,\marked_1)$ lying on the boundary of $\Sigma_1$ is glued to $p_1$ and the marked point of $(\Sigma_2,\marked_2)$ lying on the boundary of $\Sigma_2$ is glued to $p_2$. The result is a triangulated $n$-punctured sphere. This triangulation $\tau$ and the arcs $i_1$, $i_2$, $i_3$, $i_4$, $i_5$, $i_6$ and $i_7$ signaled in Figure \ref{Fig:pop_basic_cylinder} satisfy the conclusion of Lemma \ref{lemma:existence-certain-triang}.

To obtain from $(\Sigma_0,\marked_0)$ a triangulation of a positive-genus empty-boundary surface with $n\geq 2$ punctures we can do the following:
\begin{itemize}
\item if $n=2$, glue $i_1$ and $i_7$ (along suitable orientations of these curves --to ensure we end up with an oriented surface) and obtain a triangulated twice-punctured torus;
\item if $n>2$, glue $p_1$ and $p_3$, without gluing any other point in $i_1$ to any point in $i_7$, and to the result of the gluing of $p_1$ and $p_3$ glue a triangulated $(n-2)$-punctured digon along $i_1$ and $i_7$ (my means of suitable orientations of these curves --to guarantee that we end up with an oriented surface), making sure that both marked points lying on the boundary of the digon are glued to $p_1=p_3$. This way we obtain a triangulated $n$-punctured torus with $n>2$;
\item if $n=2$, glue $p_1$ and $p_3$, without gluing any other point in $i_1$ to any point in $i_7$, and to the result of the gluing of $p_1$ and $p_3$ glue an unpunctured digon along $i_1$ and $i_7$ (by means of suitable orientations of these curves --to guarantee that we end up with an oriented surface), making sure that both marked points lying on the boundary of the digon are glued to $p_1=p_3$. Then cut off an open disc inside the already glued digon, thus obtaining a surface $(\Sigma_0',\{p_1,p_2\})$ which is a torus with one boundary component, with two punctures, but with no marked points lying on the boundary component. Let $\Sigma_1$ be a positive-genus surface with exactly one boundary component and without any marked points. Glue $\Sigma_1$ to $\Sigma_0'$ along boundary components, and complete $\{i_1,i_2,i_3,i_4,i_5,i_6,i_7\}$, which is a set of pairwise compatible arcs on $(\Sigma_0'\#\Sigma_1,\{p_1,p_2\})$, to a triangulation of $(\Sigma_0'\#\Sigma_1,\{p_1,p_2\})$. This way we obtain a triangulated $n$-punctured surface of genus $g\geq 2$, with empty boundary and with $n=2$.
\item if $n>2$, take a triangulated surface $(\Sigma_1,\marked_1)$ of positive genus, with exactly two boundary components, exactly one marked point on each of these two components, and exactly $n-3$ punctures, and glue it to $(\Sigma_0,\marked_0)$ along suitable orientations of boundary components (the orientations have to be chosen in such a way as to guarantee that we end up with an oriented surface), making sure that the two marked points on the boundary of $\Sigma_1$ are respectively glued to the marked points $p_1$ and $p_3$. This way we obtain a triangulated $n$-punctured surface of genus $g\geq 2$, with empty boundary and with $n> 2$.
\end{itemize}
In any case, the triangulation $\tau$ obtained after the corresponding gluing process, and the arcs $i_1$, $i_2$, $i_3$, $i_4$, $i_5$, $i_6$ and $i_7$ contained in the original cylinder $(\Sigma_0,\marked_0)$, are easily seen to satisfy the desired conclusion of Lemma \ref{lemma:existence-certain-triang}.
\end{proof}

\begin{prop}\label{prop:specific-pop-arbitrary-genus} Let $\surf$ be a surface with empty boundary. Suppose that $|\marked|\geq 2$ and that $\surf$ is different from a sphere with less than 7 punctures. Let $\tau$ be a triangulation of $\surf$ as in the conclusion of Lemma \ref{lemma:existence-certain-triang}. With the notation used in the alluded conclusion, the QPs $\qstau$ and $\qwtau$ are right equivalent.
\end{prop}

The proof of Proposition \ref{prop:specific-pop-arbitrary-genus} will follow from the next two lemmas.

\begin{lemma}\label{lemma:S-is-W+trash-arbitrary-genus} For the ideal triangulation $\tau$ above, the QP $\qstau$ is right-equivalent to a QP of the form $(\qtau,\wtau+L_1)$, where $L_1\in \RA{\qtau}$
is a potential satisfying the following condition:
\begin{eqnarray}
\label{eq:forced-factors-arbitrary-genus}
&&\text{it can be written as $\eta_2\lambda_1\lambda_2\Theta_1+\eta_2\eta_3\lambda_2\Theta_2$ for some $\Theta_1,\Theta_2\in\maxid^4\subseteq\RA{\qtau}$.}
\end{eqnarray}
\end{lemma}

\begin{proof} Write
$\Pi=Y\lambda_1\eta_1\alpha\Lambda\Lambda\delta\varepsilon+Z\lambda_1\eta_1\alpha\Lambda\lambda_3\eta_3\lambda_2\Omega\Lambda\lambda_3\eta_3\lambda_2\Omega\delta\varepsilon$ and
define $R$-algebra automorphisms
$\psi_{\beta,\delta}$, $\psi_\alpha$, $\psi_{\eta_3}:
\RA{\qtau}\rightarrow \RA{\qtau}$
according to the rules
\begin{eqnarray}\nonumber
\psi_{\beta,\delta} & : &  \beta\mapsto x_{p_2}^{-1}\beta, \ \delta\mapsto-x_{p_2}\delta,\\
\nonumber
\psi_\alpha & : &  \alpha\mapsto\alpha+Yx_{p_1}x_{p_2}\eta_2\lambda_1\eta_1\alpha\Lambda+Zx_{p_1}x_{p_2}\eta_2\lambda_1\eta_1\alpha\Lambda\lambda_3\eta_3\lambda_2\Omega\\
\nonumber
\psi_{\eta_3} & : & \eta_3\mapsto\eta_3+x_{p_1}^2x_{p_2}^2\Pi\eta_2\lambda_1.
\end{eqnarray}
Direct calculation shows that
\begin{eqnarray}\nonumber
\psi_{\eta_3}\psi_\alpha\psi_{\beta,\delta}(\stau) &\sim_{\operatorname{cyc}}&
\wtau\\
\nonumber
&&
+Yx_{p_1}^2x_{p_2}^2x_{p_3}\lambda_3\Pi\eta_2\lambda_1\lambda_2\Omega\\
\nonumber
&&
+Z\left(x_{p_1}^3x_{p_2}^2\eta_2\lambda_1\lambda_2\Omega\beta\gamma\eta_2\lambda_1\eta_1\alpha\Lambda\lambda_3\Pi\right.\\
\nonumber
&&
-x_{p_1}^4x_{p_2}^4\eta_2\lambda_1\lambda_2\Omega\Lambda\lambda_3\eta_3\lambda_2\Omega\delta\varepsilon\eta_2\lambda_1\eta_1\eta_2\lambda_1\eta_1\alpha\Lambda\lambda_3\Pi\\
\nonumber
&&
-x_{p_1}^4x_{p_2}^4\eta_2\lambda_1\lambda_2\Omega\delta\varepsilon\eta_2\lambda_1\eta_1\eta_2\lambda_1\eta_1\alpha\Lambda\lambda_3\eta_3\lambda_2\Omega\Lambda\lambda_3\Pi\\
\nonumber
&&
\left.-x_{p_1}^6x_{p_2}^6\eta_2\lambda_1\lambda_2\Omega\delta\varepsilon\eta_2\lambda_1\eta_1\eta_2\lambda_1\eta_1\alpha\Lambda\lambda_3\Pi\eta_2\lambda_1\lambda_2\Omega\Lambda\lambda_3\Pi\right).
\end{eqnarray}
This proves the Lemma.
\end{proof}

For the next lemma, we introduce a couple of pieces of notation and terminology. Given a non-zero element $u$ of a complete path algebra $\RA{ Q}$, we denote by $\short(u)$ the largest integer $n$ with the property that $u\in\maxid^{n}$ but $u\notin\maxid^{n+1}$. Also, $\length(c)$ will denote the length of any given path $c$ on $Q$. Finally, whenever we say that a cycle $\xi$ \emph{appears} (resp. \emph{does not appear}) in a potential $L$, we mean that in the expansion of $L$ as a non-redundant linear combination of cycles, $\xi$ appears with non-zero coefficient (resp. with coefficient zero).

\begin{lemma}\label{lemma:making-tails-longer-arbitrary-genus} Let $\tau$ be the ideal triangulation from the conclusion of Lemma \ref{lemma:existence-certain-triang}. Suppose $L\in \RA{\qtau}$ is a non-zero potential satisfying
\eqref{eq:forced-factors-arbitrary-genus}.
Then there exist a potential $L'\in \RA{\qtau}$ and a right-equivalence
$\varphi:(\qtau,\wtau+L)\rightarrow(\qtau,\wtau+L')$, such that:
\begin{itemize}
\item $L'$ satisfies
\eqref{eq:forced-factors-arbitrary-genus};
\item $\short(L')>\short(L)$;
\item $\varphi$ is a unitriangular automorphism of $\RA{\qtau}$, and $\depth(\varphi)=\short(L)-3$.
\end{itemize}
\end{lemma}

\begin{proof} Since $L$ satisfies \eqref{eq:forced-factors-arbitrary-genus}, there exist elements $\Theta_1,\Theta_2\in\maxid^4\subseteq\RA{\qtau}$ such that $L=\eta_2\lambda_1\lambda_2\Theta_1+\eta_2\eta_3\lambda_2\Theta_2$.
Define an $R$-algebra automorphism $\varphi$ of $\RA{\qtau}$ according to the rule
$$
\varphi:\lambda_3\mapsto\lambda_3-\Theta_1\eta_2, \ \ \ \eta_1\mapsto\eta_1-\lambda_2\Theta_2.
$$
It is clear that $\varphi$ is unitriangular and its depth is equal to $\short(L)-3$.

Let us write
\begin{eqnarray}\nonumber
L_1' &=& \varphi\left(Y(x_{p_1}\beta\gamma\eta_2\lambda_1\eta_1\alpha\Lambda+x_{p_3}\lambda_3\eta_3\lambda_2\Omega)
+Z(x_{p_1}\beta\gamma\eta_2\lambda_1\eta_1\alpha\Lambda\lambda_3\eta_3\lambda_2\Omega)\right)\\
\nonumber
&&
-\left(Y(x_{p_1}\beta\gamma\eta_2\lambda_1\eta_1\alpha\Lambda+x_{p_3}\lambda_3\eta_3\lambda_2\Omega)+
Z(x_{p_1}\beta\gamma\eta_2\lambda_1\eta_1\alpha\Lambda\lambda_3\eta_3\lambda_2\Omega)\right),\\
\nonumber
L_2' &=& \eta_2\lambda_1\lambda_2(\varphi(\Theta_1)-\Theta_1)
+\eta_2\eta_3\lambda_2(\varphi(\Theta_2)-\Theta_2).
\end{eqnarray}
While it is obvious that $L_2'$ is a potential satisfying \eqref{eq:forced-factors-arbitrary-genus}, direct computation shows that $L_1'$ satisfies \eqref{eq:forced-factors-arbitrary-genus} too. Hence, setting $L'=L_1'+L_2'$, we see that $L'$ satisfies \eqref{eq:forced-factors-arbitrary-genus}.
Furthermore, another direct computation shows that
\begin{eqnarray}
\nonumber
\varphi(\wtau+L)
&\sim_{\operatorname{cyc}}&
\wtau+L'.
\end{eqnarray}

Write $M=Y(x_{p_1}\beta\gamma\eta_2\lambda_1\eta_1\alpha\Lambda+x_{p_3}\lambda_3\eta_3\lambda_2\Omega)+
Z(x_{p_1}\beta\gamma\eta_2\lambda_1\eta_1\alpha\Lambda\lambda_3\eta_3\lambda_2\Omega)$. Then $\short(M)>3$ and
$L_1'=\varphi(M)-M$. By Lemma \ref{lemma:depth-is-short-cycle-friendly}, we have $L_1'\in\maxid^{\short(M)+\depth(\varphi)}=\maxid^{\short(M)+\short(L)-3}$, which means that $\short(L_1')\geq\short(M)+\short(L)-3>\short(L)$. Next, note that $L_2'=\varphi(N)-N$, and hence, by Lemma \ref{lemma:depth-is-short-cycle-friendly}, we have $L_2'\in\maxid^{\short(N)+\depth(\varphi)}=\maxid^{\short(L)+\short(L)-3}$, which means that $\short(L_2')\geq2\short(L)-3>\short(L)$ (the last strict inequality is due to the fact that $\short(L)>3$). Therefore, $\short(L')=\short(L_1'+L_2')>\short(L)$.

Lemma \ref{lemma:making-tails-longer-arbitrary-genus} is proved.
\end{proof}

\begin{proof}[Proof of Proposition \ref{prop:specific-pop-arbitrary-genus}] By Lemma \ref{lemma:well-defined-limit-automorphism}, it suffices to show the existence of a sequence $(L_n)_{n> 0}$ of potentials on $\qtau$ and a sequence $(\varphi_n)_{n\geq 0}$ of $R$-algebra automorphisms of $\RA{\qtau}$ such that:
\begin{eqnarray}
\label{eq:limLn=0}
&&
\text{$\lim_{n\to\infty}L_n=0$;}\\
\label{eq:lim-depth-of-phi=infty}
&&
\text{$\varphi_n$ is unitriangular for all $n>0$, and $\lim_{n\to\infty}\depth(\varphi_n)=\infty$;}\\
\label{eq:phi0-right-equiv}
&&
\text{$\varphi_0$ is a right-equivalence $(\qtau,\stau)\rightarrow(\qtau,\wtau+L_1)$;}\\
\label{eq:phin-right-equiv}
&&
\text{$\varphi_n$ is a right-equivalence $(\qtau,\wtau+L_n)\rightarrow(\qtau,\wtau+L_{n+1})$}\\
\nonumber
&& \text{for all $n>0$.}
\end{eqnarray}
Indeed, Lemma \ref{lemma:well-defined-limit-automorphism} guarantees that, if such sequences exist, then the limit $\lim_{n\to\infty}\varphi_n\varphi_{n-1}\ldots\varphi_2\varphi_1\varphi_0$ is a right-equivalence $\qstau\rightarrow\qwtau$.

By Lemma \ref{lemma:S-is-W+trash-arbitrary-genus}, there exists a right-equivalence $\varphi_0:\qstau\rightarrow(\qtau,\wtau+L_1)$, where $L_1\in \RA{\qtau}$
is a potential satisfying \eqref{eq:forced-factors-arbitrary-genus}.

Now, given a potential $L_n\in \RA{\qtau}$ satisfying \eqref{eq:forced-factors-arbitrary-genus}, Lemma \ref{lemma:making-tails-longer-arbitrary-genus} guarantees the existence of a potential $L_{n+1}\in \RA{\qtau}$ and a right-equivalence $\varphi_{n}:(\qtau,\wtau+L_n)\rightarrow (\qtau,\wtau+L_{n+1})$, such that:
\begin{itemize}
\item $L_{n+1}$ satisfies
\eqref{eq:forced-factors-arbitrary-genus};
\item $\short(L_{n+1})>\short(L_n)$;
\item $\varphi_n$ is unitriangular and $\depth(\varphi_n)=\short(L_{n})-3$.
\end{itemize}

Consider the sequences $(L_n)_{n> 0}$ and $(\varphi_n)_{n\geq 0}$ thus constructed. We obviously have $\lim_{n\to\infty}L_n=0$ and $\lim_{n\to\infty}\depth(\varphi_n)=\infty$. Therefore, the sequences $(L_n)_{n> 0}$ and $(\varphi_n)_{n\geq 0}$ satisfy all desired properties \eqref{eq:limLn=0}, \eqref{eq:lim-depth-of-phi=infty}, \eqref{eq:phi0-right-equiv} and \eqref{eq:phin-right-equiv}.

Proposition \ref{prop:specific-pop-arbitrary-genus} is proved.
\end{proof}

The proof of the Popping Theorem is now complete in the case of all empty-boundary surfaces different from a sphere with less than 7 punctures. As pointed out at the beginning of this subsection, the sphere with exactly 6 punctures is treated in the fourth arXiv version of this paper.

\subsection{Proof for surfaces with non-empty boundary}\label{subsec:proof-Pop-Thm-nonempty-boundary}

For surfaces with non-empty boundary the Popping Theorem can be proved in at least two different ways. In this subsection we give a proof via \emph{restriction}, using the fact that by now we already know that the Popping theorem holds for all empty-boundary surfaces satisfying \eqref{eq:not-sphere-with-few-punctures}. A direct proof can be given, without going through empty-boundary considerations, by ``\emph{pulling undesired terms to the boundary}".

Let us recall the definition of restriction and a couple of results from \cite{Labardini1} and \cite{Labardini2}.

\begin{defi}[{\cite[Definition 8.8]{DWZ1}}]\label{def:restriction} Let $(Q,S)$ be any QP (not necessarily arising from a surface) and $I$ be a subset of the vertex set $Q_0$. The \emph{restriction} of $(Q,S)$ to $I$ is the QP $(Q|_I,S|_I)$, where
\begin{itemize}\item $Q|_I$ is the quiver obtained from $Q$ by deleting the arrows incident to elements from $Q_0\setminus I$, the vertex set of $Q|_I$ is $Q_0$;
\item $S|_I$ is the image of $S$ under the $R$-algebra homomorphism $\psi_I:\RA{ Q}\rightarrow \RA{ Q|_I}$ defined by the rule $\psi_I(a)=a$ for every arrow $a$ of $Q$ whose head and and tail simultaneously belong to $I$, and $\psi_I(b)=0$ for every arrow $b$ incident to at least one element of $Q_0\setminus I$.
\end{itemize}
\end{defi}

\begin{lemma}[{\cite[Proof of Lemma 2.30]{Labardini2}}]\label{lemma:res-re=re-res} Suppose $\varphi:(Q,S)\rightarrow(Q',S')$ is a right-equivalence (with the QPs $(Q,S)$ and $(Q',S')$ not necessarily arising from a surface). If $I$ is a subset of $Q_0$, then the $R$-algebra homomorphism $\varphi|_{I}:\RA{ Q|_I}\rightarrow \RA{ Q'|_I}$ defined by the rule $u\mapsto\varphi(u)|_I$ is a right-equivalence between $(Q|_I,S|_I)$ and $(Q'|_I,S'|_I)$.
\end{lemma}

The following two lemmas will be proved simultaneously in the proof of Proposition \ref{prop:pop-is-re-non-empty-boundary}.

\begin{lemma}[{\cite[Lemma 29]{Labardini1}}]\label{lemma:all-are-restrictions} For every QP of the form $\qstau$ with $\tau$ an ideal triangulation there exists an ideal triangulation $\widetilde{\tau}$ of a surface $(\widetilde{\Sigma},\widetilde{\marked})$ with empty boundary, with the following properties:
\begin{itemize}\item $\widetilde{\tau}$ contains all the arcs of $\tau$;
\item $\qstau$ can be obtained from the restriction of $(Q(\widetilde{\tau}),S(\widetilde{\tau},\widetilde{\mathbf{x}}))$ to $\tau$ by deleting the elements of $\widetilde{\tau}\setminus\tau$, where $\widetilde{\mathbf{x}}=(x_q)_{q\in\widetilde{\marked}}$ is any extension of the choice $\mathbf{x}=(x_p)_{p\in\punct}$ to $\widetilde{\marked}$.
\end{itemize}
\end{lemma}

\begin{lemma}\label{lemma:popped-pots-are-also-restrictions} For every QP of the form $\qwtau$ with $\tau$ an ideal triangulation and $\wtau$ the popped potential associated to a quadruple $(\tau,\mathbf{x},i,j)$, there exists an ideal triangulation $\widetilde{\tau}$ of a surface $(\widetilde{\Sigma},\widetilde{\marked})$ with empty boundary with the following properties:
\begin{itemize}\item $\widetilde{\tau}$ contains all the arcs of $\tau$;
\item $\qwtau$ can be obtained from the restriction of $(Q(\widetilde{\tau}),W(\widetilde{\tau},\widetilde{\mathbf{x}}))$ to $\tau$ by deleting the elements of $\widetilde{\tau}\setminus\tau$, where $\widetilde{\mathbf{x}}=(x_q)_{q\in\widetilde{\marked}}$ is any extension of the choice $\mathbf{x}=(x_p)_{p\in\punct}$ to $\widetilde{\marked}$ and $W(\widetilde{\tau},\widetilde{\mathbf{x}})$ is the popped potential associated to the quadruple $(\widetilde{\tau},\widetilde{\mathbf{x}},i,j)$.
\end{itemize}
\end{lemma}

\begin{prop}\label{prop:pop-is-re-non-empty-boundary} Proposition \ref{prop:existence-of-popping-triangulation} holds for surfaces with non-empty boundary.
\end{prop}

\begin{proof} Let $i$, $j$ be as in the statement of Proposition \ref{prop:existence-of-popping-triangulation}, and let $\tau$ be any ideal triangulation of $\surf$ containing $i$ and $j$.
For each boundary component $b$ of $\surfnoM$, let $m_b$ be the number of marked points lying on $b$. Since each $b$ is homeomorphic to a circle, we can glue $\surfnoM$ and a triangulated $5$-punctured $m_b$-gon  along $b$. Making such gluing along every boundary component of $\surfnoM$ will result in a surface $(\widetilde{\Sigma},\widetilde{\marked})$ with empty empty boundary and more than 5 punctures, and an ideal triangulation $\widetilde{\tau}$ containing $\tau$. In particular, $i,j\in\widetilde{\tau}$. It is easy to check that the ideal triangulation $\widetilde{\tau}$ satisfies the properties stated in Lemmas \ref{lemma:all-are-restrictions} and \ref{lemma:popped-pots-are-also-restrictions}.

Since Theorem \ref{thm:popping-is-right-equiv} is already known to hold in the empty boundary case (cf. Subsection \ref{subsec:proof-Pop-Thm-empty-boundary}), the pop of $(i,j)$ in $\widetilde{\tau}$ induces right-equivalence, that is, there exists a right-equivalence between $(Q(\widetilde{\tau},S(\widetilde{\tau},\widetilde{x}))$ and $(Q(\widetilde{\tau},W(\widetilde{\tau},\widetilde{x}))$. Since $\widetilde{\tau}$ simultaneously proves Lemmas \ref{lemma:all-are-restrictions} and \ref{lemma:popped-pots-are-also-restrictions}, a straightforward combination of this fact with Lemma \ref{lemma:res-re=re-res} finishes the proof of Proposition~\ref{prop:pop-is-re-non-empty-boundary}.
\end{proof}

\begin{ex} In Figure \ref{Fig:genus_3_with_boundary} we have sketched the proof of Proposition \ref{prop:pop-is-re-non-empty-boundary}. $\hfill{\blacktriangle}$
        \begin{figure}[!h]
                \caption{Sketch of proof of the Popping Theorem for $\partial\Sigma\neq\varnothing$}
                \label{Fig:genus_3_with_boundary}
                \centering
                \includegraphics[scale=.4]{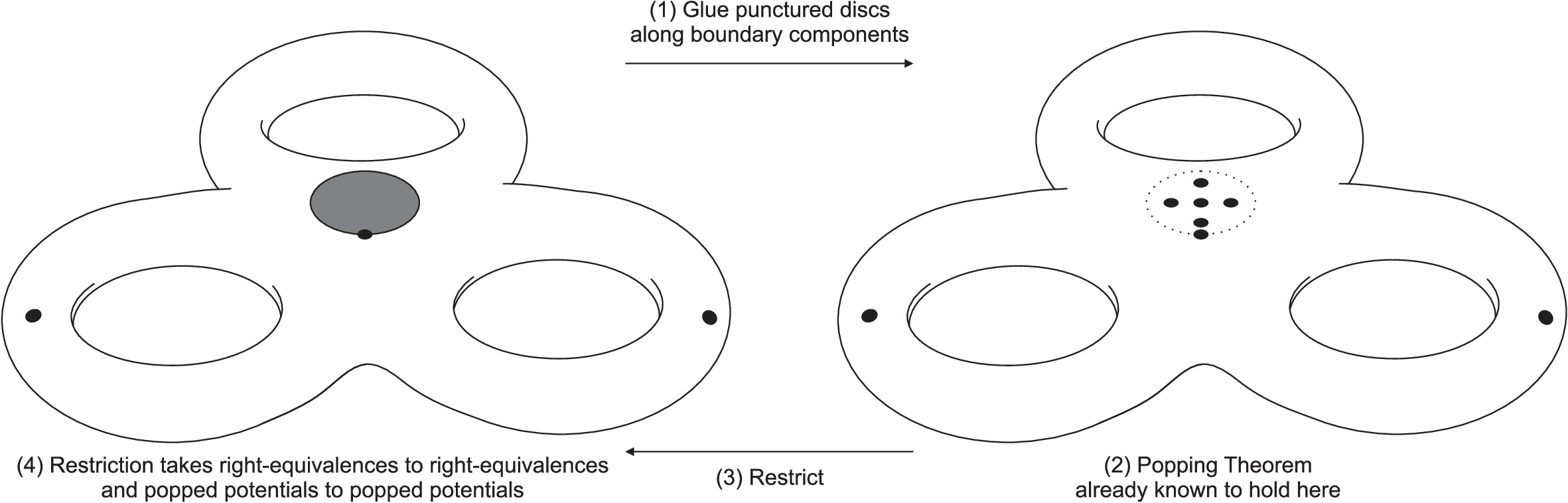}
        \end{figure}
\end{ex}

\section{Leaving the positive stratum}\label{sec:leaving-pos-stratum}

Let $\surf$ be any surface different from a 5-punctured sphere. Thus, $\surf$ is allowed to be any of the surfaces listed in \eqref{eq:no-boundary}, \eqref{eq:positive-genus} and \eqref{eq:spheres}.

Ideal triangulations are precisely the tagged triangulations that have non-negative signature, and these are precisely the tagged triangulations that have positive weak signature. Thus, in any tagged triangulation with non-negative signature we have a well-defined notion of \emph{folded side}. It is the flips of folded sides the ones that have proven to be extremely difficult to deal with, at least in regard to QP-mutations.

\begin{prop}\label{prop:flip-mut-folded-sides} Let $\tau$ be an ideal triangulation of $\surf$. Suppose that $i$ and $j$ are arcs in $\tau$ forming a self-folded triangle, with $i$ as folded side and $j$ as enclosing loop. Let $\sigma$ be the tagged triangulation obtained from $\tagfunction_{\mathbf{1}}(\tau)$ by flipping the tagged arc $\tagfunction_{\mathbf{1}}(i)$, and let $\ell$ be the unique tagged arc on $\surf$ such that $\sigma=(\tagfunction_{\mathbf{1}}(\tau)\setminus\{\tagfunction_{\mathbf{1}}(i)\})\cup\{\ell\}$. The QPs $(Q(\tagfunction_{\mathbf{1}}(\tau)),S(\tagfunction_{\mathbf{1}}(\tau),\mathbf{x}))$ and $\mu_{\ell}\qssigma$ are right-equivalent.
\end{prop}

\begin{proof} By Theorem \ref{thm:popping-is-right-equiv}, the QPs $\qstau$ and $\qwtau$ are right-equivalent. This, together with Definition \ref{def:QP-of-tagged-triangulation} and the very definition of the $K$-algebra isomorphism $\tagfunction_\epsilon:\RA{Q(\tau^\circ)}\rightarrow\RA{Q(\tau)}$ described right after Definition \ref{def:obvious-cycles}, imply that $(Q(\tagfunction_{\mathbf{1}}(\tau)),S(\tagfunction_{\mathbf{1}}(\tau),\mathbf{x}))$ and $(Q(\tagfunction_{\mathbf{1}}(\tau)),\tagfunction_{\mathbf{1}}(\wtau))$ are right-equivalent.  We shall show that $(Q(\tagfunction_{\mathbf{1}}(\tau)),\tagfunction_{\mathbf{1}}(\wtau))$ is right-equivalent to $\mu_{\ell}\qssigma$.

Let $\triangle$ be the self-folded triangle of $\tau$ containing $i$ and $j$, and let $\triangle'$ be the unique triangle of $\tau$ which is not self-folded and contains $j$. Let $k$ and $k'$ be the sides of $\triangle'$ that are different from $j$; see Figure \ref{Fig:pos_stratum_abandoned_notation}, whose notation we shall adopt for the rest of the proof.
        \begin{figure}[!h]
                \caption{}\label{Fig:pos_stratum_abandoned_notation}
                \centering
                \includegraphics[scale=.5]{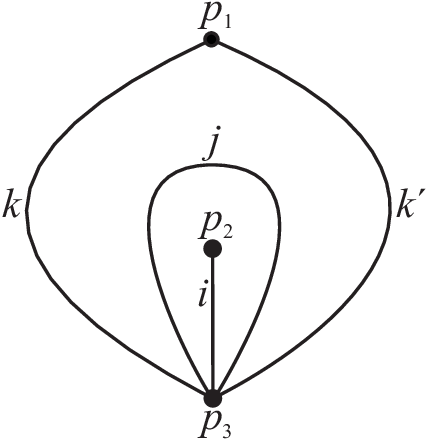}
        \end{figure}
Then we have three possibilities:
\begin{enumerate}
\item Either $k,k'\in\tau$, and none of $k$ and $k'$ is a loop enclosing a self-folded triangle of $\tau$; or
\item $k,k'\in\tau$, $k$ is a loop enclosing a self-folded triangle of $\tau$ and $k'$ is not; or
\item $k,k'\in\tau$, $k'$ is a loop enclosing a self-folded triangle of $\tau$ and $k$ is not; or
\item at least one of $k$ and $k'$ is a boundary segment of $\surf$.
\end{enumerate}
Note that $k$ and $k'$ cannot both happen to be loops enclosing self-folded triangles, for otherwise $\surf$ would be a sphere with exactly 4 punctures. We will prove that $(Q(\tagfunction_{\mathbf{1}}(\tau)),\tagfunction_{\mathbf{1}}(\wtau))$ is right-equivalent to $\mu_{\ell}\qssigma$ in the three cases that we have just listed as possibilities.

\begin{case}
Suppose that $k,k'\in\tau$, and that none of $k$ and $k'$ is a loop enclosing a self-folded triangle of $\tau$. Then we have three possibilities:
\begin{itemize}
\item[(i)] Either the marked points $q_1$ and $q_3$ are different, and the marked point $q_1$ is not a puncture incident to exactly two arcs of $\tau$; or
\item[(ii)] the marked points $q_1$ and $q_3$ are different, and the marked point $q_1$ is a puncture incident to exactly two arcs of $\tau$; or
\item[(iii)] the marked points $q_1$ and $q_3$ coincide.
\end{itemize}

If the marked points $q_1$ and $q_3$ are different, and the marked point $q_1$ is not a puncture incident to exactly two arcs of $\tau$, then the configurations that $\sigma$ and $\tagfunction_{\mathbf{1}}(\tau)$ respectively present around $l$ and $\tagfunction_{\mathbf{1}}(i)$ are the ones
shown in Figure \ref{Fig:pos_stratum_abandoned_1},
        \begin{figure}[!h]
                \caption{}\label{Fig:pos_stratum_abandoned_1}
                \centering
                \includegraphics[scale=.8]{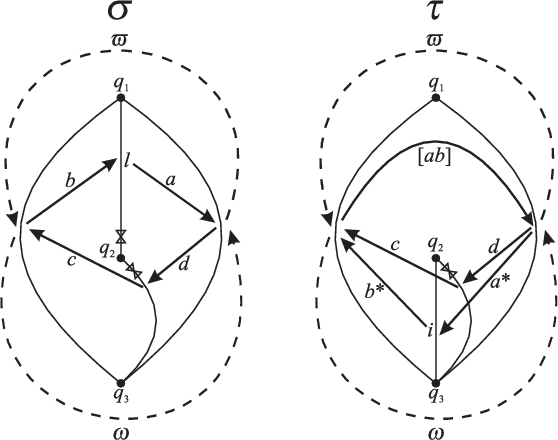}
        \end{figure}
and the potentials $\ssigma$ and $\tagfunction_{\mathbf{1}}(\wtau))$ are given by
\begin{eqnarray}\nonumber
\ssigma & = & x_{q_1}ab\varpi+x_{q_2}^{-1}abcd+x_{q_3}cd\omega+S'(\sigma) \ \ \ \text{and}\\
\nonumber
\tagfunction_{\mathbf{1}}(\wtau)) & = & a^*[ab]b^*+x_{q_1}[ab]\varpi+x_{q_2}^{-1}[ab]cd+x_{q_3}cd\omega+S'(\sigma),
\end{eqnarray}
where:
\begin{itemize}
\item[(a)] $\varpi$ is $\begin{cases}\text{a path of length greater than 1} & \text{if $q_1$ is a puncture};\\
0 & \text{if $q_1$ is not a puncture};\end{cases}$
\item[(b)] $\omega$ is $\begin{cases}\text{a path of positive length} & \text{if $q_3$ is a puncture};\\
0 & \text{if $q_3$ is not a puncture};\end{cases}$
\item[(c)] $S'(\sigma)\in\RA{\qsigma}$ is a potential not involving any of the arrows $a,b,c,d$.
\end{itemize}
Furthermore, $\mu_l(\qsigma)=\widetilde{\mu}_l(\qsigma)=Q(\tagfunction_{\mathbf{1}}(\tau))$ (with the vertex $\tagfunction_{\mathbf{1}}(i)\in Q(\tagfunction_{\mathbf{1}}(\tau))$ replacing the vertex $l\in\qsigma$) and
$$
\mu_l(\ssigma)=\widetilde{\mu}_l(\ssigma)=x_{q_1}[ab]\varpi+x_{q_2}^{-1}[ab]cd+x_{q_3}cd\omega+S'(\sigma)+a^*[ab]b^*.
$$
The QPs $\mu_l\qssigma$ and $(Q(\tagfunction_{\mathbf{1}}(\tau)),\tagfunction_{\mathbf{1}}(\wtau))$ are then obviously right-equivalent.

If the marked points $q_1$ and $q_3$ are different, and the marked point $q_1$ is a puncture incident to exactly two arcs of $\tau$, then the configurations that $\sigma$ and $\tau$ respectively present around $l$ and $i$ are the ones shown in Figure \ref{Fig:pos_stratum_abandoned_2},
        \begin{figure}[!h]
                \caption{}\label{Fig:pos_stratum_abandoned_2}
                \centering
                \includegraphics[scale=.8]{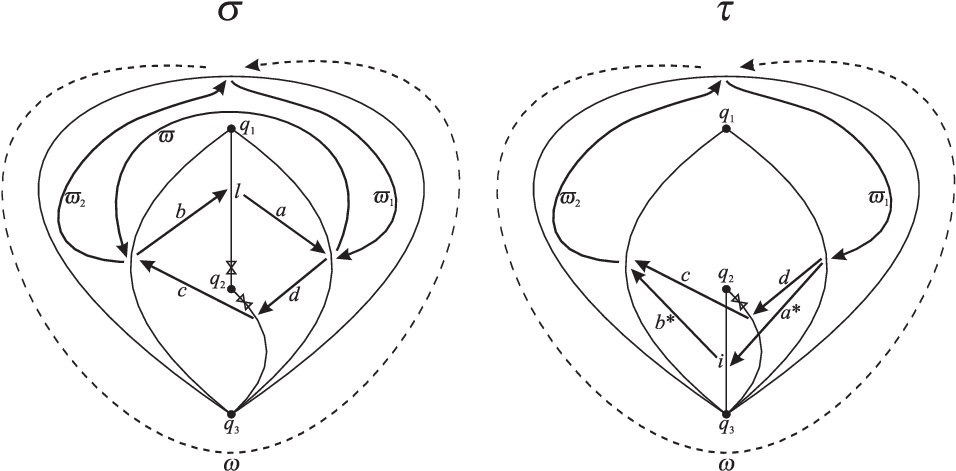}
        \end{figure}
and the potentials $\ssigma$, $\tagfunction_{\mathbf{1}}(\wtau))$ and $\mu_l(\ssigma)$ are given by
\begin{eqnarray}\nonumber
\ssigma & = & \varpi_1\varpi_2\varpi+x_{q_1}ab\varpi+x_{q_2}^{-1}abcd+x_{q_3}cd\varpi_1\omega\varpi_2+S'(\sigma),\\
\nonumber
\tagfunction_{\mathbf{1}}(\wtau)) & = &-x_{1}^{-1}\varpi_1\varpi_2b^*a^*-x_1^{-1}x_2^{-1}\varpi_1\varpi_2cd+x_3cd\varpi_1\omega\varpi_2+S'(\sigma) \ \ \
\text{and}\\
\nonumber
\mu_l(\ssigma) & = &-x_{1}^{-1}\varpi_1\varpi_2b^*a^*-x_1^{-1}x_2^{-1}\varpi_1\varpi_2cd+x_3cd\varpi_1\omega\varpi_2+S'(\sigma),
\end{eqnarray}
where:
\begin{itemize}
\item[(a)] $\varpi$, $\varpi_1$ and $\varpi_2$ are arrows of $Q(\sigma)$;
\item[(b)] $\omega$ is $\begin{cases}\text{a path of positive length} & \text{if $q_3$ is a puncture};\\
0 & \text{if $q_3$ is not a puncture};\end{cases}$
\item[(c)] $S'(\sigma)\in\RA{\qsigma}$ is a potential not involving any of the arrows $a,b,c,d,\omega,\varpi,\varpi_1,\varpi_2$.
\end{itemize}
The QPs $\mu_l\qssigma$ and $(Q(\tagfunction_{\mathbf{1}}(\tau)),\tagfunction_{\mathbf{1}}(\wtau))$ are then obviously right-equivalent.

If the marked points $q_1$ and $q_3$ coincide, then the configurations that $\sigma$ and $\tau$ respectively present around $l$ and $i$ are the ones shown in Figure \ref{Fig:pos_stratum_abandoned_3},
        \begin{figure}[!h]
                \caption{}\label{Fig:pos_stratum_abandoned_3}
                \centering
                \includegraphics[scale=.8]{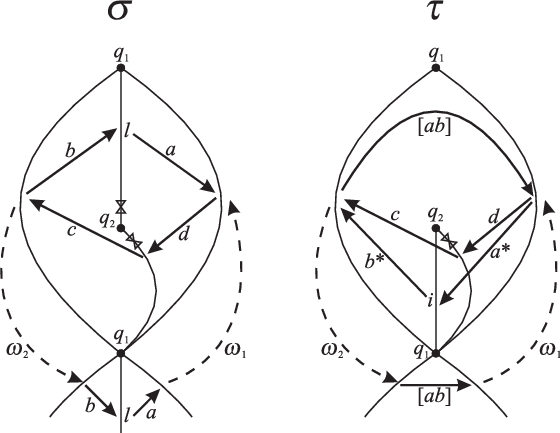}
        \end{figure}
and the potentials $\ssigma$ and $\tagfunction_{\mathbf{1}}(\wtau))$ are given by
\begin{eqnarray}\nonumber
\ssigma & = & x_{q_1}\omega_1ab\omega_2cd+x_{q_2}^{-1}abcd+S'(\sigma) \ \ \ \text{and}\\
\nonumber
\tagfunction_{\mathbf{1}}(\wtau) & = & a^*[ab]b^*+x_{q_1}\omega_1[ab]\omega_2cd+x_{q_2}^{-1}d[ab]c+S'(\sigma),
\end{eqnarray}
where:
\begin{itemize}
\item[(a)] $\omega_1$ and $\omega_2$ are $\begin{cases}\text{paths of length greater than 1} & \text{if $q_1$ is a puncture};\\
0 & \text{if $q_1$ is not a puncture};\end{cases}$
\item[(b)] $S'(\sigma)\in\RA{\qsigma}$ is a potential not involving any of the arrows $a,b,c,d$.
\end{itemize}
Furthermore, $\mu_l(\qsigma)=\widetilde{\mu}_l(\qsigma)=Q(\tagfunction_{\mathbf{1}}(\tau))$ (with the vertex $i\in Q(\tagfunction_{\mathbf{1}}(\tau))$ replacing the vertex $l\in\qsigma$) and
$$
\mu_l(\ssigma)=x_{q_1}\omega_1[ab]\omega_2cd+x_{q_2}^{-1}[ab]cd+S'(\sigma)+a^*[ab]b^*.
$$
The QPs $\mu_l\qssigma$ and $(Q(\tagfunction_{\mathbf{1}}(\tau)),\tagfunction_{\mathbf{1}}(\wtau))$ are then obviously right-equivalent.
\end{case}

\begin{case} Suppose that $k,k'\in\tau$, and that $k$ is a loop enclosing a self-folded triangle of $\tau$ and $k'$ is not. Then the configurations that $\sigma$ and $\tau$ respectively present around $l$ and $i$ are the ones shown in Figure \ref{Fig:pos_stratum_abandoned_5},
        \begin{figure}[!h]
                \caption{}\label{Fig:pos_stratum_abandoned_5}
                \centering
                \includegraphics[scale=.8]{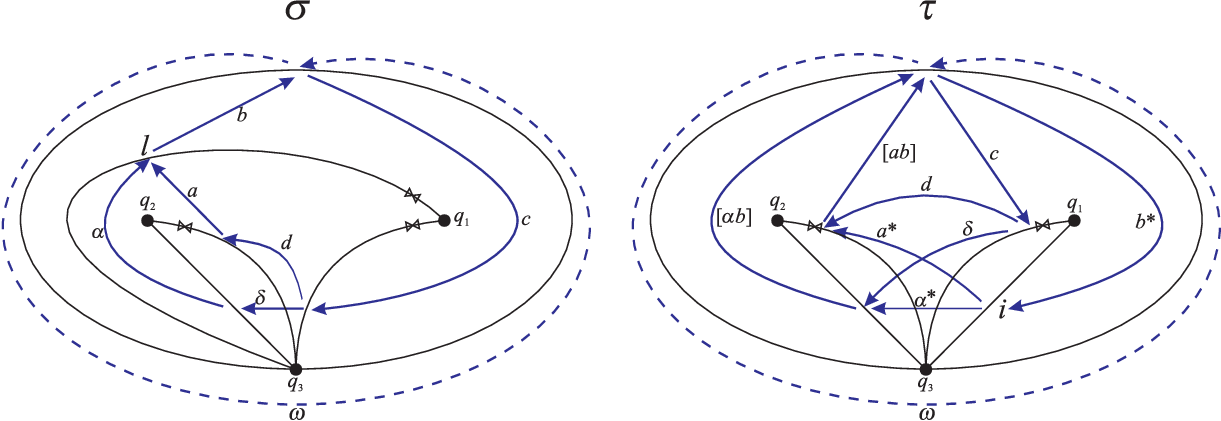}
        \end{figure}
and the potentials $\ssigma$ and $\tagfunction_{\mathbf{1}}(\wtau))$ are given by
\begin{eqnarray}\nonumber
\ssigma & = & x_{q_2}^{-1}dcba-x_{q_2}^{-1}x_{p}^{-1}\delta cb\alpha+x_{q_3}\delta c\omega b\alpha+S'(\sigma) \ \ \ \text{and}\\
\nonumber
\tagfunction_{\mathbf{1}}(\wtau)) & = & b^*[ba]a^*+x_{q_2}^{-1}c[ba]d-x_{p}b^*[b\alpha]\alpha^*-x_{q_2}^{-1}x_{p}^{-1}c[b\alpha]\delta+
x_{q_3}c\omega[b\alpha]\delta+S'(\sigma),
\end{eqnarray}
where:
\begin{itemize}
\item[(a)] $\omega$ is $\begin{cases}\text{a path of positive length} & \text{if $q_3$ is a puncture};\\
0 & \text{if $q_3$ is not a puncture};\end{cases}$
\item[(b)] $S'(\sigma)\in\RA{\qsigma}$ is a potential not involving any of the arrows $a,b,c,d,\alpha,\delta$.
\end{itemize}
Furthermore, $\mu_l(\qsigma)=\widetilde{\mu}_l(\qsigma)=Q(\tagfunction_{\mathbf{1}}(\tau))$ (with the vertex $i\in Q(\tagfunction_{\mathbf{1}}(\tau))$ replacing the vertex $l\in\qsigma$) and
$$
\widetilde{\mu}_l(\ssigma)=\mu_l(\ssigma)=x_{q_2}^{-1}dc[ba]-x_{q_2}^{-1}x_{p}^{-1}\delta c[b\alpha]
+x_{q_3}\delta c\omega[b\alpha]+S'(\sigma)+b^*[ba]a^*+b^*[b\alpha]\alpha^*.
$$
The $R$-algebra isomorphism $\varphi:\RA{\mu_l(Q(\sigma))}\rightarrow \RA{Q(\tagfunction_{\mathbf{1}}(\tau))}$ acting by $\alpha^*\mapsto-x_{p}\alpha^*$
and the identity on the rest of the arrows, is easily seen to be a right-equivalence $\varphi:\mu_l(\qsigma,\ssigma)\rightarrow(Q(\tagfunction_{\mathbf{1}}(\tau)),\tagfunction_{\mathbf{1}}(\wtau))$.
\end{case}

\begin{case} Suppose that $k,k'\in\tau$, and that $k'$ is a loop enclosing a self-folded triangle of $\tau$ and $k$ is not. Then the configurations that $\sigma$ and $\tau$ respectively present around $l$ and $i$ are the ones shown in Figure \ref{Fig:pos_stratum_abandoned_4},
        \begin{figure}[!h]
                \caption{}\label{Fig:pos_stratum_abandoned_4}
                \centering
                \includegraphics[scale=.8]{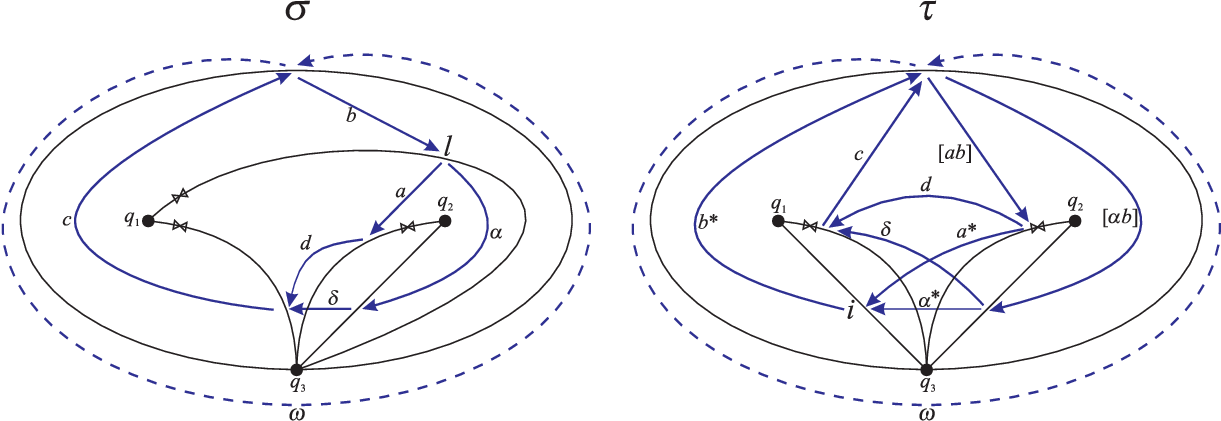}
        \end{figure}
and the potentials $\ssigma$ and $\tagfunction_{\mathbf{1}}(\wtau))$ are given by
\begin{eqnarray}\nonumber
\ssigma & = & x_{q_2}^{-1}abcd-x_{q_2}^{-1}x_{p}^{-1}\alpha bc\delta+x_{q_3}\alpha b\omega c\delta+S'(\sigma) \ \ \ \text{and}\\
\nonumber
\wtau & = & a^*[ab]b^*+x_{q_2}^{-1}d[ab]c-x_{p}\alpha^*[\alpha b]b^*-x_{q_2}^{-1}x_{p}^{-1}\delta[\alpha b]c+x_{q_3}\delta[\alpha b]\omega c+S'(\sigma),
\end{eqnarray}
where:
\begin{itemize}
\item[(a)] $\omega$ is $\begin{cases}\text{a path of positive length} & \text{if $q_3$ is a puncture};\\
0 & \text{if $q_3$ is not a puncture};\end{cases}$
\item[(b)] $S'(\sigma)\in\RA{\qsigma}$ is a potential not involving any of the arrows $a,b,c,d,\alpha,\delta$.
\end{itemize}
Furthermore, $\mu_l(\qsigma)=\widetilde{\mu}_l(\qsigma)=Q(\tagfunction_{\mathbf{1}}(\tau))$ (with the vertex $i\in Q(\tagfunction_{\mathbf{1}}(\tau))$ replacing the vertex $l\in\qsigma$) and
$$
\widetilde{\mu}_l(\ssigma)=\mu_l(\ssigma)=x_{q_2}^{-1}[ab]cd-x_{q_2}^{-1}x_{p}^{-1}[\alpha b]c\delta+x_{q_3}[\alpha b]\omega c\delta+S'(\sigma)+a^*[ab]b^*+\alpha^*[\alpha b]b^*.
$$
The $R$-algebra isomorphism $\varphi:\RA{\mu_l(Q(\sigma))}\rightarrow \RA{Q(\tagfunction_{\mathbf{1}}(\tau))}$ acting by $\alpha^*\mapsto-x_{p}\alpha^*$
and the identity on the rest of the arrows, is easily seen to be a right-equivalence $\varphi:\mu_l(\qsigma,\ssigma)\rightarrow(Q(\tagfunction_{\mathbf{1}}(\tau)),\tagfunction_{\mathbf{1}}(\wtau))$.
\end{case}

\begin{case} Suppose that at least one of $k$ and $k'$ is a boundary segment of $\surf$. Then $\ell$ is a sink or a source of the quiver $\qsigma$ and $\ssigma=S(\tagfunction_{\mathbf{1}}(\tau),\mathbf{x})=\tagfunction_{\mathbf{1}}(\wtau))$, whence it is obvious that $\mu_l(\qsigma,\ssigma)$ and $(Q(\tagfunction_{\mathbf{1}}(\tau)),\tagfunction_{\mathbf{1}}(\wtau))$ are right-equivalent.
\end{case}

We have thus proved that $\mu_l(\qsigma,\ssigma)$ and $(Q(\tagfunction_{\mathbf{1}}(\tau)),\tagfunction_{\mathbf{1}}(\wtau))$ are right-equivalent. Since $(Q(\tagfunction_{\mathbf{1}}(\tau)),\tagfunction_{\mathbf{1}}(\wtau))$ is right-equivalent to $(Q(\tagfunction_{\mathbf{1}}(\tau)),S(\tagfunction_{\mathbf{1}}(\tau),\mathbf{x}))$, Proposition \ref{prop:flip-mut-folded-sides} now follows from the fact that composition of right equivalences is right equivalence.
\end{proof}

\begin{thm}\label{thm:leaving-positive-stratum} Let $\tau$ be any ideal triangulation of $\surf$, and let $i$ be any arc belonging to $\tau$. Then $\mu_{\tagfunction_{\mathbf{1}}(i)}(Q(\tagfunction_{\mathbf{1}}(\tau)),S(\tagfunction_{\mathbf{1}}(\tau),\mathbf{x}))$ is right-equivalent to $\qssigma$, where $\sigma=f_{\tagfunction_{1}(i)}(\tagfunction_{\mathbf{1}}(\tau))$ is the tagged triangulation obtained from $\tagfunction_{\mathbf{1}}(\tau)$ by flipping the arc $\tagfunction_{\mathbf{1}}(i)$.
\end{thm}

\begin{proof} By Theorem \ref{thm:ideal-flips<->mutations} we can assume that $i$ is a folded side of $\tau$. Let $l$ be the unique tagged arc such that $\sigma=(\tagfunction_{\mathbf{1}}(\tau)\setminus\{\tagfunction_{\mathbf{1}}(i)\})\cup\{l\}$. Since QP-mutations are involutive up to right-equivalence, to prove Theorem \ref{thm:leaving-positive-stratum} it suffices to show that $\mu_l\qssigma$ and $(Q(\tagfunction_{\mathbf{1}}(\tau)),S(\tagfunction_{\mathbf{1}}(\tau),\mathbf{x}))$ are right-equivalent. But this is precisely what Proposition \ref{prop:flip-mut-folded-sides} asserts.
\end{proof}

\section{Flip $\leftrightarrow$ QP-mutation compatibility}\label{sec:flip<->mutations}

Let $\surf$ be any surface different from a 5-punctured sphere. Thus, $\surf$ is allowed to be any of the surfaces listed in \eqref{eq:no-boundary}, \eqref{eq:positive-genus} and \eqref{eq:spheres}.

\begin{thm}\label{thm:tagged-flips<->mutations} If $\tau$ and $\sigma$ are tagged triangulations of $\surf$ related by the flip of a tagged arc $\arc\in\tau$, then the QPs $\mu_\arc\qstau$ and $\qssigma$ are right-equivalent.
\end{thm}

\begin{proof} We have two possibilities: the weak signatures $\epsilon_\tau$ and $\epsilon_\sigma$ either are equal, or they differ at exactly one puncture.

\begin{mainthmcase}Suppose $\epsilon_\tau=\epsilon_\sigma$; then $\sigma^\circ=f_{\arc^\circ}(\tau^\circ)$, and if $j\in\sigma$ is the unique arc such that $\tau=f_j(\sigma)$, then $\tau^\circ=f_{j^\circ}(\sigma^\circ)$ (cf. Proposition \ref{prop:tagfunction-and-circ}, note that we also have $\tagfunction_{\epsilon_\tau}(i^\circ)=i$ and $\tagfunction_{\epsilon_\sigma}(j^\circ)=j$). By Theorem \ref{thm:ideal-flips<->mutations}, $\mu_{\arc^{\circ}}(Q(\tau^\circ),S(\tau^\circ,\epsilon_\tau\cdot\mathbf{x}))$ is right-equivalent to $(Q(\sigma^\circ),S(\sigma^\circ,\epsilon_\tau\cdot\mathbf{x}))=(Q(\sigma^\circ),S(\sigma^\circ,\epsilon_\sigma\cdot\mathbf{x}))$. This in particular means, since $Q(\sigma^\circ)$ is 2-acyclic, that we can write $\mu_{\arc^{\circ}}(Q(\tau^\circ),S(\tau^\circ,\epsilon_\tau\cdot\mathbf{x}))=(\mu_{\arc^\circ}(Q(\tau^\circ)),\mu_{\arc^{\circ}}(S(\tau^\circ,\epsilon_\tau\cdot\mathbf{x})))$. Thus, there exists an $R$-algebra isomorphism $\varphi:\RA{\mu_{\arc^\circ}(Q(\tau^\circ))}\rightarrow\RA{Q(\sigma^\circ)}$ such that $\varphi(\mu_{\arc^{\circ}}(S(\tau^\circ,\epsilon_\tau\cdot\mathbf{x})))\sim_{\operatorname{cyc}}S(\sigma^\circ,\epsilon_\sigma\cdot\mathbf{x})$.

By its very definition, the quiver $Q(\tau)$ is obtained from the quiver $Q(\tau^\circ)$ by replacing each vertex $\ell\in\tau^\circ$ with $\tagfunction_{\epsilon_\tau}(\ell)\in\tagfunction_{\epsilon_\tau}(\tau^\circ)=\tau$. Furthermore, the arrow set of $Q(\tau)$ is equal to the arrow set of $Q(\tau^\circ)$, and whenever $a$ is an arrow of $Q(\tau^\circ)$ going from $\ell_1\in\tau^\circ$ to $\ell_2\in\tau^\circ$, then $a$ itself is an arrow of $Q(\tau)$ going from $\tagfunction_{\epsilon_\tau}(\ell_1)\in\tau$ to $\tagfunction_{\epsilon_\tau}(\ell_2)\in\tau$. The quiver isomorphism $Q(\tau^\circ)\rightarrow Q(\tau)$ acting by $\ell\mapsto\tagfunction_{\epsilon_\tau}(\ell)$ on the vertices and as the identity on the arrows has been denoted by $\tagfunction_{\epsilon_\tau}$, and the $K$-algebra isomorphism $\RA{Q(\tau^\circ)}\rightarrow\RA{Q(\tau)}$ it induces has been denoted by $\tagfunction_{\epsilon_\tau}$ as well (see the paragraph right after Definition \ref{def:obvious-cycles}). By definition of $S(\tau,\mathbf{x})$ we have $S(\tau,\mathbf{x})=\tagfunction_{\epsilon_\tau}(S(\tau^\circ,\epsilon_\tau\cdot\mathbf{x}))\in \RA{Q(\tau)}$. Thus, since $\tagfunction_{\epsilon_\tau}(i^\circ)=i$, if we simultaneously apply the QP-mutation $\mu_{i^\circ}$ to $(Q(\tau^\circ),S(\tau^\circ,\epsilon_\tau\cdot\mathbf{x}))$ and the QP-mutation $\mu_i$ to $(Q(\tau),S(\tau,\mathbf{x}))$, we see that:
\begin{enumerate}
\item since the underlying quiver of $\mu_{i^\circ}(Q(\tau^\circ),S(\tau^\circ,\epsilon_\tau\cdot\mathbf{x}))$ is 2-acyclic (by Theorem \ref{thm:ideal-flips<->mutations}), the underlying quiver of $\mu_i(Q(\tau),S(\tau,\mathbf{x}))$ is 2-acyclic as well, and hence we can write $\mu_i(Q(\tau),S(\tau,\mathbf{x}))=(\mu_i(Q(\tau)),\mu_i(S(\tau,\mathbf{x})))$;
\item we can obtain $\mu_\arc(Q(\tau))$ from $\mu_{\arc^\circ}(Q(\tau^\circ))$ by replacing each vertex $\ell$ of $\mu_{\arc^\circ}(Q(\tau^\circ))$ with the vertex $\tagfunction_{\epsilon_\tau}(\ell)$ of $\mu_\arc(Q(\tau))$. The arrow sets of $\mu_\arc(Q(\tau))$ and $\mu_{\arc^\circ}(Q(\tau^\circ))$ then coincide, and whenever $a$ is an arrow of $\mu_{\arc^\circ}(Q(\tau^\circ))$ going from $\ell_1$ to $\ell_2$, then $a$ itself is an arrow of $\mu_\arc(Q(\tau))$ going from $\tagfunction_{\epsilon_\tau}(\ell_1)$ to $\tagfunction_{\epsilon_\tau}(\ell_2)$. In other words, we have a quiver isomorphism $\theta_{\epsilon_\tau}:\mu_{\arc^\circ}(Q(\tau^\circ))\rightarrow\mu_{\arc}(Q(\tau))$ acting by $\ell\mapsto\tagfunction_{\epsilon_\tau}(\ell)$ on the vertices and as the identity on the arrows. This quiver isomorphism induces a $K$-algebra isomorphism $\RA{\mu_{\arc^\circ}(Q(\tau^\circ))}\rightarrow\RA{\mu_{\arc}(Q(\tau))}$, which we also denote by $\theta_{\epsilon_\tau}$;
\item $\theta_{\epsilon_\tau}(\mu_{i^\circ}(S(\tau^\circ,\epsilon_\tau\cdot\mathbf{x})))=\mu_i(S(\tau,\mathbf{x}))$.
\end{enumerate}

On the other hand, the quiver $Q(\sigma)$ is obtained from the quiver $Q(\sigma^\circ)$ by replacing each vertex $\ell\in\sigma^\circ$ with $\tagfunction_{\epsilon_\sigma}(\ell)\in\tagfunction_{\epsilon_\sigma}(\sigma^\circ)=\sigma$, and we have a quiver isomorphism $\tagfunction_{\epsilon_\sigma}:Q(\sigma^\circ)\rightarrow Q(\sigma)$ acting by $\ell\mapsto\tagfunction_\sigma(\ell)$ on the vertices and as the identity on the arrows. This quiver isomorphism induces a $K$-algebra isomorphism $\RA{Q(\sigma^\circ)}\rightarrow\RA{Q(\sigma)}$, which we denote also by $\tagfunction_{\epsilon_\sigma}$. By definition of $S(\sigma,\mathbf{x})$ we have $S(\sigma,\mathbf{x})=\tagfunction_{\epsilon_\sigma}(S(\sigma^\circ,\epsilon_\sigma\cdot\mathbf{x}))\in \RA{Q(\sigma)}$.

The $K$-algebra isomorphism $\tagfunction_{\epsilon_\sigma}\varphi\theta_{\epsilon_\tau}^{-1}:\RA{\mu_{\arc}(Q(\tau))}
\rightarrow\RA{Q(\sigma)}$ satisfies
\begin{eqnarray}\nonumber
\tagfunction_{\epsilon_\sigma}\varphi\theta_{\epsilon_\tau}^{-1}(\mu_i(S(\tau,\mathbf{x}))) &=&
\tagfunction_{\epsilon_\sigma}\varphi(\mu_{i^\circ}(S(\tau^\circ,\epsilon_\tau\cdot\mathbf{x})))\\
\nonumber
&\sim_{\operatorname{cyc}}&
\tagfunction_{\epsilon_\sigma}(S(\sigma^\circ,\epsilon_\sigma\cdot\mathbf{x}))\\
\nonumber
&=&
S(\sigma,\mathbf{x}).
\end{eqnarray}
Furthermore, for $\ell\in\sigma^\circ\cap\tau^\circ$ we have
$\tagfunction_{\epsilon_\sigma}\varphi\theta_{\epsilon_\tau}^{-1}(\tagfunction_{\epsilon_\tau}(\ell))=\tagfunction_{\epsilon_\sigma}\varphi(\ell)=
\tagfunction_{\epsilon_\sigma}(\ell)=\tagfunction_{\epsilon_\tau}(\ell)\in\sigma\cap\tau$ (the last equality follows from the fact that $\epsilon_\sigma=\epsilon_\tau$). Therefore, $\tagfunction_{\epsilon_\sigma}\varphi\theta_{\epsilon_\tau}^{-1}$ acts as the identity on the elements of $\sigma\cap\tau$. Since $\sigma\cap\tau$ contains all but one of the vertices of the quiver $\mu_i(\qtau)$ (resp. $\qsigma$), we conclude that $\tagfunction_{\epsilon_\sigma}\varphi\theta_{\epsilon_\tau}^{-1}$ is a right-equivalence $\mu_i\qstau\rightarrow\qssigma$.
\end{mainthmcase}

\begin{mainthmcase} Now suppose that $\epsilon_\tau$ and $\epsilon_\sigma$ differ at exactly one puncture $q$. Since both flips and QP-mutations are involutive (the latter up to right equivalence), we can assume, without loss of generality,
that $\epsilon_\tau(q)=1=-\epsilon_\sigma(q)$. Then $\arc^\circ$ is a folded side of $\tau^\circ$ incident to the puncture $q$, and
$\tagfunction_{\epsilon_\tau\epsilon_\sigma}(\sigma^\circ)=f_{\tagfunction_{\mathbf{1}}(i^\circ)}(\tagfunction_{\mathbf{1}}(\tau^\circ))$. By Theorem \ref{thm:leaving-positive-stratum}, $\mu_{\tagfunction_{\mathbf{1}}(\arc^\circ)}(Q(\tagfunction_{\mathbf{1}}(\tau^\circ)),S(\tagfunction_{\mathbf{1}}(\tau^\circ),\epsilon_\tau\cdot\mathbf{x}))$ is right-equivalent to $(Q(\tagfunction_{\epsilon_\tau\epsilon_\sigma}(\sigma^\circ)),
S(\tagfunction_{\epsilon_\tau\epsilon_\sigma}(\sigma^\circ),\epsilon_\tau\cdot\mathbf{x}))$. This in particular means, since $Q(\tagfunction_{\epsilon_\tau\epsilon_\sigma}(\sigma^\circ))$ is 2-acyclic, that we can write $\mu_{\tagfunction_{\mathbf{1}}(\arc^\circ)}(Q(\tagfunction_{\mathbf{1}}(\tau^\circ)),S(\tagfunction_{\mathbf{1}}(\tau^\circ),\epsilon_\tau\cdot\mathbf{x}))
=(\mu_{\tagfunction_{\mathbf{1}}(\arc^\circ)}(Q(\tagfunction_{\mathbf{1}}(\tau^\circ))),
\mu_{\tagfunction_{\mathbf{1}}(\arc^\circ)}(S(\tagfunction_{\mathbf{1}}(\tau^\circ),\epsilon_\tau\cdot\mathbf{x})))$. Thus, there exists an $R$-algebra isomorphism $\varphi:\RA{\mu_{\tagfunction_{\mathbf{1}}(\arc^\circ)}(Q(\tagfunction_{\mathbf{1}}(\tau^\circ)))}\rightarrow
\RA{Q(\tagfunction_{\epsilon_\tau\epsilon_\sigma}(\sigma^\circ))}$ such that $\varphi(\mu_{\tagfunction_{\mathbf{1}}(i^\circ)}(S(\tagfunction_{\mathbf{1}}(\tau^\circ),\epsilon_\tau\cdot\mathbf{x})))
\sim_{\operatorname{cyc}}S(\tagfunction_{\epsilon_\tau\epsilon_\sigma}(\sigma^\circ),\epsilon_\tau\cdot\mathbf{x})$.

By the very definition of $Q(\tau)$ and $Q(\tagfunction_{\mathbf{1}}(\tau^\circ))$, the quiver $Q(\tau)$ can be obtained from the quiver $Q(\tagfunction_{\mathbf{1}}(\tau^\circ))$ by replacing each vertex $\ell\in\tagfunction_{\mathbf{1}}(\tau^\circ)$ with $\tagfunction_{\epsilon_\tau}(\tagfunction_{\mathbf{1}}^{-1}(\ell))\in\tagfunction_{\epsilon_\tau}(\tau^\circ)=\tau$. Furthermore, the arrow set of $Q(\tau)$ is equal to the arrow set of $Q(\tagfunction_{\mathbf{1}}(\tau^\circ))$, and whenever $a$ is an arrow of $Q(\tagfunction_{\mathbf{1}}(\tau^\circ))$ going from $\ell_1\in\tagfunction_{\mathbf{1}}(\tau^\circ)$ to $\ell_2\in\tagfunction_{\mathbf{1}}(\tau^\circ)$, then $a$ itself is an arrow of $Q(\tau)$ going from $\tagfunction_{\epsilon_\tau}(\tagfunction_{\mathbf{1}}^{-1}(\ell_1))\in\tau$ to $\tagfunction_{\epsilon_\tau}(\tagfunction_{\mathbf{1}}^{-1}(\ell_2))\in\tau$. We thus have a quiver isomorphism $Q(\tagfunction_{\mathbf{1}}(\tau^\circ))\rightarrow Q(\tau)$, to be denoted $\tagfunction_{\epsilon_\tau}\tagfunction_{\mathbf{1}}^{-1}$, acting by $\ell\mapsto\tagfunction_{\epsilon_\tau}(\tagfunction_{\mathbf{1}}^{-1}(\ell))$ on the vertices and as the identity on the arrows. The $K$-algebra isomorphism $\RA{Q(\tagfunction_{\mathbf{1}}(\tau^\circ))}\rightarrow\RA{Q(\tau)}$ it induces will be denoted by $\tagfunction_{\epsilon_\tau}\tagfunction_{\mathbf{1}}^{-1}$ as well (see the paragraph right after Definition \ref{def:obvious-cycles}). By definition of $S(\tau,\mathbf{x})$ and of $S(\tagfunction_{\mathbf{1}}(\tau^\circ),\epsilon_\tau\cdot\mathbf{x})$ we have $S(\tau,\mathbf{x})=\tagfunction_{\epsilon_\tau}(S(\tau^\circ,\epsilon_\tau\cdot\mathbf{x}))\in \RA{Q(\tau)}$ and
$S(\tagfunction_{\mathbf{1}}(\tau^\circ),\epsilon_\tau\cdot\mathbf{x})=
\tagfunction_{\mathbf{1}}(S((\tagfunction_{\mathbf{1}}(\tau^\circ))^\circ,\tagfunction_{\mathbf{1}}\cdot(\epsilon_\tau\cdot\mathbf{x})))=
\tagfunction_{\mathbf{1}}(S(\tau^\circ,\epsilon_\tau\cdot\mathbf{x}))$, hence $S(\tau,\mathbf{x})=\tagfunction_{\epsilon_\tau}\tagfunction_{\mathbf{1}}^{-1}(S(\tagfunction_{\mathbf{1}}(\tau^\circ),\epsilon_\tau\cdot\mathbf{x}))$. Thus, since $\tagfunction_{\epsilon_\tau}\tagfunction_{\mathbf{1}}^{-1}(\tagfunction_{\mathbf{1}}(i^\circ))=i$, if we simultaneously apply the QP-mutation $\mu_{\tagfunction_{\mathbf{1}}(i^\circ)}$ to $(Q(\tagfunction_{1}(\tau^\circ)),S(\tagfunction_{\mathbf{1}}(\tau^\circ),\epsilon_\tau\cdot\mathbf{x}))$ and the QP-mutation $\mu_i$ to $(Q(\tau),S(\tau,\mathbf{x}))$, we see that:
\begin{enumerate}
\item since the underlying quiver of $\mu_{\tagfunction_{\mathbf{1}}(i^\circ)}(Q(\tagfunction_{1}(\tau^\circ)),S(\tagfunction_{\mathbf{1}}(\tau^\circ),\epsilon_\tau\cdot\mathbf{x}))$ is 2-acyclic (by Theorem \ref{thm:leaving-positive-stratum}), the underlying quiver of $\mu_i(Q(\tau),S(\tau,\mathbf{x}))$ is 2-acyclic as well, and hence we can write $\mu_i(Q(\tau),S(\tau,\mathbf{x}))=(\mu_i(Q(\tau)),\mu_i(S(\tau,\mathbf{x})))$;
\item we can obtain $\mu_\arc(Q(\tau))$ from $\mu_{\tagfunction_{\mathbf{1}}(\arc^\circ)}(Q(\tagfunction_{\mathbf{1}}(\tau^\circ)))$ by replacing each vertex $\ell$ of $\mu_{\tagfunction_{\mathbf{1}}(\arc^\circ)}(Q(\tagfunction_{\mathbf{1}}(\tau^\circ)))$ with the vertex $\tagfunction_{\epsilon_\tau}\tagfunction_{\mathbf{1}}^{-1}(\ell)$ of $\mu_\arc(Q(\tau))$. The arrow sets of $\mu_\arc(Q(\tau))$ and $\mu_{\tagfunction_{\mathbf{1}}(\arc^\circ)}(Q(\tagfunction_{\mathbf{1}}(\tau^\circ)))$ then coincide, and whenever $a$ is an arrow of $\mu_{\tagfunction_{\mathbf{1}}(\arc^\circ)}(Q(\tagfunction_{\mathbf{1}}(\tau^\circ)))$ going from $\ell_1$ to $\ell_2$, then $a$ itself is an arrow of $\mu_\arc(Q(\tau))$ going from $\tagfunction_{\epsilon_\tau}\tagfunction_{\mathbf{1}}^{-1}(\ell_1)$ to $\tagfunction_{\epsilon_\tau}\tagfunction_{\mathbf{1}}^{-1}(\ell_2)$. In other words, we have a quiver isomorphism $\theta_{\epsilon_\tau}:\mu_{\tagfunction_{\mathbf{1}}(\arc^\circ)}(Q(\tagfunction_{\mathbf{1}}(\tau^\circ)))\rightarrow\mu_{\arc}(Q(\tau))$ acting by $\ell\mapsto\tagfunction_{\epsilon_\tau}\tagfunction_{\mathbf{1}}^{-1}(\ell)$ on the vertices and as the identity on the arrows. This quiver isomorphism induces a $K$-algebra isomorphism $\RA{\mu_{\tagfunction_{\mathbf{1}}(\arc^\circ)}(Q(\tagfunction_{\mathbf{1}}(\tau^\circ)))}\rightarrow\RA{\mu_{\arc}(Q(\tau))}$, which we also denote by $\theta_{\epsilon_\tau}$;
\item $\theta_{\epsilon_\tau}(\mu_{\tagfunction_{\mathbf{1}}(i^\circ)}(S(\tagfunction_{\mathbf{1}}(\tau^\circ),\epsilon_\tau\cdot\mathbf{x})))=
    \mu_i(S(\tau,\mathbf{x}))$.
\end{enumerate}

On the other hand, by the definition of the quivers $Q(\sigma)$ and $Q(\tagfunction_{\epsilon_\tau\epsilon_\sigma}(\sigma^\circ))$, it is possible to obtain $Q(\sigma)$ from $Q(\tagfunction_{\epsilon_\tau\epsilon_\sigma}(\sigma^\circ))$ by replacing each vertex $\ell\in\tagfunction_{\epsilon_\tau\epsilon_\sigma}(\sigma^\circ)$ with $\tagfunction_{\epsilon_\sigma}(\tagfunction_{\epsilon_\tau\epsilon_\sigma}^{-1}(\ell))\in\tagfunction_{\epsilon_\sigma}(\sigma^\circ)=\sigma$, and we have a quiver isomorphism $\tagfunction_{\epsilon_\sigma}\tagfunction_{\epsilon_\tau\epsilon_\sigma}^{-1}:Q(\tagfunction_{\epsilon_\tau\epsilon_\sigma}(\sigma^\circ))\rightarrow Q(\sigma)$ acting by $\ell\mapsto\tagfunction_\sigma(\tagfunction_{\epsilon_\tau\epsilon_\sigma}^{-1}(\ell))$ on the vertices and as the identity on the arrows. This quiver isomorphism induces a $K$-algebra isomorphism $\RA{Q(\tagfunction_{\epsilon_\tau\epsilon_\sigma}(\sigma^\circ))}\rightarrow\RA{Q(\sigma)}$, which we denote also by $\tagfunction_\sigma\tagfunction_{\epsilon_\tau\epsilon_\sigma}^{-1}$. By definition of $S(\sigma,\mathbf{x})$ and of $S(\tagfunction_{\epsilon_\tau\epsilon_\sigma}(\sigma^\circ),\epsilon_\tau\cdot\mathbf{x})$ we have $S(\sigma,\mathbf{x})=\tagfunction_{\epsilon_\sigma}(S(\sigma^\circ,\epsilon_\sigma\cdot\mathbf{x}))\in \RA{Q(\sigma)}$ and
$S(\tagfunction_{\epsilon_\tau\epsilon_\sigma}(\sigma^\circ),\epsilon_\tau\cdot\mathbf{x})= \tagfunction_{\epsilon_\tau\epsilon_\sigma}(S((\tagfunction_{\epsilon_\tau\epsilon_\sigma}(\sigma^\circ))^\circ,
(\epsilon_\tau\epsilon_\sigma)\cdot(\epsilon_\tau\cdot\mathbf{x})))=
\tagfunction_{\epsilon_\tau\epsilon_\sigma}(S(\sigma^\circ,\epsilon_\sigma\cdot\mathbf{x})))
\in\RA{Q(\tagfunction_{\epsilon_\tau\epsilon_\sigma}(\sigma^\circ))}$, hence
$S(\sigma,\mathbf{x})=
\tagfunction_{\epsilon_\sigma}\tagfunction_{\epsilon_\tau\epsilon_\sigma}^{-1}(S(\tagfunction_{\epsilon_\tau\epsilon_\sigma}(\sigma^\circ),\epsilon_\tau\cdot\mathbf{x}))$.

The $K$-algebra isomorphism $\tagfunction_{\epsilon_\sigma}\tagfunction_{\epsilon_\tau\epsilon_\sigma}^{-1}\varphi\theta_{\epsilon_\tau}^{-1}:
\RA{\mu_{\arc}(Q(\tau))}
\rightarrow
\RA{Q(\sigma)}$ satisfies
\begin{eqnarray}\nonumber
\tagfunction_{\epsilon_\sigma}\tagfunction_{\epsilon_\tau\epsilon_\sigma}^{-1}\varphi\theta_{\epsilon_\tau}^{-1}(\mu_i(S(\tau,\mathbf{x}))) &=&
\tagfunction_{\epsilon_\sigma}\tagfunction_{\epsilon_\tau\epsilon_\sigma}^{-1}\varphi(\mu_{\tagfunction_{\mathbf{1}}(i^\circ)}(S(\tagfunction_{\mathbf{1}}(\tau^\circ),\epsilon_\tau\cdot\mathbf{x})))\\
\nonumber
&\sim_{\operatorname{cyc}}&
\tagfunction_{\epsilon_\sigma}\tagfunction_{\epsilon_\tau\epsilon_\sigma}^{-1}(S(\tagfunction_{\epsilon_\tau\epsilon_\sigma}(\sigma^\circ),\epsilon_\tau\cdot\mathbf{x}))\\
\nonumber
&=&
S(\sigma,\mathbf{x}).
\end{eqnarray}
Furthermore, for $\ell\in\sigma\cap\tau$ we have $\tagfunction_{\mathbf{1}}\tagfunction_{\epsilon_\tau}^{-1}(\ell)\in \tagfunction_{\epsilon_\tau\epsilon_\sigma}(\sigma^\circ)$, and  $\tagfunction_{\epsilon_\tau\epsilon_\sigma}^{-1}\tagfunction_{\mathbf{1}}\tagfunction_{\epsilon_\tau}^{-1}(\ell)=\tagfunction_{\epsilon_\sigma}^{-1}(\ell)$, and hence
$\tagfunction_{\epsilon_\sigma}\tagfunction_{\epsilon_\tau\epsilon_\sigma}^{-1}\varphi\theta_{\epsilon_\tau}^{-1}(\ell)=
\tagfunction_{\epsilon_\sigma}\tagfunction_{\epsilon_\tau\epsilon_\sigma}^{-1}\varphi(\tagfunction_{\mathbf{1}}\tagfunction_{\epsilon_\tau}^{-1}(\ell))=
\tagfunction_{\epsilon_\sigma}\tagfunction_{\epsilon_\tau\epsilon_\sigma}^{-1}\tagfunction_{\mathbf{1}}\tagfunction_{\epsilon_\tau}^{-1}(\ell)=
\tagfunction_{\epsilon_\sigma}\tagfunction_{\epsilon_\sigma}^{-1}(\ell)=\ell$. Therefore, $\tagfunction_{\epsilon_\sigma}\tagfunction_{\epsilon_\tau\epsilon_\sigma}^{-1}\varphi\theta_{\epsilon_\tau}^{-1}$ acts as the identity on the elements of $\sigma\cap\tau$. Since $\sigma\cap\tau$ contains all but one of the vertices of the quiver $\mu_i(\qtau)$ (resp. $\qsigma$), we conclude that $\tagfunction_{\epsilon_\sigma}\tagfunction_{\epsilon_\tau\epsilon_\sigma}^{-1}\varphi\theta_{\epsilon_\tau}^{-1}$ is a right-equivalence $\mu_i\qstau\rightarrow\qssigma$.
\end{mainthmcase}

Theorem \ref{thm:tagged-flips<->mutations} is proved.
\end{proof}

\section{Non-degeneracy of the QPs $\qstau$}\label{sec:nondegeneracy}

The following is an immediate consequence of Theorem \ref{thm:tagged-flips<->mutations}.

\begin{coro}\label{coro:non-degenerate} The QPs associated to tagged triangulations of $\surf$ are non-degenerate provided $\surf$ is not a sphere with less than 6 punctures, that is, provided $\surf$ is one of the following:
\begin{itemize}\item a surface with non-empty boundary, with or without punctures, and arbitrary genus;
\item a positive-genus surface without boundary, and any number of punctures;
\item a sphere with at least 6 punctures.
\end{itemize}
Thus, Conjecture 33 of \cite{Labardini1} holds for all these surfaces.
\end{coro}

\begin{remark}\begin{enumerate}\item In \cite[Theorem 31]{Labardini1} it was shown that the QPs associated to ideal triangulations of (possibly punctured) surfaces with non-empty boundary are rigid, and this automatically implied their non-degeneracy by \cite[Corollary 8.2]{DWZ1}. Here we have presented a different proof of the non-degeneracy of these QPs: no allusion to rigidity has been made, and the potentials corresponding to tagged triangulations, missing in \cite{Labardini1}, have been calculated.
\item Since positive-genus closed surfaces with exactly one puncture do not posses ideal triangulations with self-folded triangles whatsoever, Theorem \ref{thm:ideal-flips<->mutations} (which is Theorem 30 of \cite{Labardini1}) already implies the non-degeneracy of the QPs associated to triangulations of these surfaces, so the Popping Theorem is completely superfluous in this case.
\item Before the present paper was posted, the only empty-boundary surfaces for which the non-degeneracy of the QPs $\qstau$ had been shown, are the positive-genus closed surfaces with exactly one puncture. Corollary \ref{coro:non-degenerate} is therefore a strong improvement regarding the question of non-degeneracy, solving it in practically all cases.
\item It has been shown in \cite{GLFS} that if $\tau$ and $\sigma$ are tagged triangulations of a 5-punctured sphere, and $\sigma=f_{i}(\tau)$ for a tagged arc $i\in\tau$, then $\mu_i\qstau$ is right-equivalent to $(\qsigma,\lambda\ssigma)$ for some non-zero scalar $\lambda\in K$. This implies the non-degeneracy of the QPs $\qstau$ associated to the tagged triangulations of the 5-punctured sphere; that is, Corollary \ref{coro:non-degenerate} holds for the 5-punctured sphere too.
\end{enumerate}
\end{remark}

\section{Irrelevance of signs and scalars attached to the punctures}\label{sec:irrelevant}

Let $\surf$ be any surface. A standard construction associates a \emph{dimer model} $D(\sigma)$ to each ideal triangulation $\sigma$ of $\surf$ without self-folded triangles. For the reader's convenience we give a quick description of this construction.

Pick an arbitrary point $q_{\triangle}$ inside each ideal triangle $\triangle$ of $\sigma$, in such a way that $q_\triangle$ does not lie on an arc of $\sigma$. Define a bipartite graph $D(\sigma)$ as follows:
\begin{itemize}
\item The vertex set of $D(\sigma)$ is $\marked\cup\{q_\triangle\suchthat\triangle$ is an ideal triangle of $\sigma\}$;
\item given $p\in\marked$ and $q_\triangle\in\triangle$ with $\triangle$ ideal triangle of $\sigma$, connect $p$ and $q_\triangle$ with an edge if and only if $p\in\triangle$;
\item besides the edges just introduced, do not include more edges between vertices of $D(\sigma)$.
\end{itemize}

Since $\sigma$ does not have self-folded triangles, we can suppose that the edges of $D(\sigma)$ have been drawn on the surface, and that whenever $p\in\marked$ and $q_\triangle\in\triangle$ are connected by an edge of $D(\sigma)$, such edge is entirely contained in $\triangle$ and the only point in common it has with any arc of $\tau$ is $p$.

The straightforward proof of the following lemma is left to the reader.

\begin{lemma}\label{lemma:nice-dimer} Suppose $\partial\Sigma\neq\varnothing$. Then $\surf$ has an ideal triangulation $\tau$ without self-folded triangles such that
every puncture of $\surf$ can be connected to a marked point lying on $\partial\Sigma$ through a simple walk on $D(\tau)$ (that is, a walk that does not repeat vertices of $D(\tau)$).
\end{lemma}

The absence of self-folded triangles in $\tau$ allows us to draw the unreduced signed-adjacency quiver $\unredqtau$ on $\surf$ as well. We notice then that $\unredqtau$ is a subquiver of the \emph{dimer dual} of $D(\tau)$.

We are now ready to show that the signs and scalars appearing in $\stau$ are irrelevant whenever $\tau$ is a tagged triangulation of a surface with non-empty boundary. More precisely:

\begin{prop}\label{prop:scalars-are-irrelevant} Let $\surf$ be a surface with non-empty boundary and let $\mathbf{x}=(x_p)_{p\in\punct}$ and $\mathbf{y}=(y_p)_{p\in\punct}$ be any two choices of non-zero scalars. Then for every tagged triangulation $\tau$ of $\surf$, the QPs $(Q(\tau),S(\tau,\mathbf{x}))$ and $(Q(\tau),S(\tau,\mathbf{y}))$ are right-equivalent.
\end{prop}

\begin{proof} By Theorem \ref{thm:tagged-flips<->mutations}, Lemma \ref{lemma:nice-dimer} and the fourth assertion of Proposition \ref{prop:ideal-triangs-seqs-of-flips}, we can assume, without loss of generality, that $\tau$ is as in the conclusion of Lemma \ref{lemma:nice-dimer}.

Let $p_1,\ldots,p_{|\punct|}$, be any ordering of the punctures of $\surf$. For $\ell=0,1,\ldots,|\punct|$, let
$\mathbf{y}_{\ell}=(y_{\ell,p})_{p\in\punct}$ be the choice of scalars defined as follows:
$$
y_{\ell,p}=\begin{cases}
1 & \text{if $p\in\{p_1,\ldots,p_\ell\}$},\\
x_p & \text{otherwise}.
\end{cases}
$$
Since compositions of right-equivalences is right-equivalence, to prove the proposition it is sufficient to show that
$(Q(\tau),S(\tau,\mathbf{y}_\ell))$ is right-equivalent to $(Q(\tau),S(\tau,\mathbf{y}_{\ell+1}))$ for $\ell=0,\ldots,|\punct|-1$ (notice that $(Q(\tau),S(\tau,\mathbf{y}_0))=\qstau$), and in order to do so, it is sufficient to show that
$(\widehat{Q}(\tau),\widehat{S}(\tau,\mathbf{y}_\ell))$ is right-equivalent to $(\widehat{Q}(\tau),\widehat{S}(\tau,\mathbf{y}_{\ell+1}))$ (this is because reduced parts of QPs are uniquely determined up to right-equivalence).

Let $w=(q_1,e_1,q_2,e_2,\ldots,q_{t-1},e_{t-1},q_t)$ be a simple walk on $D(\tau)$ such that
\begin{itemize}
\item $q_1=p_{n+1}$,
\item $q_t$ lies on $\partial\Sigma$;
\item none of the vertices $q_1,\ldots,q_{t-1}$ of $D(\tau)$ lies on $\partial\Sigma$.
\end{itemize}
Each edge $e_s$, $s=1,\ldots,t-2$, uniquely determines an arrow $\alpha_s$ of $\unredqtau$, whereas $e_{t-1}$ may or may not do so. In any case, all arrows determined by these edges are pairwise distinct. This implies that the $R$-algebra automorphism $\varphi$ of $\RA{\unredqtau}$ defined by
$$
\varphi(\alpha_s)=\begin{cases}
x_{p_{n+1}}^{-1}\alpha_s & \text{if $s$ is odd},\\
x_{p_{n+1}}\alpha_s & \text{if $s$ is even},
\end{cases}
$$
is a right-equivalence $\varphi:(\widehat{Q}(\tau),\widehat{S}(\tau,\mathbf{y}_{\ell}))\rightarrow(\widehat{Q}(\tau),\widehat{S}(\tau,\mathbf{y}_{\ell+1}))$.
\end{proof}

\begin{ex}\label{ex:walk-on-dimer} In Figure \ref{Fig:dimer_3punc_hexagon} we can see an ideal triangulation $\tau$, with the bipartite graph $D(\tau)$ drawn on the surface as well. The walk on $D(\tau)$ that consists of the edges of $D(\tau)$ that have been drawn bolder, helps to get rid of the scalar attached to the puncture in the center of the surface. $\hfill{\blacktriangle}$
        \begin{figure}[!h]
                \caption{Getting rid of scalars attached to the punctures}\label{Fig:dimer_3punc_hexagon}
                \centering
                \includegraphics[scale=.7]{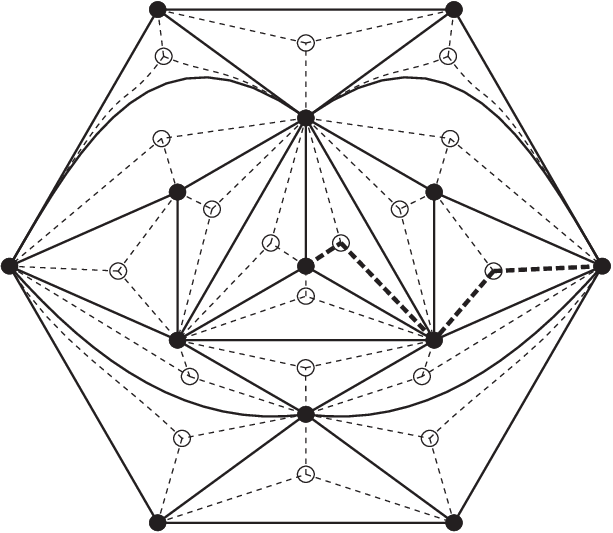}
        \end{figure}

\end{ex}

Proposition \ref{prop:scalars-are-irrelevant} implies that the QP defined in \cite{CI-LF} for an arbitrary tagged triangulation of a surface with non-empty boundary is right-equivalent to the QP $\qstau$ given in Definition \ref{def:QP-of-tagged-triangulation} above. Indeed, the one defined in \cite{CI-LF} and the one defined here differ only by the appearance of negative signs at some cycles corresponding to punctures.

For surfaces with empty boundary, we have the following :

\begin{prop}[{\cite{GLFS}}]\label{prop:scalars-are-irrelevant-empty-bound} Let $\surf$ be a surface with empty boundary and let $\mathbf{x}=(x_p)_{p\in\punct}$ and $\mathbf{y}=(y_p)_{p\in\punct}$ be any two choices of non-zero scalars. If the ground field $K$ is algebraically closed, then for every tagged triangulation $\tau$ of $\surf$, the QPs $(Q(\tau),S(\tau,\mathbf{x}))$ and $(Q(\tau),\lambda S(\tau,\mathbf{y}))$ are right-equivalent for some non-zero scalar $\lambda$.
\end{prop}

\begin{remark} The original motivation for working with arbitrary choices $(x_p)_{p\in\punct}$ of non-zero scalars, rather than with the particular choices $\mathbf{1}=(1)_{p\in\punct}$ and $-\mathbf{1}=(-1)_{p\in\punct}$, came from the wish of being able to obtain as many non-degenerate potentials as possible.
\end{remark}

\section{Jacobian algebras are finite-dimensional}\label{sec:Jacobi-finiteness}

It is natural to ask whether the Jacobian algebras of the QPs arising from (tagged) triangulations are finite-dimensional or not.
For polygons with at most one puncture, the answer was known to Caldero, Chapoton and Schiffler (cf. \cite{CCS1} and \cite{CCS2})
and Schiffler (cf. \cite{S}), although these authors did not use the language of potentials and Jacobian algebras in the referred
papers.
For more general surfaces, the first partial answers to the Jacobi-finiteness question were given in \cite{ABCP}, \cite{Labardini1} and \cite{Labardini2}.

Assem, Br\"{u}stle, Charbonneau-Jodoin and Plamondon proved:

\begin{thm}[{\cite{ABCP}}]\label{thm:unpunct-fin-dim} If $\surf$ is an unpunctured surface, then all QPs of ideal triangulations of $\surf$ have finite-dimensional Jacobian algebras.
\end{thm}

The following more general result was simultaneously found by the author.

\begin{thm}[{\cite[Theorem 36]{Labardini1}}]\label{thm:non-empty-bound-fin-dim} If the surface $\surf$ has non-empty boundary, then for every ideal triangulation $\tau$ of $\surf$ the Jacobian algebra $\mathcal{P}\qstau$ is finite-dimensional.
\end{thm}

\begin{ex}[{\cite[Example 8.2]{Labardini2}}] Let $\tau$ be any ideal triangulation of the once-punctured torus (with empty boundary), then the Jacobian algebra $\mathcal{P}\qstau$ is finite-dimensional.
\end{ex}

Trepode and Valdivieso-D\'{i}az have recently shown:

\begin{thm}[\cite{TV}]\label{thm:TV} Let $\surf$ be a sphere with $n\geq 5$ punctures. Then for any ideal triangulation $\tau$ of $\surf$, the Jacobian algebra $\mathcal{P}\qstau$ is finite-dimensional.
\end{thm}

Ladkani proved the following generalization.

\begin{thm}[{\cite{Ladkani}}]\label{thm:Ladkani} The Jacobian algebra $\mathcal{P}\qstau$ is finite-dimensional for every ideal triangulation $\tau$ of a surface with empty boundary.
\end{thm}

Since finite-dimensionality of Jacobian algebras is preserved by QP-mutations (cf. \cite[Corollary 6.6]{DWZ1}), a straightforward combination of Theorem \ref{thm:tagged-flips<->mutations} with Theorems \ref{thm:unpunct-fin-dim}, \ref{thm:non-empty-bound-fin-dim}, \ref{thm:TV} and \ref{thm:Ladkani},  yields:

\begin{coro} All the QPs associated to tagged triangulations of the surfaces considered in the present note have finite-dimensional Jacobian algebras.
\end{coro}

\begin{remark} This last corollary does not imply that a Jacobian algebra of the form $\mathcal{P}(\qtau,S(\tau,\mathbf{x}))$ is necessarily isomorphic to the quotient $R\langle\qtau\rangle/J_0(S(\tau,\mathbf{x}))$ (by definition, $J_0(S(\tau,\mathbf{x}))$ is the ideal of $R\langle\qtau\rangle$ generated by the cyclic derivatives of $S(\tau,\mathbf{x})$), nor that $J_0(S(\tau,\mathbf{x}))$ is an admissible ideal of $R\langle\qtau\rangle$ (cf. \cite[Remark 5.2]{CI-LF}). Indeed, if the underlying surface has empty boundary, then $R\langle\qtau\rangle/J_0(S(\tau,\mathbf{x}))$ is infinite-dimensional, while $\mathcal{P}(\qtau,S(\tau,\mathbf{x}))$ is not.  On the other hand, in the non-empty boundary situation, $J_0(S(\tau,\mathbf{x}))$ is always an admissible ideal of $R\langle\qtau\rangle$, and $\mathcal{P}(\qtau,S(\tau,\mathbf{x}))$ is always isomorphic to $R\langle\qtau\rangle/J_0(S(\tau,\mathbf{x}))$ (cf. \cite{ABCP} and \cite[Theorem 5.5]{CI-LF}).
\end{remark}

\end{document}